\documentclass[11pt,leqno]{amsart}

\usepackage{amsmath,amsthm,amscd,amssymb,amsfonts, amsbsy}
\usepackage{latexsym}
\usepackage{txfonts}
\usepackage{exscale}
\usepackage{mathrsfs}
\usepackage{esint} 
\usepackage{enumitem} 
\usepackage{color}
 
\usepackage[
  hmarginratio={1:1},     
  vmarginratio={1:1},     
  textwidth=400pt,			  
	textheight=580pt,        
  heightrounded,          
]{geometry}

\usepackage[utf8]{inputenc}

\usepackage[colorlinks,citecolor=blue,pagebackref,hypertexnames=false,pagebackref]{hyperref}

\usepackage{appendix}

\makeatletter
\renewcommand\part{%
   \if@noskipsec \leavevmode \fi
   \par
   \addvspace{4ex}%
   \@afterindentfalse
   \secdef\@part\@spart}
   
\renewcommand \thepart {\Roman{part}}

\def\@part[#1]#2{%
    \ifnum \c@secnumdepth >\m@ne
      \refstepcounter{part}%
      \addcontentsline{toc}{part}{\partname \nobreakspace \thepart \hspace{1em}#1}%
    \else
      \addcontentsline{toc}{part}{#1}%
    \fi
    {\parindent \z@ \raggedright
     \interlinepenalty \@M
     \normalfont
     \ifnum \c@secnumdepth >\m@ne
       \centering \Large \bfseries \partname \nobreakspace \thepart. \nobreakspace
     \fi
     \Large \bfseries #2%
     \par}%
    \nobreak
    \vskip 3ex
    \@afterheading}
\def\@spart#1{%
    {\parindent \z@ \raggedright
     \interlinepenalty \@M
     \normalfont
     \huge \bfseries #1\par}%
     \nobreak
     \vskip 3ex
     \@afterheading}
\makeatother

\allowdisplaybreaks

\numberwithin{equation}{section}

\theoremstyle{plain}
\newtheorem{theorem}[equation]{Theorem}
\newtheorem{prop}[equation]{Proposition}
\newtheorem{corollary}[equation]{Corollary}
\newtheorem{lemma}[equation]{Lemma}

\theoremstyle{definition}
\newtheorem{defn}[equation]{Definition}

\theoremstyle{remark}
\newtheorem{remarks}[equation]{Remarks}
\newtheorem{remark}[equation]{Remark}

\numberwithin{equation}{section}

\newcommand{\RR}{{\mathbb{R}}}
\newcommand{\NN}{{\mathbb{N}}}
\newcommand{\ZZ}{{\mathbb{Z}}}

\newcommand{\WW}{\mathcal{W}}

\newcommand{\C}{\mathcal{C}}
\newcommand{\DD}{\mathbb{D}}
\newcommand{\dd}{\mathbb{D}}
\newcommand{\PP}{\mathbb{P}}
\newcommand{\HH}{\mathcal{H}}
\newcommand{\F}{\mathcal{F}}
\newcommand{\N}{\mathcal{N}}

\newcommand{\W}{\mathcal{W}}

\newcommand{\R}{\mathcal{R}}

\newcommand{\dist}{\operatorname{dist}}

\newcommand{\pom}{\partial\Omega}

\newcommand{\hm}{\omega}

\renewcommand{\emptyset}{\mbox{\textup{\O}}}

\DeclareMathOperator*{\osc}{osc}

\DeclareMathOperator{\diam}{diam}
\DeclareMathOperator{\interior}{int}

\DeclareMathOperator*{\Lip}{Lip}

\def\XXint#1#2#3{{\setbox0=\hbox{$#1{#2#3}{\int}$}
     \vcenter{\hbox{$#2#3$}}\kern-.5\wd0}}

\DeclareMathOperator{\divg}{div}
\DeclareMathOperator{\spt}{spt}

\DeclareMathOperator{\Id}{Id}
\DeclareMathOperator{\capacity}{cap_2}

\newcommand{\pO}{\partial\Omega}

\newcommand{\Oj}{\Omega_j}
\newcommand{\Oinf}{\Omega_\infty}
\newcommand{\pOj}{\partial\Omega_j}
\newcommand{\pOinf}{\partial\Omega_\infty}

\newcommand{\sj}{\sigma_j}
\newcommand{\sinf}{\mu_\infty}
\newcommand{\oj}{\omega_j}
\newcommand{\oinf}{\omega_{L_\infty}}
\newcommand{\uj}{u_j}

\newcommand{\wXj}{X_j}
\newcommand{\wLj}{L_j}
\newcommand{\wAj}{A_j}
\newcommand{\wcalA}{\mathcal{A}}

\newcommand{\wdj}{\delta_j}

\newcommand{\mj}{\mu_j}
\newcommand{\minf}{\mu_\infty}

\newcommand{\ojA}{\oj^{A(p,r)}}

\reversemarginpar

\begin{document}
\allowdisplaybreaks

\title[Uniform rectifiability and elliptic operators]{Uniform rectifiability and elliptic operators satisfying a Carleson measure 
	condition}

\author[S. Hofmann]{Steve Hofmann}

\address{Steve Hofmann
\\
Department of Mathematics
\\
University of Missouri
\\
Columbia, MO 65211, USA} \email{hofmanns@missouri.edu}

\author[J.M. Martell]{José María Martell}

\address{José María Martell
\\
Instituto de Ciencias Matemáticas CSIC-UAM-UC3M-UCM
\\
Consejo Superior de Investigaciones Científicas
\\
C/ Nicolás Cabrera, 13-15
\\
E-28049 Madrid, Spain} \email{chema.martell@icmat.es}

\author[S. Mayboroda]{Svitlana Mayboroda}

\address{Svitlana Mayboroda
\\
Department of Mathematics
\\
University of Minnesota
\\
Minneapolis, MN 55455, USA} \email{svitlana@math.umn.edu}

\author[T. Toro]{Tatiana Toro}

\address{Tatiana Toro 
\\ 
University of Washington 
\\
Department of Mathematics 
\\
Seattle, WA 98195-4350, USA}

\email{toro@uw.edu}

\author[Z. Zhao]{Zihui Zhao}

\address{Zihui Zhao
\\ 
Department of Mathematics
\\
University of Chicago
\\
Chicago, IL 60637, USA}

\email{zhaozh@uchicago.edu}

\thanks{The first author was partially supported by NSF grant number DMS-1664047.
The second author acknowledges that
the research leading to these results has received funding from the European Research
Council under the European Union's Seventh Framework Programme (FP7/2007-2013)/ ERC
agreement no. 615112 HAPDEGMT. He also acknowledges financial support from the Spanish Ministry of Economy and Competitiveness, through the ``Severo Ochoa Programme for Centres of Excellence in R\&D'' (SEV-2015-0554). 
The third author was partially supported by the NSF INSPIRE Award DMS 1344235, the NSF RAISE-TAQ grant DMS 1839077, and the Simons 
Foundation grant 563916, SM.
The fourth author was partially supported by the Craig McKibben \& Sarah Merner Professor in Mathematics, by NSF grant number DMS-1664867, and by the Simons Foundation Fellowship 614610. 
The fifth author was partially supported by NSF grants DMS-1361823, DMS-1500098, DMS-1664867, DMS-1902756 and by the Institute for Advanced Study. }
\thanks{This material is based upon work supported by the National Science Foundation under Grant No. DMS-1440140 while the authors were in residence at the Mathematical Sciences Research Institute in Berkeley, California, during the Spring 2017 semester.}

\date{\today}


\begin{abstract}
The present paper establishes the correspondence between the properties of the solutions of a class of PDEs and the geometry of sets in Euclidean space. We settle the question of whether (quantitative) absolute continuity of the elliptic measure with respect to the surface measure and uniform rectifiability of the boundary are equivalent, in an optimal class of divergence form elliptic operators satisfying a suitable Carleson measure condition. The result can be viewed as a quantitative analogue of the Wiener criterion adapted to the singular $L^p$ data case.

The first step is taken in Part \ref{part:small}, where we considered the case in which the desired Carleson measure condition on the coefficients holds with \textit{sufficiently small constant}, using a novel application of techniques developed in geometric measure theory.
In Part \ref{part:large} we establish the final result, that is, 
the ``large constant case''. The key elements are a powerful extrapolation argument, which provides a general pathway to self-improve scale-invariant small constant estimates, and a new mechanism to transfer quantitative absolute continuity of elliptic measure between a domain and its subdomains.
\end{abstract}

\maketitle

\tableofcontents

\section{Introduction}
\subsection{Background and Previous Results}

The present paper, together with its converse in \cite{KP} (see also \cite{DJe}), culminate many years of activity at the intersection of harmonic analysis, geometric measure theory, and PDEs, devoted to the complete understanding of necessary and sufficient conditions on the operator and the geometry of the domain guaranteeing absolute continuity of the elliptic measure with respect to the surface measure of the boundary.

The celebrated 1924 Wiener criterion \cite{Wiener} provided the necessary and sufficient conditions on the geometry of the domain responsible for the continuity of the harmonic functions at the boundary. In the probabilistic terms, it characterized the points of the boundary which are  ``seen" by the Brownian travelers coming from the interior of the domain. 

The question of finding necessary and sufficient geometric conditions which could guarantee adequate regularity, so that, roughly speaking, the pieces of the boundary are seen by the Brownian travelers according to their surface measure, turned out to be much more intricate. Curiously, already in 1916 F. \& M. Riesz correctly identified the key geometric notion in this context:  rectifiability of the boundary $\pom$, i.e., the existence of tangent planes almost everywhere 
with respect to arc length $\sigma$ on $\pom$.  In particular, they showed in  \cite{Rfm} that harmonic measure is (mutually) absolutely continuous with respect to $\sigma$ for
a simply connected domain in the plane with rectifiable boundary. It took more than a hundred years to establish the converse of the F. \& M. Riesz theorem and its higher dimensional analogues. The first such result appeared in 2016 \cite{7au}, and the question was fully settled for the harmonic functions in 2018 \cite{AHMMT}.

The question of what happens in the general PDE setting has been puzzling from the beginning.
The Wiener criterion is universal: it applies to all uniformly elliptic divergence form operators with bounded coefficients and characterizes
points of continuity of the solution at the boundary. It was realized early 
on that no such general criterion exists for determining the absolute continuity of elliptic measure with respect to the surface measure to the boundary of a domain.
Some of the challenges that arise when considering this question were highlighted by the counterexamples in \cite{CFK}, \cite{MM}. In 1984 Dahlberg 
formulated a conjecture concerning optimal conditions on a matrix of coefficients which guarantee absolute continuity of elliptic measure with respect to 
Lebesgue measure in a half-space. This question was a driving force of a thread of outstanding developments in harmonic analysis in the 80s and 90s  due to Dahlberg, Jerison, Kenig, Pipher, and others, stimulating some beautiful and far-reaching new techniques in the theory of weights and singular integral operators, to mention only a few approaches. In \cite{KP}, Kenig and Pipher proved Dahlberg's conjecture, they showed that whenever the gradient of coefficients satisfies a Carleson measure condition (to be defined below in \eqref{KP-cond}) the elliptic measure and the Lebesgue measure are mutually absolutely  continuous on a half-space and, by a change of variables argument, above a  Lipschitz graph. 

 

Given the aforementioned developments, it was natural to conjecture that the equivalence of rectifiability and regularity of elliptic measure should be valid in the full generality of Dahlberg-Kenig-Pipher (DKP) coefficients.
Despite numerous attempts this question turned out to be notoriously resistant to existing methods.
The passage from the regularity of the solutions to partial differential equations to rectifiability, or to any geometric information on the boundary, is generally referred to as free boundary problems. This in itself is, of course, a well-studied and rich subject. Unfortunately, the typical techniques arising from minimization of the functionals are both too qualitative and too rigid to treat structural irregularities of  rectifiable sets and such weak assumptions as absolute continuity of harmonic measure. The latter became accessible only recently, with the development of the analysis of singular integrals and similar objects on uniformly rectifiable sets. In particular, the first converse of the F. \& M. Riesz theorem, \cite{7au}, directly relies on the 2012 solution of the David-Semmes conjecture regarding the boundedness of the Riesz transforms in $L^2$ \cite{NTV}. At the same time, the techniques stemming from such results for the harmonic functions are not amenable to more general operators of the DKP type, again, due to simple yet fundamental algebraic deficiencies: the derivatives of the coefficients do not offer sufficient cancellations. 

The main goal of the present paper is to addresses the conjecture in full generality. We establish
 the {\it equivalence} of the absolute continuity of the elliptic measure with respect to the surface measure and the uniform rectifiability of the boundary of a domain under the {DKP}  condition on the coefficients, thus providing the final, optimal geometric results (given the assumed background hypotheses).



We now describe our goal and relevant previous results more precisely. Throughout the paper we shall work under the assumptions that the domain $\Omega$ is uniform, i.e., open and connected in a quantitative way, and that its boundary is $(n-1)$-Ahlfors regular, that is, $(n-1)$-dimensional in a quantitative way (see Section \ref{sec:def}). Under these conditions one can, for instance,  show that scale-invariant absolute continuity of harmonic  measure is related to the uniform  rectifiability of the boundary and even to the non-tangential accessibility of the exterior domain:

\begin{theorem}\label{thm:hmu}
Let $\Omega\subset\RR^n$, $n\ge 3$, be a uniform domain (bounded or unbounded) with 
Ahlfors regular boundary \textup{(}see Definitions \ref{def:uniform} and \ref{def:ADR}\textup{)}, set $\sigma=\mathcal{H}^{n-1}|_{\pO}$ 
and let $\omega_{-\Delta}$ denote its associated harmonic measure. The following statements are equivalent:
\begin{enumerate}[label=\textup{(\alph*)}, itemsep=0.2cm] 

\item\label{1-thm:hmu} $\omega_{-\Delta} \in A_\infty(\sigma)$ \textup{(}Definition \ref{def:AinftyHMU}\textup{)}.

\item\label{2-thm:hmu} $\partial\Omega$ is uniformly rectifiable \textup{(}Definition \ref{def:UR}\textup{)}.

\item\label{3-thm:hmu} $\Omega$ satisfies the exterior corkscrew condition \textup{(}see Definition \ref{def:ICC}), hence, in particular, it is a chord-arc domain \textup{(}Definition \ref{def:nta}).
\end{enumerate}
 \end{theorem}

 Postponing all the rigorous definitions to Section~\ref{sec:def}, we remark for the moment that uniform rectifiability is a quantitative version of the notion of rectifiability of the boundary and the Muckenhoupt condition $\omega\in A_\infty(\sigma)$ is, respectively, a quantitative form of the mutual absolute continuity of $\omega$ with respect to $\sigma$. Thus, Theorem \ref{thm:hmu} above is a quantitative form of the rigorous connection between the boundary behavior of harmonic functions and geometric properties of sets that we alluded to above. 
Returning to the ties with Wiener criterion, we point out that the property of the scale invariant absolute continuity of harmonic measure with respect to surface measure, 
at least in the presence of Ahlfors regularity of $\pom$,
is equivalent to the solvability of the Dirichlet problem
with data in some $L^p(\pom)$, with $p<\infty$\footnote{See, e.g., \cite{H}, although the result is folkloric,
	and well known in less austere settings \cite{Ke}.}; thus, such a characterization is in some sense
an analogue of Wiener's criterion for singular, rather than continuous data.

Theorem~\ref{thm:hmu} in the present form appears in \cite[Theorem 1.2]{AHMNT}. That \ref{1-thm:hmu} implies \ref{2-thm:hmu} is the main result in \cite{HMU} (see also \cite{HM4, HLMN});  that \ref{2-thm:hmu} yields \ref{3-thm:hmu} is \cite[Theorem 1.1]{AHMNT}; and the fact that \ref{3-thm:hmu}  implies \ref{1-thm:hmu} was proved in \cite{DJe}, and independently in \cite{Sem}. 

Theorem \ref{thm:hmu} 
and other recent results\footnote{We refer the reader also to recent work of Azzam \cite{Az}, in 
	which the author characterizes the domains with Ahlfors regular boundaries for which $\omega_{-\Delta}\in A_\infty(\sigma)$:  they are precisely the domains with uniformly rectifiable boundary which are semi-uniform in the sense of Aikawa and Hirata \cite{AH};  see also
	\cite{AHMMT, AMT-char, HM3} for related results characterizing $L^p$ solvability in the general
	case that $\omega_{-\Delta}$ need not be doubling.} illuminate how the $A_\infty$ condition\footnote{And also its
	non-doubling version, the weak $A_\infty$ condition.} 
of  harmonic measure is related to the geometry of the domain $\Omega$. Unfortunately, as we pointed out above, their proofs do not extend to the optimal class of operators with variable coefficients. Indeed, the best known results in this direction  pertain to the ``direct" rather than the ``free boundary" problem. A description of the elliptic measure in a given geometric environment, is essentially due to C. Kenig and  J. Pipher. 
In 2001  \cite{KP} C. Kenig and  J. Pipher proved what they referred to as a 1984 Dahlberg conjecture: if $\Omega\subset \RR^n$ is a bounded Lipschitz domain and the 
elliptic matrix $\mathcal{A}$ satisfies the following Carleson measure condition:
\begin{equation}\label{KP-cond}
\sup_{\substack{q\in\pO \\ 0< r<\diam(\Omega)} } \frac{1}{r^{n-1}} \iint_{B(q,r) \cap\Omega} 
\bigg(
\sup_{Y\in B(X,\frac{\delta(X)}{2})} |\nabla \mathcal{A}(Y)|^2 \delta(Y)\bigg)dX <\infty,
\end{equation}
where here and elsewhere we write $\delta(\cdot)=\dist(\cdot,\partial\Omega)$, then the 
corresponding elliptic measure $\omega_L\in A_{\infty}(\sigma) $.  
As observed in \cite{HMT1}, 	
one may carry through the proof in \cite{KP}, essentially 
unchanged, with a slightly weakened reformulation of  \eqref{KP-cond}, namely
by assuming, in place of  \eqref{KP-cond}, the following properties: 
\begin{enumerate}[label=(H\arabic*), itemsep=0.2cm]
	\item\label{H1} $\wcalA\in \Lip_{\rm loc}(\Omega)$ and $|\nabla \wcalA|\delta(\cdot) \in L^\infty(\Omega)$, where $\delta(\cdot) := \dist(\cdot,\pO)$.
	\item\label{H2} $|\nabla \wcalA|^2 \delta(\cdot)$ satisfies the Carleson measure assumption:
	\begin{equation}\label{KP-s-relaxed}
	\|\mathcal{A}\|_{\rm Car}:=\sup_{\substack{q\in\pO\\0< r<\diam(\Omega)} } \frac{1}{r^{n-1}} \iint_{B(q,r) \cap\Omega} 
	|\nabla \mathcal{A}(X)|^2 \delta(X)dX <\infty\,.
	\end{equation}	
\end{enumerate}
We shall refer to these hypotheses (jointly) as the \textbf{Dahlberg-Kenig-Pipher  (DKP)  condition}.
Note that each of \ref{H1} and \ref{H2} is implied by \eqref{KP-cond}.

Since properties
\ref{H1} and \ref{H2} are preserved in subdomains, one can use
the method of \cite{DJe} to extend the result of \cite{KP} to chord-arc domains, 
and hence the analogue of \ref{3-thm:hmu} implies \ref{1-thm:hmu} (in Theorem \ref{thm:hmu}) 
holds for operators satisfying the DKP condition.

An attempt to address the ``free boundary" part of the problem that is to prove that  \ref{1-thm:hmu} implies  \ref{2-thm:hmu} or  \ref{3-thm:hmu} 
led, the first, second and fourth authors of the present paper (see \cite{HMT1}) to show that under the same background hypothesis as in Theorem \ref{thm:hmu},  \ref{1-thm:hmu} implies 
 \ref{3-thm:hmu} (and hence also  \ref{2-thm:hmu})
 for elliptic operators with variable-coefficient matrices $\mathcal{A}$ 
 satisfying \ref{H1} and the Carleson measure estimate 
\begin{equation}\label{HMT-CM}
\sup_{\substack{q\in\pO \\ 0< r<\diam(\Omega)} } \frac{1}{r^{n-1}} \iint_{B(q,r) \cap\Omega} 
|\nabla \mathcal{A}(X)|dX <\infty.
\end{equation}
We observe that, in the presence of hypothesis \ref{H1}, 
\eqref{HMT-CM}
implies \eqref{KP-s-relaxed}. The weighted
$W^{1,2}$ Carleson measure
estimate \eqref{KP-s-relaxed} 
 is both weaker, and more natural than the $W^{1,1}$ version \eqref{HMT-CM}.
 For example, operators verifying 
 \eqref{KP-s-relaxed} arise as pullbacks of constant coefficient operators (see \cite[Introduction]{KP}), and also
 in the linearization of ``$A$-harmonic" (i.e., generalized $p$-harmonic) operators (see \cite[Section 4]{LV}).
We also mention in
 passing that a qualitative version of the results in \cite{HMT1} was obtained in \cite{ABHM}.
There are also
related (quantitative) results in \cite{HMM} and \cite{AGMT} that are 
valid in the absence of any connectivity hypothesis.



\subsection{Main Result and Proof Techniques}
From the geometric measure theory point of view the main motivation for this paper is to understand whether the  elliptic measure of a DKP divergence form elliptic 
operator distinguishes between a rectifiable and a purely unrectifiable boundary. As in Theorem \ref{thm:hmu}, we make the background assumption
that $\Omega\subset \mathbb{R}^n$, $n\ge 3$, is a uniform domain
(see Definition \ref{def:uniform}) 
 with an Ahlfors regular boundary (Definition 
\ref{def:ADR}).
Analytically we consider second order divergence form elliptic
operators, that is, 
$L=-\divg(\mathcal{A}(\cdot)\nabla)$, where $\mathcal{A}= \big( a_{ij}\big)_{i,j=1}^n$ is a (not necessarily symmetric) 
real matrix-valued function on $\Omega$, 
satisfying the usual uniform ellipticity condition 
\begin{equation}\label{def:UE}
 \langle \mathcal{A}(X)\xi, \xi  \rangle \ge \lambda |\xi|^2,
\qquad
| \langle \mathcal{A}(X)\xi, \zeta \rangle | \leq \Lambda |\xi \,|\zeta|, \qquad\text{for all } \xi,\zeta \in\mathbb{R}^n\setminus\{0\}\,,
\end{equation}
for uniform constants $0<\lambda\le \Lambda<\infty$, and for a.e. $X\in\Omega$.
We further assume that $\mathcal{A}$ 
satisfies the Dahlberg-Kenig-Pipher condition, that is, \ref{H1} and \ref{H2}, and, additionally, that the associated 
elliptic measure is an $A_\infty$ weight (see Definition \ref{def:AinftyHMU}) with 
respect to the surface measure $\sigma=\mathcal{H}^{n-1}|_{\pO}$. 
Our goal is to understand how this 
analytic information yields insight on the geometry of the
domain and its boundary. 


Our main result is as follows:
\begin{theorem}\label{thm:main}
	Let $\Omega\subset \RR^n$, $n\ge 3$, be a uniform domain with Ahlfors regular boundary and set $\sigma=\mathcal{H}^{n-1}|_{\pO}$. Let $\wcalA$ be a (not necessarily symmetric) uniformly elliptic matrix on $\Omega$ satisfying \textup{\ref{H1}} and \textup{\ref{H2}}. Then the following are equivalent:
\begin{enumerate}[label=\textup{(\arabic*)}, itemsep=0.2cm] 
		\item\label{1-thm-main} The elliptic measure $\omega_L$ associated with the operator $L=-\divg(\wcalA(\cdot)\nabla)$ is of class $A_\infty$ with respect to the surface measure. 
		\item\label{3-thm-main} $\pO$ is uniformly rectifiable.
		\item\label{2-thm-main} $\Omega$ is a chord-arc domain.
		\end{enumerate} 
\end{theorem} 
\begin{remark}
	In Corollary \ref{cOsc} we show that Theorem \ref{thm:main} remains true when we replace the assumptions \ref{H1} and \ref{H2} by a slightly weaker assumption involving the oscillation of the elliptic matrix in place of its gradient.
\end{remark}

The equivalence of  \ref{3-thm-main} and \ref{2-thm-main} (under the stated background hypotheses)
was previously known: that \ref{2-thm-main} $\implies$ \ref{3-thm-main} follows from
the main geometric result of \cite{DJe} (namely,
that chord-arc domains can be approximated in a big pieces sense by 
Lipschitz subdomains), and the converse  \ref{3-thm-main}
$\implies$ \ref{2-thm-main}  is proved in \cite{AHMNT}. 
Moreover, as mentioned above, it was also known
that \ref{2-thm-main} $\implies$ \ref{1-thm-main}, and the proof
comprises two main ingredients:
first, that the properties \ref{H1} and \ref{H2} are preserved in subdomains, and therefore
by the result of \cite{KP}\footnote{The formulation in terms of  
 \ref{H1} and \ref{H2} in place of \eqref{KP-cond} appears in \cite{HMT1}, but the result is implicit
in \cite{KP}; see \cite[Appendix A]{HMT1}.}, $\omega_L \in A_\infty (\sigma)$ in a Lipschitz subdomains
of $\Omega$;
and second, by the aforementioned big piece approximation result of  \cite{DJe},
that the $A_\infty$ property may be passed
from Lipschitz subdomains to the original chord-arc domain, by use of the maximum principle and a 
change of pole argument (see \cite{DJe} or, originally, \cite{JK}). 
In this paper we close the circle by proving the implication \ref{1-thm-main} $\implies$ \ref{3-thm-main}, thus providing a 
characterization of chord-arc domains in terms of the properties of the elliptic measure.

We first prove the implication \ref{1-thm-main} $\implies$ \ref{3-thm-main} in the ``small constant case'', that is, when the Carleson condition \ref{H2} holds with a sufficiently small constant (i.e, $\|\mathcal{A}\|_{\rm Car}$ is small, see \eqref{KP-s-relaxed}). To finish the proof of Theorem \ref{thm:main}, 
we then utilize a bootstrapping argument to pass from the case of small Carleson norm to 
the ``large constant case'', in which the Carleson condition \ref{H2} is assumed merely to be finite. The arguments we use to treat the small and large constant cases are quite different in nature, and each is of 
independent interest in its own right. (For example, what we prove in the small constant case is stronger than what is necessary for Theorem \ref{thm:main}.) Therefore we divide the proof into two parts, and deal with the small and large constant cases in Parts \ref{part:small} and \ref{part:large}, respectively. In the end of Part \ref{part:large} we also discuss the optimality of the above theorem.

Throughout this paper, and unless otherwise specified, by \textit{allowable constants}, we mean the dimension $n\geq 3$; the constants involved in the definition of a uniform domain, that is, $M, C_1>1$ (see Definition \ref{def:uniform}); the Ahlfors regular constant $C_{AR}>1$ (see Definition \ref{def:ADR}); the ratio of the ellipticity constants $\Lambda/\lambda \geq 1$ (see \eqref{def:UE}), and the $A_{\infty}$ constants $C_0>1$ and $\theta\in(0,1)$ (see Definition \ref{def:AinftyHMU}).


In Part \ref{part:small}, we develop an approach which combines, or rather interlaces, the ``classical" free boundary blow-up and compactness arguments (originated in geometric measure theory) with the scale-invariant harmonic analysis methods. This allows us to take advantage of the appropriate amelioration of the coefficients obtained via a compactness approach to show that the desired uniform rectifiability follows from regularity of elliptic measure whenever the coefficients of the underlying equation exhibit {\it small} oscillations, in the appropriate Carleson measure sense (see Theorem \ref{thm:main:I} and Corollary \ref{corol:DKP-small}). The smallness condition, while obviously suboptimal, could not be removed directly, for it is essentially built in the nature of the compactness arguments. To be precise, we prove the following:

\begin{theorem}\label{thm:main:I}
Given the values of allowable constants $n\ge 3$, $M, C_1, C_{AR}>1$, $\Lambda \ge \lambda =1$, 
$C_0>1$, and $0<\theta<1$, there exist $N$ and $\epsilon >0$ depending on the allowable constants, such that the following holds.  Let $\Omega\subset\mathbb{R}^n$ be a bounded uniform domain with constants $M, C_1$ and whose boundary $\pO$ is Ahlfors regular with constant $C_{AR}$ and set $\sigma=\mathcal{H}^{n-1}|_{\pO}$. Let $L = -\divg(\mathcal{A}(\cdot)\nabla)$ be an elliptic operator with real \textbf{symmetric} matrix $\mathcal{A}$ satisfying \eqref{def:UE} with ellipticity constants $1=\lambda\leq \Lambda$ such that the corresponding elliptic measure satisfies $\omega_{L} \in A_{\infty}(\sigma)$ with constants $C_0$ and $\theta$. If $\mathcal{A}$ verifies  
	\begin{equation}\label{def:oscA}
\osc(\Omega, \mathcal{A}) := \sup_{X\in\Omega} \fint_{B(X,\delta(X)/2)} |\mathcal{A}(Y) - \langle\mathcal{A}\rangle_{B(X,\delta(X)/2)}| dY <\epsilon,
\end{equation}
where $\delta(\cdot)=\dist(\cdot,\partial\Omega)$ and $\langle\wcalA\rangle_{B(X,\delta(X)/2)}$ denotes the average of $\wcalA$ on $B(X,\delta(X)/2)$,  
then $\Omega$ satisfies the exterior corkscrew condition with constant $N$, and hence $\Omega$ has uniformly rectifiable boundary.
\end{theorem}

\begin{remark}\label{remark:mainthm}
\ \vskip-.8cm\  

\begin{enumerate}[label=\textup{(\roman*)}, itemsep=0.2cm]  

\item Given $\mathcal{A}\in W^{1,1}_{\rm loc }(\Omega)$ we introduce 
\begin{equation}\label{def:smallCarleson}
\C(\Omega,\mathcal{A}) : = \sup_{X\in\Omega}  \fint_{B\left(X,\delta(X)/2 \right)} |\nabla \mathcal{A}(Y)|\delta(Y)  dY.
\end{equation}
Poincaré's inequality easily yields that $\osc(\Omega, \mathcal{A})\le C_n \C(\Omega,\mathcal{A})$ with $C_n$ depending just on dimension. Hence, one can formulate a version of Theorem~\ref{thm:main:I} (with a slightly different $\epsilon$) with $\C(\Omega,\mathcal{A})<\epsilon$ replacing \eqref{def:oscA}.

\item We note that our assumption \eqref{def:oscA} on the matrix $\mathcal{A}$ is much weaker than the smallness of the relaxed DKP condition \eqref{KP-s-relaxed}. To see this, given $X\in\Omega$, let $q_X \in\pO$ be such that $|X-q_X| = \delta(X)$. Then by Hölder's inequality
\begin{equation}\label{holder-cond}
\fint_{B\left(X,\delta(X)/2\right)} |\nabla \mathcal{A}(Y)|\delta(Y)dY
\lesssim 
\left(\frac{1}{\delta(X)^{n-1}} \int_{B\left(q_X, 3\delta(X)/2 \right) \cap\Omega} |\nabla \mathcal{A}(Y)|^2 \delta(Y) dY\right)^{\frac12} 
.\end{equation}
Hence \eqref{KP-s-relaxed} with sufficiently small constant gives smallness of $\C(\Omega,\mathcal{A})$ (and hence \eqref{def:oscA}). On the other hand, it is easy to see that the latter is much weaker. Assume for instance that $|\nabla \mathcal{A}|\delta\sim \epsilon$ in $\Omega$ in which case  
$\C(\Omega,\mathcal{A})\sim\epsilon$ but \eqref{KP-s-relaxed} fails since every integral is infinity.

%


\item In the hypothesis of Theorem~\ref{thm:main:I}, boundedness of the domain and symmetry of the operator might seem restrictive and it is very likely that the proof can be modified to remove those restrictions. Nevertheless, Theorem~\ref{thm:main:I} as stated is enough
to prove Theorem~\ref{thm:main} and we leave the details to the interested reader. 

\end{enumerate}    
\end{remark}

As a consequence we immediately obtain the ``small constant'' case of Theorem \ref{thm:main}: 
\begin{corollary}\label{corol:DKP-small}
Under the same background hypothesis as in Theorem~\ref{thm:main:I} if $\wcalA$ is a symmetric uniformly elliptic matrix on $\Omega$ satisfying \textup{\ref{H1}} and \textup{\ref{H2}} with $	\|\mathcal{A}\|_{\rm Car}<\epsilon$ and $\omega_{L} \in A_{\infty}(\sigma)$ then $\pO$ is uniformly rectifiable.
\end{corollary}

\begin{remark}\label{normalize}
We note that the $A_\infty$ constants for $\omega_L$ are not affected by the normalization $\lambda = 1$, however,
the small parameter $\epsilon$ in \eqref{def:oscA} clearly depends upon this normalization.
\end{remark}

\begin{remark}\label{rem:subtle}
Having fixed the desired 
ellipticity constants $\lambda=1$ and $\Lambda$ and the geometric parameters $M, C_1, C_{AR}>1$, one may ask whether operators $L= -\divg(\mathcal{A}(\cdot)\nabla)$ such that $\omega_L\in A_{\infty}(\sigma)$ and 
 $\wcalA$ satisfies \eqref{def:oscA} with 
small constant $\epsilon$ exist. Choosing a matrix for which the left-hand side of \eqref{KP-cond} is small (e.g., a constant coefficient matrix), we can guarantee that \eqref{def:oscA} holds with a desired $\epsilon$, see Remark~\ref{remark:mainthm}. It is a consequence of the work in \cite{KP} that on a chord arc domain (see Definition \ref{def:nta})
the $A_{\infty}$ constants of $\omega_L$ only depend on the ellipticity constants, the norm \eqref{KP-cond} and the geometric parameters (which include
$M, C_1, C_{AR}>1$). Thus in this case there exist constants $C_0>1$ and $\theta\in (0,1)$ such that all the conditions of Theorem~\ref{thm:main:I} are satisfied.
\end{remark}

The proof in Part \ref{part:large} is based on the method of ``extrapolation of Carleson measures", by means of which we bootstrap
the small constant case treated in Part \ref{part:small}.  This method was first introduced in the work of 
Lewis and Murray \cite{LM}, based on the Corona construction which has its origins
in the work of Carleson \cite{Ca}, 
and Carleson and Garnett \cite{CG}.  
In order to carry out this procedure, we go back and forth between the original domain $\Omega$ and certain \textit{sawtooth domains},  taking advantages of the scale-invariant nature of Carleson measures. In particular we need to transfer the $A_\infty$ property of
elliptic measures, from the original domain to its sawtooth subdomains.  This last step is really the 
heart of the proof.







\medskip

{\bf Acknowledgments:} The authors would like to express their gratitude to Bruno Giuseppe Poggi Cevallos who pointed out that the examples in \cite{MM} could be used to access the optimality of our results. See Proposition \ref{p:bruno}. They would also like to thank MSRI for its hospitality during the Spring of 2017, all the authors were in residence there when this work was started. 

\section{Preliminaries}\label{sPrelim}

In this section we state the definitions and some auxiliary results that will be used throughout the paper. 



\subsection{Definitions}\label{sec:def}

\begin{defn}\label{def:ADR}
	We say a closed set $E\subset \RR^n$ is \textbf{Ahlfors regular} with constant $C_{AR}>1$ if for any $q\in E$ and $0<r<\diam(E)$,
	\[ C_{AR}^{-1}\, r^{n-1} \leq \mathcal{H}^{n-1}(B(q,r)\cap E) \leq C_{AR}\, r^{n-1}. \]
\end{defn}

There are many equivalent characterizations of a uniformly rectifiable set, see \cite{DS2}. Since uniformly rectifiability is not the main focus of our paper, we only state one of the geometric characterizations as its definition. 
\begin{defn}\label{def:UR}
	An Ahlfors regular set $E\subset \RR^n$ is said to be \textbf{uniformly rectifiable}, if it has big pieces of Lipschitz images of $\RR^{n-1}$. That is, there exist $\theta, M>0$ such that for each $q\in E$ and $0<r<\diam(E)$, there is a Lipschitz mapping $\rho: B_{n-1}(0, r) \to \RR^n$ such that $\rho$ has Lipschitz norm $\leq M$ and 
	\[ \mathcal{H}^{n-1} \left( E\cap B(q,r) \cap \rho(B_{n-1}(0,r)) \right) \geq \theta r^{n-1}. \]
	Here $B_{n-1}(0,r)$ denote a ball of radius $r$ in $\RR^{n-1}$.
\end{defn}

\begin{defn}\label{def:ICC}
An open set $\Omega\subset\mathbb{R}^n$ is said to satisfy the \textbf{\textup{(}interior\textup{)} corkscrew condition} \textup{(}resp. the exterior corkscrew condition\textup{)} with constant $M>1$ if for every $q\in\pO$ and every $0< r<\diam(\Omega)$, there exists $A=A(q,r) \in \Omega$ \textup{(}resp. $A\in \Omega_{\rm ext}:=\mathbb{R}^n\setminus\overline{\Omega}$\textup{)} such that
 \begin{equation}\label{eqn:nta-M}
 	B\left(A, \frac{r}{M} \right) \subset B(q,r) \cap \Omega
	\qquad
	\Big(\mbox{resp. }B\left(A, \frac{r}{M} \right) \subset B(q,r) \cap \Omega_{\rm ext}.
	\Big)
	 \end{equation} 
	 The point $A$ is called  a Corkscrew point (or a non-tangential point)  relative to $\Delta(q,r)=B(q,r)\cap\pom$ in $\Omega$ (resp. $\Omega_{\rm ext}$).
\end{defn}

\begin{defn}\label{def:HCC}
An open connected set $\Omega\subset\mathbb{R}^n$ is said to satisfy the \textbf{Harnack chain condition} with constants $M, C_1>1$ if for every pair of points $A, A'\in \Omega$
there is a chain of balls $B_1, B_2, \dots, B_K\subset \Omega$ with $K \leq  M(2+\log_2^+ \Pi)$ that connects $A$ to $A'$,
where
\begin{equation}\label{cond:Lambda}
\Pi:=\frac{|A-A'|}{\min\{\delta(A), \delta(A')\}}.
\end{equation} 
Namely, $A\in B_1$, $A'\in B_K$, $B_k\cap B_{k+1}\neq\emptyset$ and for every $1\le k\le K$
\begin{equation}\label{preHarnackball}
 	C_1^{-1} \diam(B_k) \leq \dist(B_k,\partial\Omega) \leq C_1 \diam(B_k).
\end{equation}
         \end{defn}

We note that in the context of the previous definition if $\Pi\le 1$ we can trivially form the Harnack chain $B_1=B(A,3\delta(A)/5)$ and $B_2=B(A', 3\delta(A')/5)$ where \eqref{preHarnackball} holds with $C_1=3$. Hence the Harnack chain condition is non-trivial only when $\Pi> 1$.

\begin{defn}\label{def:uniform}
An open connected set $\Omega\subset\RR^n$ is said to be a \textbf{uniform} domain with constants $M, C_1$,  if it satisfies the interior corkscrew condition with constant $M$ and the Harnack chain condition with constants $M, C_1$.
\end{defn}

\begin{defn}\label{def:nta}
A uniform domain $\Omega\subset\RR^n$ is said to be \textbf{NTA} if it satisfies the exterior corkscrew condition. If one additionally assumes that $\partial\Omega$ is Ahlfors regular, the $\Omega$ is said to be a \textbf{chord-arc} domain.
\end{defn}

For any $q\in\pO$ and $r>0$, let $\Delta=\Delta(q,r)$ denote the surface ball $B(q,r) \cap \pO$, and let $T(\Delta)=B(q,r)\cap\Omega$ denote the Carleson region above $\Delta$. We always implicitly assume that $0<r< \diam ( \Omega)$. We will also write $\sigma=\mathcal{H}^{n-1}|_{\pO}$.

Given an open connected set $\Omega$ and an elliptic operator $L$ we let $\{\omega_L^X\}_{X\in\Omega}$ be the associated elliptic measure.  In the statement of Theorem~\ref{thm:main:I} we assume that $\omega_L \in A_{\infty}(\sigma)$ in the following sense:

\begin{defn}\label{def:AinftyHMU}
	The elliptic measure associated with $L$ in $\Omega$ is said to be of class $A_{\infty}$ with respect to the surface measure $\sigma= \mathcal{H}^{n-1}|_{\pO}$, which we denote by $\omega_L\in A_\infty(\sigma)$, if there exist $C_0>1$ and $0<\theta<\infty$ such that for any surface ball $\Delta(q,r)=B(q,r)\cap \pO$, with $x\in\partial\Omega$ and $0<r<\diam (\Omega)$, any surface ball $\Delta'=B'\cap \pO$ centered at $\pO$ with $B'\subset B(q,r)$, and any Borel set $F\subset \Delta'$, the elliptic measure with pole at $A(q,r)$ (a corkscrew point relative to $\Delta(q,r)$) satisfies
	   \begin{equation}\label{eqn:Ainfty}
		\frac{\omega_L^{A(q,r)}(F)}{\omega_L^{A(q,r)}(\Delta')} \leq C_0\left( \frac{\sigma(F)}{\sigma(\Delta')} \right)^{\theta}.
	\end{equation}
\end{defn}

Since $\sigma$ is a doubling measure, it is well-known that the condition $\omega_L\in A_\infty(\sigma)$ 
is equivalent to the fact that $\omega_L\in RH_q(\sigma)$ for some $q>1$ in the following sense: $\omega_L \ll \sigma$ and the
Radon-Nikodym derivative $\textbf{k}_L:= d\omega_L/d\sigma$ satisfies the reverse H\"older estimate
\begin{equation}\label{eq1.wRH}
\left(\fint_{\Delta'} \big(\textbf{k}_L^{A(q,r)}\big)^q d\sigma \right)^{\frac1q} \lesssim \fint_{\Delta'} \textbf{k}_L^{A(q,r)} \,d\sigma\,
=\frac{\omega_L^{A(q,r)}(\Delta')}{\sigma(\Delta')}\,,
\end{equation}
for all  $\Delta(q,r)=B(q,r)\cap \pO$, with $x\in\partial\Omega$ and $0<r<\diam (\Omega)$, any surface ball $\Delta'=B'\cap \pO$ centered at $\pO$ with $B'\subset B(q,r)$. 

\medskip

\begin{defn}\label{def:CDC}
A domain $\Omega\subset \mathbb{R}^n$ with $n\ge 3$ is said to satisfy the \textbf{capacity density condition (CDC)} if there exists a constant $c_0>0$ such that 
\begin{equation}\label{eqn:CDC}
	\frac{\capacity(B_r(q)\setminus\Omega)}{\capacity(B_r(q))} \geq c_0,\quad \text{for any } q\in\pO \text{ and } 0<r<\diam(\Omega),
\end{equation}	
where for any set $K\subset\mathbb{R}^n$, the capacity is defined as  
	\begin{equation*}
		\capacity(K) = \inf \bigg\{\int |\nabla \varphi|^2 dX: \varphi \in C_c^{\infty}(\mathbb{R}^n), K\subset \interior\{\varphi\geq 1\}\bigg\}.
	\end{equation*}
\end{defn}
It was proved in \cite[Section 3 ]{Zh} and \cite[Lemma 3.27]{HLMN} that a domain in $\mathbb{R}^n$, $n\ge 3$, with $(n-1)$-Ahlfors regular boundary satisfies the capacity density condition with constant $c_0$ depending only on $n$ and the Ahlfors regular constant $C_{AR}$. In particular such a domain is Wiener regular and hence for any elliptic operator $L$, and any function $f\in C(\pO)$, we can define
\begin{equation}\label{eqn:ellipt-rep}
	u(X) = \int_{\partial\Omega} f(q) d\omega_L^X(q),
	\qquad X\in\Omega,
\end{equation}
and obtain that $u\in W_{\rm loc}^{1,2}(\Omega)\cap C(\overline\Omega)$, $u|_{\partial\Omega}=f$ on $\pO$ and  $Lu = 0$ in $\Omega$ in the weak sense. Moreover, if additionally $f\in \Lip(\Omega)$ then $u\in W_{}^{1,2}(\Omega)$.

\subsection{Properties of solutions and elliptic measures}

For following lemmas, we always assume that $\Omega$ is a uniform domain with Ahlfors regular boundary (in fact they hold under the weaker assumption that $\Omega$ is a uniform domain the CDC). Let $L=-\divg(\wcalA(\cdot)\nabla)$ be a real uniformly elliptic operator, and we write $\omega=\omega_L$ for the corresponding elliptic measure. Although in our main result we consider non necessarily symmetric uniformly elliptic matrices, we will reduce matters to the symmetric case, in particular all the following properties will be used in that case, hence during this section we assume that $\wcalA$ is symmetric. 
All constants will only depend on the allowable constants, that is, those involved in the fact that the domain in question is uniform and has Ahlfors regularity boundary, and also in the uniform ellipticity of $\wcalA$. (Note that we may assume $\wcalA$ has been normalized so that $\lambda =1$ in \eqref{def:UE}.) In Part \ref{part:large} we will apply these lemmas to $\Omega$ as well as its sawtooth domains $\Omega_{\F,Q}$. 

Under the above assumptions, one can construct the associated elliptic measure $\omega_L$ and Green function $G$. For the latter the reader is referred to the work of Grüter and Widman \cite{GW}, while the existence of the corresponding elliptic measures is an application of the Riesz representation theorem. The behavior of $\omega$ and $G$, as well as the relationship between them, depends crucially on the properties of $\Omega$, and assuming that $\Omega$ is a uniform domain with the CDC one can follow the program carried out in \cite{JK}. We summarize below the results which will be used later in this paper. For a comprehensive treatment of the subject and the proofs we refer the reader to the forthcoming monograph \cite{HMT2} (see also \cite{Ke} for the case of NTA domains).

\begin{theorem}\label{thm:gw}
There is a unique non-negative function $G: \Omega \times \Omega \to \RR\cup \{\infty\}$, the 
Green function associated with $L$, and a positive, finite constant $C$, depending only on dimension, and 
(given the normalization) $\Lambda$, such that the following hold:
\begin{equation}
G(\cdot, Y) \in W^{1,2}(\Omega \setminus B(Y,s))\cap W^{1,1}_0(\Omega)\cap W_0^{1,r}(\Omega), \quad \forall\,Y\in\Omega,\ \forall\,s>0,\ \forall\,r\in \left[1,\tfrac{n}{n-1}\right);
\end{equation}
\begin{equation}
\int\left\langle \wcalA(X)\nabla_X G(X,Y), \nabla \varphi(X) \right\rangle dX = \varphi(Y),
\quad \text{for all }\varphi \in C_c^{\infty}(\Omega);
\end{equation} 
\begin{equation}\label{eqn:qweak}
		\|G(\cdot,Y)\|_{L^{\frac{n}{n-2},\infty}(\Omega)} +\|\nabla G(\cdot,Y)\|_{L^{\frac{n}{n-1},\infty}(\Omega)}\le C, \quad\forall\,Y\in \Omega;
\end{equation}
\begin{equation}\label{eqn:gub}
		G(X,Y) \leq C|X-Y|^{2-n};
	\end{equation}
and
\begin{equation}\label{eqn:glb}
		G(X,Y) \geq C|X-Y|^{2-n}, \quad\text{if\ \ } |X-Y|\leq \frac78{\delta(Y)}.
	\end{equation}	
Furthermore, if $\Omega$ is a uniform domain satisfying the CDC, for any $\varphi \in C_c^{\infty}(\RR^n)$ and for almost all $Y\in\Omega$
	\begin{equation}\label{eqn:int-parts}
		-\int_{\Omega} \left\langle \mathcal{A}(X)\nabla_X G(X,Y), \nabla \varphi(X) \right\rangle dX = \int_{\pO} \varphi d\omega^{Y} -\varphi(Y)
	\end{equation}
	where $\{\omega^Y\}_{Y\in\Omega}$ is the associated elliptic measure.

\end{theorem}

We observe that \eqref{eqn:qweak} and Kolmogorov's inequality give that for every $1\le r<\frac{n}{n-1}$ 
\begin{equation}\label{eqn:glq}
\|G(\cdot,Y)\|_{L^r(\Omega)}\le C C_3^{\frac1r} 
|\Omega|^{\frac1r-\frac{n-2}n},
\qquad
\|\nabla G(\cdot,Y)\|_{L^r(\Omega)}\le C C_4^{\frac1r}
|\Omega|^{\frac1r-\frac{n-1}n},
	\end{equation}
where $C$ is the constant in \eqref{eqn:qweak}, $C_4=(\frac{n}{(n-2)r})'$, and $C_4=(\frac{n}{(n-1)r})'$.

\begin{lemma}[Boundary H\"older regularity]\label{lem:vanishing}
	There exist constants $C, \beta>0$ (depending on the allowable constants) such that for $q\in\pO$ and $0<r<\diam(\pO)$, and $u\geq 0$ with $Lu=0$ in $ B(q,2r) \cap\Omega$, if $u$ vanishes continuously on $\Delta(q,2r) =  B(q,2r) \cap \pO$, then
	\begin{equation}\label{eqn:1.1}
		u(X) \leq C\left( \frac{|X-q|}{r} \right)^{\beta} \sup_{B(q, 2r)\cap\Omega} u,  \qquad \text{for any } X\in\Omega\cap B(q,r).
	\end{equation}
	\end{lemma}
	
\medskip

\begin{lemma}[Comparison principle]\label{lm:comp}
	Let $u$ and $v$ be non-negative solutions to $Lu=Lv=0$ in $B(q,4r)\cap\Omega$ which vanish continuously on $\Delta(q,4r)$.  Let $A=A(q,r)$ be a corkscrew point relative to $\Delta(q,r)$. Then
	\begin{equation}\label{P6}
		\frac{u(X)}{v(X)} \approx \frac{u(A)}{v(A)} \quad\text{~for any~} X\in B(q,r)\cap\Omega. 
	\end{equation} 
\end{lemma}


\begin{lemma}[Non-degeneracy of elliptic measure]\label{lm:Bourgain}
	There exists $m_0\in(0,1)$ depending on the allowable constants such that for any $q\in\pO$ and $0<r<\diam (\pO)$,
	\begin{equation}\label{eqn:1.2}
		\omega^{X} (\Delta(q,r)) \geq m_0 \qquad \text{for } X\in B(q,r/2)\cap \Omega.
	\end{equation}
\end{lemma}

\medskip

\begin{lemma}[Boundary Harnack inequality]\label{harn-princ}
	Let $\Omega$ be a uniform domain satisfying the CDC. There exists a constant $C$ (depending on the allowable constants) such that for $q\in\pO$ and $0<r<\diam(\pO)$. If $u\geq 0$ with $Lu=0$ in $\Omega\cap B(q,2r)$ and $u$ vanishes continuously on $\Delta(q,2r)$, then 
	\begin{equation}\label{eqn:1.3}
		u(X) \leq C u(A(q,r)), \qquad \text{for any}\ X\in\Omega\cap B(q,r).
	\end{equation}
	\end{lemma}
	
\medskip

\begin{lemma}[Change of pole formula]\label{lm:cop}
	Let $q\in\partial\Omega$ and $0<r<\diam(\partial\Omega)$ be given, and let $A = A(q,r)$ be a corkscrew point relative to $\Delta(q,r)$. Let $F, F' \subset \Delta(q,r)$ be two Borel subsets such that $\omega^A(F)$ and $\omega^A(F')$ are positive. Then 
	\begin{equation}
		\frac{\omega^X(F)}{\omega^X(F')} \approx \frac{\omega^A(F)}{\omega^A(F')}, \quad \text{ for any }X\in \Omega \setminus B(q,2r).
	\end{equation}
	In particular with the choice $F=\Delta(q,r)$, we have
	\begin{equation}
		\frac{\omega^X(F')}{\omega^X(\Delta(x,r))} \approx \omega^A(F') \quad \text{ for any } X\in\Omega \setminus B(q,2r).
	\end{equation}
\end{lemma}

\medskip

\begin{lemma}[CFMS estimate]\label{lm:CFMS}
	There exists a constant $C\geq 1$, such that for any $q\in\partial\Omega$, $0<r<\diam(\partial\Omega)/M$, and $A = A(q,r)$, a corkscrew point relative to $\Delta(q,r)$, the Green's function $G=G_L$ satisfies
	\begin{equation}\label{eq:CFMS}
		C^{-1} \frac{G(X_0,A)}{r} \leq \frac{\omega^{X_0}(\Delta(q,r))}{r^{n-1}} \leq C ~ \frac{G(X_0,A)}{r}
	\end{equation}
	for any $X_0 \in \Omega \setminus B(q,2r)$.
\end{lemma}

\medskip


%
%

\begin{lemma}[Doubling property of the elliptic measure]\label{lm:doubling}
	For every $q\in\partial\Omega$ and  $0<r<\diam(\partial\Omega)/4$,  we have
	\begin{equation}\label{eq:doubling}
		\omega^X(\Delta(q,2r)) \leq C \omega^X(\Delta(q,r))
	\end{equation}
	for any $X\in\Omega \setminus B(q,4r)$.
\end{lemma}

\begin{remark}\label{rem:doubling:needed}
 The following simple observation will be useful. If $M$ denotes the corkscrew constant for $\Omega$, it follows easily from the previous result,  Lemma \ref{eqn:1.2} and Harnack's inequality that 
\begin{equation}\label{doubling:needed}
\omega^X(\Delta(q,2r)) \leq C_2\omega^X(\Delta(q,r)),
\end{equation}
for every $q\in\partial\Omega$, $0<r<\diam(\partial\Omega)/4$ and for all $X\in\Omega$ with $\delta(X)\ge r/(2M)$. Here $C_2$ is a constant that depends on the allowable parameters associated with $\Omega$ and the ellipticity constants of $L$. 

\end{remark}

Our next result establishes that if a domain satisfies the Harnack chain condition then we can modify the chain of balls so that they avoid a non-tangential balls inside:

\begin{lemma}\label{lemm:NC-avoid}
Let $\Omega\subset \mathbb{R}^n$ be an open set satisfying the Harnack chain condition with constants $M, C_1>1$. Given $X_0\in\Omega$, let $B_{X_0}=B(X_0,\delta(X_0)/2)$. For every $X,Y\in \Omega\setminus \overline{B_{X_0}}$, if  we set $\Pi=|X-Y|/\min\{\delta(X), \delta(Y)\}$, then 
there is a chain of open Harnack balls $B_1, B_2, \dots, B_K\subset \Omega$ with $K \leq  100 (M+C_1^2)(2+\log_2^+ \Pi)$ that connects $X$ to $Y$. Namely, $X\in B_1$, $Y\in B_K$, $B_k\cap B_{k+1}\neq\emptyset$ for every $1\le k\le K-1$ 
and for every $1\le k\le K$
\begin{equation}\label{preHarnackball:new}
 	(100\,C_1)^{-2} \diam(B_k) \leq \dist(B_k,\partial\Omega) \leq 100\,C_1^2 \diam(B_k).
\end{equation}
Moreover, $B_k\cap \frac12 B_{X_0}=\emptyset$ for every $1\le k\le K$.
\end{lemma}

\begin{proof}
Fix $X, Y$ as in the statement and without loss of generality we assume that $\delta(X)\le\delta(Y)$. Use the Harnack chain condition for $\Omega$ to construct the chain of balls $B_1,\dots, B_K$ as in Definition \ref{def:HCC}. If none of $B_k$ meets $B_{X_0}$ then there is nothing to do as this original chain satisfies all the required condition.
Hence we may suppose that some $B_k$ meets $B_{X_0}$. The main idea is that then we can modify the chain of balls by adding some small balls that surround $X_0$. To be more precise, we let $k_-$ and $k_+$ be respectively the first and last ball in the chain meeting $B_{X_0}$. Note that $1\le k_-\le k_+\le K$.

We pick $X_{-}\in B_{k_-}\setminus\overline{B_{X_0}}$:  If $k_-=1$ we let $X_{-}=X$ or if $k_->1$ we pick $X_{-}\in B_{k_--1}\cap B_{k_-}$. Since $B_{k_-}$ meets $B_{X_0}$ then we can find  $Y_{-}\in B_{k_-}\cap\partial B_{X_0}$ such that the open  segment joining $X_{-}$ and $Y_{-}$ is contained in $B_{k_-}\setminus\overline{B_{X_0}}$. Analogously we can find $X_{+}\in B_{k_+}\setminus\overline{B_{X_0}}$ and $Y_{+}\in B_{k_+}\cap\partial B_{X_0}$ such that the open  segment joining $X_{+}$ and $Y_{+}$ is contained in $B_{k_+}\setminus\overline{B_{X_0}}$.

Next set $r=\delta(X)/(16 C_1)$ and let $N_\pm \ge 0$ be such that $N_\pm \le |X_{\pm}-Y_{\pm}|/r < N_\pm +1$. For $j=0,\dots, N_\pm$, let 
$$
B_\pm^j=B(X_{\pm}^j,r),
\qquad
\text{where}\quad
X_{\pm}^j= X_\pm+jr \frac{Y_{\pm}-X_{\pm}}{|Y_{\pm}-X_{\pm}|}
$$
Straightforward arguments show that $N_\pm\le 32 C_1^2$, $X_\pm\in B_\pm^0$, $Y_\pm\in B_\pm^{N_\pm}$, 
$B_\pm^j\cap B_\pm^{j+1} \neq\emptyset$ for every $0\le j\le N_\pm-1$,  and
$$
(32 C_1^2)^{-1}\diam(B_\pm^j)
\le
\dist(B_\pm^j, \partial\Omega)
\le
32 C_1^2\diam(B_\pm^j),
\qquad
B_\pm^j\cap\frac12 B_{X_0}=\emptyset,
$$
for every $0\le j\le N_\pm-1$.

Next, since $X_{\pm}\in\partial B_{X_0}$ we can find a sequence of balls $B^0,\dots, B^{N}$ centered at $\partial B_{X_0}$ and with radius $\delta(X)/16$ (hence $B^j\cap\frac12 B_{X_0}=\emptyset$) so that $N\le 64$, $Y_-\in B^0$, $Y_+\in B^N$, $B^j\cap B^{j+1}\neq\emptyset$ for $0\le j\le N-1$ and $32^{-1}\le \dist(B^j, \partial\Omega)/\diam(B^j)\le 32$.

Finally, to form the desired Harnack chain we concatenate the sub-chains 
$\{B_1, \dots B_{k_--1}\}$, $\{B_{-}^0,\dots B_-^{N_-}\}$, $\{B^0,\dots B^N\}$, $\{B_+^{N},\dots, B_+^0\}$, $\{B_{k_++1},\dots B_K\}$ and the resulting chain have all the desired properties. To complete the proof we just need to observe that the length of the chain is controlled by $K+N_-+N+N_++3\le 100 (M+C_1^2)(2+\log_2^+ \Pi)$.
\end{proof}

\subsection{Compactness of closed sets and Radon measures}

The reader may be familiar with the notion of convergence of compact sets in the Hausdorff distance; for general closed sets, not necessarily compact, we use the following notion of convergence, see \cite[Section 8.2]{DS1} for details. (It was pointed out to us that this notion is also referred to as the Attouch-Wets topology, see for example \cite[Chapter 3]{Be}.)
\begin{defn}[Convergence of closed sets]\label{def:cvsets}
	Let $\{E_j\}$ be a sequence of non-empty closed subsets of $\RR^n$, and let $E$ be another non-empty closed subset of $\RR^n$. We say that $E_j$ converges to $E$, and write $E_j \to E$, if
	\[ \lim_{j\to \infty} \sup_{x\in E_j \cap B(0,R)} \dist(x, E) = 0 \]
	and
	\[ \lim_{j\to \infty} \sup_{x\in E \cap B(0,R)} \dist(x, E_j) = 0 \]
	for all $R>0$. By convention, these suprema are interpreted to be zero when the relevant sets are empty.
\end{defn}

We remark that in the above definition, we may replace the balls $B(0,R)$ by arbitrary balls in $\RR^n$. The following compactness property has been proved in \cite[Lemma 8.2]{DS1}:
\begin{lemma}[Compactness of closed sets]\label{lm:cptHd}
	Let $\{E_j\}$ be a sequence of non-empty closed subsets of $\RR^n$, and suppose that there exists an $r>0$ such that $E_j \cap B(0,r) \neq \emptyset$ for all $j$. Then there is a subsequence of $\{E_j\}$ that converges to a nonempty closed subset $E$ of $\RR^n$ in the sense defined above.
\end{lemma}

Given a Radon measure $\mu$ on $\R^n$ (i.e., a non-negative Borel such that the measure of any compact set is finite) we define
$$
\spt\mu = \overline{\big\{ x\in\RR^n: \mu(B(x,r)) > 0 \text{ for any }r>0 \big\}}.
$$

\begin{defn}\label{def:mu-ADR}
We say that a Radon measure $\mu
$ on $\R^n$  is Ahlfors regular with constant $C\ge 1$, if there exits a constant $C\ge 1$ such that for any $x\in E$ and $0<r<\diam(E)$,
	\[ C^{-1}\, r^{n-1} \leq \mu(B(q,r)) \leq C\, r^{n-1},
	\qquad
	\forall\,x\in\spt\mu,\ 0<r<\diam(\spt\mu).	
	\]
\end{defn}

\begin{defn}\label{def:wcvm}
	Let $\{\mu_j\}$ be a sequence of Radon measures on $\RR^n$. We say $\mu_j$ converge weakly to a Radon measure $\mu_{\infty}$ and write $\mu_j \rightharpoonup \mu_{\infty}$, if 
	\[ \int f d\mu_j \to \int f d\mu_{\infty} \]
	for any $f\in C_c(\R^n)$.
\end{defn}

We finish this section by stating a compactness type lemma for Radon measures which are uniformly doubling and ``bounded below''.

\begin{lemma}[{\cite[Lemma 2.19]{TZ}}]\label{lm:sptcv}
Let $\{\mu_j\}_j$ be a sequence of Radon measures. Let $A_1, A_2>0$ be fixed constants, and assume the following conditions:
\begin{enumerate}[label=\textup{(\roman*)}, itemsep=0.2cm] 
		\item\label{1-lm:sptcv} $0\in \spt\mj$ and $\mj(B(0,1)) \geq A_1$ for all $j$,
		
		\item\label{2-lm:sptcv} For all $j\in\NN$, $q\in \spt \mj$ and $r>0$,
			\begin{equation}\label{unif-doubling}
				\mu_j(B(q,2r))\le A_2\mu_j(B(q,r))
			\end{equation}
\end{enumerate}
If there exists a Radon measure $\minf$ such that $\mj \rightharpoonup \minf$, then $\mu_\infty$ is doubling and
	\begin{equation}\label{eqn:spt-conv}
	 \spt \mj \to \spt \minf,
	 	 \end{equation}
	 	 in the sense of Definition \ref{def:cvsets}.
\end{lemma}

\part{Small constant case}\label{part:small}

The plan of Part \ref{part:small} is as follows. We prove Theorem \ref{thm:main:I} by contradiction using a compactness argument. More precisely, we assume that there is a family of bounded uniform domains with Ahlfors regular boundaries and a family of divergence form elliptic operators with associated elliptic measures in the class $A_\infty$ and with all the implicit constants uniformed controlled. We also have that the oscillations of the coefficients convergence to 0 and that the exterior corkscrew condition fails for each domain in the family. This is all detailed in Section~\ref{comp-tt}. The goal is to reach a contradiction and with that goal in mind we show in Section~\ref{sect:blowup} that passing to a subsequence there are a limiting domain which is uniform and has Ahlfors regular boundary and a limiting constant coefficient elliptic operator in that limiting domain whose associated elliptic measure belongs to the  class $A_\infty$. With this in hand in Section~\ref{section:proof-main} we are in a position to apply Theorem~\ref{thm:hmu} to obtain that the limiting domain satisfies the exterior corkscrew. This in turn leads us to a contradiction since the exterior corkscrew condition fails for any of the domains in the family.

\section{Compactness argument}\label{comp-tt}

To prove Theorem~\ref{thm:main:I} we will proceed by 
contradiction. First we discuss the constant $N$. 
Recall that, as noted above, the assertion that \ref{1-thm:hmu} implies \ref{3-thm:hmu} in Theorem \ref{thm:hmu}
extends routinely to all constant coefficient second order elliptic operators; alternatively, this fact follows
from the results of \cite{HMT1}
as \eqref{HMT-CM} and \ref{H1} holds trivially in the constant coefficient case.
Thus given values of the allowable constants 
$M, C_1, C_{AR},\Lambda/\lambda$, 
$C_0,\theta$, 
let $\Omega\subset\mathbb{R}^n$, $n\ge 3$, be a uniform domain with constants $M, C_1$, 
whose boundary is Ahlfors regular with constant $C_{AR}$, and let
$L = -\divg(\mathcal{A}_0\nabla)$ be a constant coefficient elliptic operator where the constant 
real symmetric matrix 
$\mathcal{A}_0$ satisfies \eqref{def:UE} with ellipticity constants $\lambda, \Lambda$,
and such that the corresponding elliptic measure $\omega_{L} \in A_{\infty}(\sigma)$ with 
constants $C_0$ and $\theta$.  Then
there exists a constant $N_0=N_0(M, C_1, C_{AR},\Lambda/\lambda,C_0,\theta)$  
such that  $\Omega$ satisfies the exterior corkscrew condition 
with constant $N_0$. We underline that this $N_0$ depends on the ratio of the ellipticity constants rather than the matrix $\mathcal{A}_0$ per se. 

With this in mind, set
 \begin{equation}\label{N}
 N=4N_0(4M,2C_1, 2^{5(n-1)}C_{AR}^2,\Lambda/\lambda, C_0  C_2 C_{AR}^{4\theta} 2^{8(n-1)\theta},\theta)
 \end{equation}
 where the constant 
$C_2=C_2(M,C_1, C_{AR}, \Lambda/\lambda)$ can be found in Remark \ref{rem:doubling:needed}.


We now state the contradiction hypothesis: 
for fixed $n\ge 3$, we
suppose that there exists a set of allowable constants $M, C_1, C_{AR}>1$, 
$\Lambda \ge \lambda = 1$,
$C_0>1$ and $0<\theta<1$, and a sequence $\epsilon_j$ (with $\epsilon_j \to 0$ as $j\to\infty$), so that the following holds:


\begin{enumerate}[label=\textup{\textbf{Assumption (\alph*):}},ref=\textup{\textbf{Assumption (\alph*)}}, itemsep=0.2cm, wide, leftmargin=1cm] 

	\item\label{1-assump} For each $j$ there is a bounded domain $\Omega_j\subset \mathbb{R}^n$, which is  uniform with constants $M, C_1$ and whose boundary is Ahlfors regular with constant $C_{AR}$. Also, there is an elliptic matrix  $\mathcal{A}_j$ defined on $\Oj$, with ellipticity constants $\lambda=1$ and $\Lambda$, and we write $L_j = -\divg(\mathcal{A}_j\nabla)$.
	

	\item\label{2-assump} $\osc(\Omega_j, \mathcal{A}_j) <\epsilon_j$ (see \eqref{def:oscA}).

	\item\label{3-assump} The elliptic measure of the operator $L_j$ in $\Omega_j$ is of class $A_{\infty}$ with respect to the surface measure $\sj=\HH^{n-1}|_{\partial\Omega_j}$ with constants $C_0$ and $\theta$ (see Definition \ref{def:AinftyHMU}).

\item[\textup{\textbf{Contrary to conclusion:}}] For each $j$ there is $q_j\in\partial\Omega_j$ and $0<r_j<\diam(\pOj)$ such that $\Omega_j$ has no exterior corkscrew point with constant $N$ (as in \eqref{N}). That is, there is no ball of radius $r_j/N$ contained in $B(q_j,r_j)\setminus\overline{\Omega_j}$.
\end{enumerate}


Our goal is to obtain a contradiction and as a consequence our main result Theorem~\ref{thm:main:I} will be proved.
Without loss of generality we may assume $q_j = 0$ and $r_j = 1$ for all $j$, hence $\diam(\partial\Omega_j)>1$. Otherwise,  we just replace the domain $\Omega_j$ by $(\Omega_j - q_j)/r_j$, and replace the elliptic matrix $\wcalA_j(\cdot)$ by $\wcalA_j(q_j + r_j \cdot)$. Note that the new domain and matrix have the same allowable constants, in particular the corresponding $A_{\infty}$ constants stay the same by the scale-invariant nature of Definition \ref{def:AinftyHMU}; moreover after rescaling, the above \ref{2-assump} is still satisfied:
\[ \osc\left(\frac{\Omega_j - q_j}{r_j}, \wcalA_j(q_j + r_j \cdot) \right) = \osc(\Omega_j, \wcalA_j)<\epsilon_j. \]

\section{Limiting domains}\label{sect:blowup}
We want to use a compactness argument similar to the blow-up argument in \cite{TZ}. The crucial difference is that in \cite{TZ}, the elliptic operator tends to a constant-coefficient operator as we zoom in on the boundary and blow up the given domain; whereas here we need to work with a \textit{sequence} of domains and their associated elliptic operators. In particular the geometric convergence of domains does not come for free, and more work is needed to analyze the limiting domain.

To be more precise, getting to the point where we can apply Theorem \ref{thm:hmu} (more precisely, its extension to the elliptic operators with constants coefficients or alternatively \cite{HMT1} applied again to constant coefficient operators) requires showing first that if $\Oinf$ is a ``limiting domain''  of the domains $\{\Omega_j\}$'s, then $\Oinf$ is an unbounded or bounded uniform domain with Ahlfors regular boundary. To accomplish this we also need to find the limit of the Green functions. Once we have this, to show that $\omega_{L_\infty} \in A_{\infty}(\sigma_{\infty})$ for the limiting domain $\Oinf$ and the limiting operator $L_{\infty}$, we need to construct the elliptic measure $\omega_{L_\infty}^Z$ for any $Z\in\Oinf$ as a limiting measure compatible with the procedure. We will also show that $L_\infty$ is an elliptic operator with constants coefficients.

Throughout the rest of paper we follow the following conventions in terms of notations:
\begin{itemize}\itemsep=0.2cm
\item For any $Z\in\Omega_j$ we write $\delta_j(Z) = \dist(Z,\partial\Omega_j)$.

\item For any $q\in\partial\Omega_j$ and $r\in (0, \diam( \partial\Omega_j))$, we use $A_j(q,r)$ to denote a corkscrew point in $\Omega_j$ relative to  $B(q,r)\cap\partial\Omega_j$, i.e.,
		\begin{equation} \label{nta-pt}
			B\left(A_j(q,r), \frac{r}{M} \right) \subset B(q,r) \cap \Oj.
		\end{equation}

\end{itemize}

\subsection{Geometric limit}
Since $\diam(\pOj) > 1$, modulo passing to a subsequence, one of the following two scenarios occurs:

\begin{enumerate}[label=\textup{\textbf{Case \Roman*:}} , ref=\textup{\textbf{Case \Roman*}}, itemsep=0.2cm, align=left, leftmargin=2cm] 
	
\item\label{CaseI} $\diam(\Oj) = \diam(\pOj) \to \infty $ as $j\to\infty$.

\item\label{CaseII} $\diam(\Oj)=\diam(\pOj) \to R_0 \in [1,\infty)$ as $j\to\infty$.

\end{enumerate}

Therefore if $\Oj$ ``converges'' to a limiting domain $\Oinf$, respectively \ref{CaseI} and \ref{CaseII} indicate that $\Oinf$ is unbounded or bounded.

Let $X_j\in \Omega_j$ be a corkscrew point relative to  $B(0,\diam(\Oj)/2)\cap\partial\Omega_j$, then
\begin{equation}\label{eq:polej}
	|X_j| \sim \wdj(\wXj) \sim \diam(\Oj),
\end{equation} 
with constants depending on the uniform constant $M$.
Let $G_j$ be the Green function associated with $\Oj$ and the operator $L_j = -\divg(\wcalA_j \nabla)$, and $\{\oj^X\}_{X\in\Oj}$ be the corresponding elliptic measure. 
In  \ref{CaseI}  we have
\begin{equation}\label{eq:temp1}
	|\wXj| \sim \wdj(\wXj) \sim \diam(\Oj) \to \infty,
\end{equation}
i.e., the poles $\wXj$ tend to infinity eventually. We let
\begin{equation}\label{def:ujcasei}
	u_j(Z) = \frac{G_j(X_j,Z)}{\oj^{X_j}(B(0,1))}.
\end{equation}
In \ref{CaseII}, we may assume that $\diam(\Oj) \sim R_0$ for all $j$ sufficiently large (one could naively rescale again so that $R_0=1$, should that be the case one may lose the property that $r_j=1$ for all $j$). Hence, there are constants $0<c_1<c_2$ such that
\begin{equation}\label{eq:wXjcase2}
	c_1 R_0 \leq \wdj(\wXj) \le |\wXj| \leq c_2 R_0 \quad \text{ for all } j \text{ sufficiently large}.
\end{equation}
Thus modulo passing to a subsequence, $\wXj$ converges to some point $X_0$ satisfying 
\begin{equation}\label{eq:polecase2}
	c_1 R_0 \leq |X_0| \leq c_2 R_0.
\end{equation}
Note that \eqref{eq:wXjcase2} and \eqref{eq:polecase2} in particular imply that for any $\rho$ sufficiently small (depending on $R_0$ and $c_1, c_2$), the ball $B(X_0,\rho)$ is contained in $\Oj$ and $\dist(B(X_0,\rho),\pOj)\ge c_1 R_0/2$.
In this case we let
\begin{equation}\label{def:ujcaseii}
	u_j(Z) = G_j(X_j,Z).
\end{equation}

Our next goal is to describe what happens with the objects in question as we let $j\to\infty$. 
This is done in Theorems \ref{thm:pseudo-blow-geo}, \ref{thm:AW11}, \ref{thm:blow-ana-pole} below.

\begin{theorem}\label{thm:pseudo-blow-geo}
\leavevmode
Under  \ref{1-assump}, and using the notation above, we have the following properties (modulo passing to a subsequence which we relabel):
\begin{enumerate}[label=\textup{(\arabic*)}, itemsep=0.2cm] 
	\item\label{1-thm:pseudo-blow-geo} \ref{CaseI}: there is a function $u_\infty\in C(\R^n)$ such that $\uj \to u_\infty$ uniformly on compact sets; moreover $\nabla \uj\rightharpoonup\nabla u_\infty$ in $L^2_{\rm loc}(\R^n)$. 
	
\item \label{2-thm:pseudo-blow-geo} \ref{CaseII}: there is a function $u_\infty \in C(\R^n\setminus \{X_0\})$ such that $\uj \to u_{\infty}$ uniformly on compact sets in $\R^n \setminus\{X_0\}$ and $\nabla \uj \rightharpoonup \nabla u_{\infty}$ in $L_{\rm loc}^2(\R^n \setminus \{X_0\})$.

\item \label{3-thm:pseudo-blow-geo} Let $\Oinf=\{Z\in\R^n: u_\infty >0\}$\footnote{In \ref{CaseII}, see Remark \ref{rem:deawffr} part \ref{2-rem:deawffr} we extend $u_\infty$ to all of $\R^n$ by setting $u_\infty(X_0)=+\infty$.}. Then $\overline{\Oj}\to\overline{\Omega_\infty}$ and $\pOj\to \pOinf$, in the sense of Definition \ref{def:cvsets}. Moreover, $\Oinf$ is an unbounded set with  unbounded  boundary in \ref{CaseI}, and it is bounded with diameter $R_0 \geq 1$ in \ref{CaseII} .
	
	\item \label{4-thm:pseudo-blow-geo} $\Oinf$ is a nontrivial uniform domain with constants $4M$ and $2C_1$. 
	
	\item \label{5-thm:pseudo-blow-geo}  There is an Ahlfors regular measure $\mu_\infty$ with constant $2^{2(n-1)}C_{AR}$ such that $\sj \rightharpoonup \mu_\infty$. Moreover, $\spt\mu_\infty=\pOinf$. In particular, this implies that 
	\begin{equation}  \label{comp-mu-H}
	2^{-3(n-1)}C_{AR}^{-1} \minf \leq \mathcal{H}^{n-1}|_{\pOinf} \leq 2^{3(n-1)}C_{AR} \minf. 
	\end{equation}
	and hence $\partial\Omega_\infty$ is Ahlfors regular with constant $2^{5(n-1)}C_{AR}^2$.
\end{enumerate}

\end{theorem}

\begin{remark}
Note that this result is purely geometric. The proof only uses   \ref{1-assump}, which states the geometric characters of domains $\Omega_j$ (i.e., they are uniform domains with Ahlfors regular boundaries) and the ellipticity of the matrix operator $\mathcal{A}_j$. The other assumptions are irrelevant for this.
\end{remark}


\begin{proof}[Proof of \ref{1-thm:pseudo-blow-geo} in Theorem \ref{thm:pseudo-blow-geo}]
Let $R>1$ and  note that for $j$ large enough (depending on $R$) we have that $\wXj \notin B(0,4R)$ since by \eqref{eq:temp1} 
\[ 
|\wXj| = |\wXj - 0| \geq \wdj(\wXj) \sim \diam(\Oj) \to \infty,
\quad\text{as }j\to\infty.
\]	
In particular, $L_j\uj=0$ in $B(0,4R)\cap\Oj$ in the weak sense. Recall that all our domains $\Omega_j$ have Ahlfors regular boundary and hence all boundary points are Wiener regular. This in turn implies that $\uj$ is a non-negative $L$-solution on $B(0,4R)\cap\Oj$ which vanishes continuously on $B(0,4R)\cap\partial\Oj$.

On the other hand, $0\in\partial\Oj$ and, using our convention \eqref{nta-pt}, $ \wAj(0,1)$ is a corkscrew point relative to $B(0,1)\cap\partial\Oj$ in the domain $\Oj$.  
Thus, by Lemma \ref{lm:CFMS} 
			\begin{equation}\label{eq:pjzj}
				\uj(\wAj(0,1)) \sim 1.  
			\end{equation} 
We can then invoke Lemma \ref{harn-princ}, the fact that $\wAj(0,2R)\in\Oj$ is a corkscrew point relative to  $B(0,2R)\cap\partial\Omega_j$ for the domain $\Oj$,  Harnack's inequality, and \eqref{eq:pjzj} to obtain 
			\begin{equation} \label{uj-loc-bdd}
			\sup_{Z\in \Oj\cap B(0,2R)} \uj(Z) 
			\leq C \uj(\wAj(0,2R))
			\le C_R \uj(\wAj(0,1)) 
			\leq 
			C_R.  
			\end{equation}
			Extending $\uj$ by 0 outside of $\Oj$ we conclude that the sequence $\{\uj\}_{j\geq j_0}$ is uniformly bounded in $\overline{B(0,R)}$ for some $j_0$ large enough. Since for each $j$, 
$\wcalA_{j}$ has ellipticity constants bounded below by $\lambda=1$ and above by $\Lambda$, and $\Oj$ is uniform and satisfies the CDC (as $\pOj$ is Ahlfors regular) with the same constants as $\Omega_j$, then combining Lemma \ref{lem:vanishing} with the DeGiorgi-Nash-Moser estimates we conclude that the sequence $\{\uj\}_j$ is equicontinuous on $\overline{B(0,R)}$ (in fact uniformly H\"older continuous with same exponent). 
			Using Arzela-Ascoli combined with a diagonalization argument applied on a sequence of balls with radii going to infinity, we produce $u_\infty\in C(\R^n)$ and a subsequence (which we relabel) such that $\uj \to u_\infty$ uniformly on compact sets of $\RR^n$.

As observed before,  $\uj$ is a non-negative $L$-solution on $B(0,4R)\cap\Oj$ which vanishes continuously on $B(0,4R)\cap\partial\Oj$ and which has been extended by 0 outside of $\Oj$. Thus it is a positive $L$-subsolution on $B(0,4R)$ and we can use Caccioppoli's inequality along with \eqref{uj-loc-bdd} to conclude that
\begin{equation}\label{eqn:3.1A}
\int_{B(0,R)}|\nabla \uj|^2\, dZ 
\le
C\,R^{-2} \int_{B(0,2R)}|\uj|^2\, dZ
\le
C_R. 
\end{equation}
This and \eqref{uj-loc-bdd} allow us to conclude that
\begin{equation}\label{eqn:3.4}
	\sup_j \|\uj\|_{W^{1,2}(B(0,R))} \leq C_R <\infty.
\end{equation}		
Thus, there exists a subsequence (which we relabel) which converges weakly in $W^{1,2}_{\rm loc}(\RR^n)$. Since we already know that $\uj \to u_\infty$ uniformly on compact sets of $\RR^n$, we can use again \eqref{uj-loc-bdd} to easily see that $u_\infty\in W^{1,2}_{\rm loc}(\RR^n)$, 
 and $\nabla \uj \rightharpoonup \nabla u_{\infty} $ in $L^2_{\rm loc}(\RR^n)$. This completes the proof of \ref{1-thm:pseudo-blow-geo} in Theorem \ref{thm:pseudo-blow-geo}.	
\end{proof}

\begin{proof}[Proof of \ref{2-thm:pseudo-blow-geo} in Theorem \ref{thm:pseudo-blow-geo}]
Recall that in this case $\wXj\to X_0$ as $j\to\infty$. For any $0<\rho\le c_1R_0/2$ and for all $j$ large enough we have
\begin{equation}\label{poleawaycase2}
	B\left(\wXj,\frac{\rho}{2} \right) 
	\subset 
	B(X_0,\rho) 
	\subset 
	B(\wXj,2\rho)
	\subset 
	\overline{B(\wXj,2\rho)} 
	\subset 
	\overline{B(\wXj,\wdj(\wXj)/2)}
	\subset \Oj, 
\end{equation} 
where we have used \eqref{eq:wXjcase2}. Moreover, for $j$ sufficiently large,
\begin{equation}\label{eq:temp3}
	\dist(B(\wXj,2\rho),\pOj) > \frac{c_1 R_0}{2}.
\end{equation}
For any $Z\in \Oj \setminus B(\wXj,\rho/4)$, using \eqref{def:ujcaseii} and \eqref{eqn:gub} it follows that
\begin{equation}\label{eq:ujboundcase2}
	\uj(Z) \leq \frac{C}{|Z-\wXj|^{n-2}} \leq \frac{4^{n-2}C}{\rho^{n-2}}.
\end{equation}
Extending $\uj$ by $0$ outside $\Oj$ the previous estimate clearly holds for every $Z\in \R^n\setminus\Oj$. Thus $\sup_j\|\uj\|_{L^\infty(\R^n \setminus B(X_0,\rho))}\le C(\rho)$. Moreover, as in \ref{CaseI}, the sequence is also equicontinuous (in fact uniformly H\"older continuous). Using Arzela-Ascoli theorem with a diagonalization argument, we can find $u_\infty \in C(\R^n\setminus \{X_0\})$ and a subsequence (which we relabel) such that $\uj \to u_{\infty}$ uniformly on compact sets of $\R^n\setminus\{X_0\}$. 

Let $0<R\leq \sup_{j\gg 1} \diam(\Oj) \sim R_0$. We claim that
\begin{equation}\label{eq:L2gradcase2}
	\int_{B(0,R) \setminus B(X_0,\rho)} |\nabla \uj|^2 dZ \leq C(R,\rho) <\infty.
\end{equation}

To prove this, we first  take arbitrary $q\in\pOj$ and $s$ such that $0<s\le  \wdj(\wXj)/5 \sim R_0$. In particular, if $0<\rho<c_1 R_0/10\leq \wdj(\wXj)/10$ it follows that $B(q,4s)\subset\R^n\setminus B(\wXj,2 \rho)\subset 
\R^n\setminus B(X_0,\rho)$. Thus, proceeding as in \ref{CaseI}, $\uj$ is non-negative subsolution on $B(q, 2s)$ and we can use Caccioppoli's inequality and \eqref{eq:ujboundcase2} to obtain 
\begin{align}\label{eq:L2gradbd}
	\int_{B(q,s)\setminus B(X_0,\rho)} |\nabla \uj|^2 dZ 
	=
	\int_{B(q,s)} |\nabla \uj|^2 dZ 
		\leq \frac{C}{s^2} \int_{B(q,2s )} |\uj(Z)|^2 dZ \lesssim \frac{s^{n-2}}{\rho^{2(n-2)}}.
\end{align}

Note that the previous estimate, with $q=0$ and $s=R$, gives our claim \eqref{eq:L2gradcase2} when $0<R\leq \wdj(\wXj)/5$. 

Consider next the case $R_0\sim \wdj(\wXj)/5 < R\le \sup_{j\gg 1} \diam(\Oj)\sim R_0$.  Note first that the set $\Theta_j:=\{Z\in \Oj: \wdj(Z) < \wdj(\wXj)/25\}$ can be covered by a family of balls $\{B(q_i,\wdj(\wXj)/5)\}_i$ with $q_i\in\partial\Omega$ and whose cardinality is uniformly bounded (here we recall that $\wdj(\wXj)\sim  \diam(\Oj)$), Thus, \eqref{eq:L2gradbd} applied to these each ball in the family yields
\begin{equation}\label{eq:temp4}
	\int_{\left( B(0,R) \setminus B(X_0,\rho) \right) \cap \Theta_j } |\nabla \uj|^2 dZ \leq 
\sum_i
	\int_{B(q_i,\wdj(\wXj)/5) \setminus B(X_0,\rho)} |\nabla \uj|^2 dZ \leq 
	C(R, \rho) <\infty.
\end{equation} 
On the other hand, the set $\{Z\in \Oj\setminus B(\wXj,\rho/2): \wdj(Z) \geq \wdj(\wXj)/25\}$ can be covered by a family of balls $\{B_i\}_i$ so that $r_{B_i}=\rho/16$, $4B_i \subset\Oj\setminus 
B(\wXj,\rho/4)$. Moreover, the cardinality of the family is uniformly bounded depending on dimension and the ratio $\diam(\Oj)/\rho\sim R_0/\rho$. Using \eqref{poleawaycase2}, Caccioppoli's inequality in each $B_i$ since $4B_i \subset\Oj\setminus B(\wXj,\rho/4)$, and \eqref{eq:ujboundcase2} we obtain 
\begin{equation}\label{eq:temp5}
	\int_{\left( B(0,R) \setminus B(X_0,\rho) \right) \setminus \Theta_j } |\nabla \uj|^2 dZ 
	\leq 
	\sum_i\int_{B_i} |\nabla \uj|^2 dZ 
	\lesssim 
	\sum_i\frac{1}{r_{B_i}^2} \int_{2B_i} | \uj(Z)|^2 dZ 
	\leq C(R,\rho).
\end{equation}
Combining \eqref{eq:temp4} and \eqref{eq:temp5} we obtain the desired estimate and hence proof of the claim \eqref{eq:L2gradcase2} is complete.

Next, we combine \eqref{eq:L2gradcase2} with the fact that $\sup_j\|\uj\|_{L^\infty(\R^n \setminus B(X_0,\rho))}\le C(\rho)$ to obtain 	that $\sup_j \|\uj\|_{W^{1,2}(B(0,R)\setminus B(X_0,\rho))} \leq C(R,\rho) <\infty$. Thus, there exists a subsequence (which we relabel) which converges weakly in $W^{1,2}_{\rm loc}(\RR^n\setminus B(X_0,\rho))$. Since we already know that $\uj \to u_\infty$ uniformly on compact sets of $\RR^n\setminus B(X_0,\rho)$, we can easily see that $u_\infty\in W^{1,2}_{\rm loc}(\RR^n\setminus B(X_0,\rho))$,  and $\nabla \uj \rightharpoonup \nabla u_{\infty} $ in $L^2_{\rm loc}(\RR^n\setminus B(X_0,\rho))$. This completes the proof of \ref{2-thm:pseudo-blow-geo} in Theorem \ref{thm:pseudo-blow-geo}.	
\end{proof}

\medskip

\begin{remark}\label{rem:deawffr}

In the \ref{CaseII} scenario the following remarks will become useful later. In what follows we assume that  
$0<\rho\le c_1R_0/2$ and $j$ is large enough.
\begin{enumerate}[label=\textup{(\roman*)}, itemsep=0.2cm]  
	
	\item\label{1-rem:deawffr} Let us pick $Y\in \partial B(\wXj,3\delta_j(X_j)/4)$ and note that \eqref{eq:wXjcase2} gives $Y, \wAj(0,c_1 R_0/2) \in \Oj \setminus \overline{B(\wXj,\delta_j(X_j)/2)}$, $|Y-\wAj(0,c_1 R_0/2)|<(c_1+2c_2)R_0$, and $\delta_j(Y)\ge c_1 R_0/4$. Recalling that $\Omega_j$ satisfies the interior corkscrew condition with constant $M$, it follows by definition that $\delta_j(\wAj(0,c_1 R_0/2))\ge c_1 R_0/(2M)$.  All these allow us to invoke Lemma \ref{lemm:NC-avoid} to then use \eqref{eqn:glb} and \eqref{eq:wXjcase2} and eventually show 
		\begin{equation}\label{eq:ujfixedvalue}
			\uj \left( \wAj\left(0,\frac{c_1 R_0}{2} \right) \right)
			\sim \uj(Y)
			\gtrsim \left| Y - \wXj \right|^{2-n} \sim \delta_j(X_j)^{2-n}\sim R_0^{2-n},
		\end{equation}
		where the implicit constants are independent of $j$.
		
		\item\label{2-rem:deawffr} The set $\partial B(X_0,\rho)$ is compact and away from $X_0$, so $\uj \to u_\infty$ uniformly in $\partial B(X_0,\rho)$. Since $\wXj \to X_0$, for any $Z\in \partial B(X_0,\rho)$ we have
			$\rho/2< |Z-\wXj| < 2\rho$ for $j$ sufficiently large. In particular by choosing $\rho<R_0/(16M)$, we have for $j$ large enough
			\begin{equation}
				|Z-\wXj| < 2\rho < \frac{R_0}{8M} \leq \frac{\diam(\Oj)}{4M} \leq \frac{\wdj(\wXj)}{2}, 
			\end{equation}
			where the last estimate uses that $X_j\in \Omega_j$ is a corkscrew point relative to the surface ball  $B(0,\diam(\Oj)/2)\cap\partial\Omega_j$ with constant $M$. Thus by \eqref{eqn:glb} if $j$ is large enough
\[ 
		\uj(Z) \gtrsim |Z-\wXj|^{2-n} \gtrsim \rho^{2-n}, \qquad  \forall\,Z\in \partial B(X_0,\rho) 
\]
			with implicit constants which are independent of $j$. Therefore,
			\begin{equation}\label{eq:uinftynearpole}
				u_\infty(Z) = \lim_{j\to\infty} \uj(Z) \gtrsim \rho^{2-n},
								\qquad  \forall\,Z\in \partial B(X_0,\rho) 
			\end{equation}
			For this reason it is natural to extend the definition of $u_\infty$ to all of $\R^n$ by simply letting $u_\infty(X_0) =+ \infty$.
			
			\item\label{3-rem:deawffr} Since $\uj$ is the Green function in $\Oj$ for $L_j$, an 
			elliptic operator with uniform ellipticity constants $\lambda =1$ and $\Lambda$, 
			by \eqref{eqn:glq} we know for any $1<r<\frac{n}{n-1}$,
				\begin{equation}\label{eq:upperr}
					\|\nabla \uj\|_{L^r(\Oj)} \lesssim |\Oj|^{\frac1r-\frac{n-1}n}\lesssim 
					R_0^{\frac{n}r-n+1}<\infty,
				\end{equation}
				provided $j$ is large enough and where the implicit constants depend on 
				dimension, $r$, and $\Lambda$, but are independent of $j$. 
				Note that $\nabla \uj \equiv 0$ outside of $\Oj$ by construction. Thus, one can easily show that passing to a subsequence (and relabeling) $\nabla \uj \rightharpoonup \nabla u_{\infty}$ in $L^r_{\rm loc}(\RR^n)$ for $1<r<n/(n-1)$.
			
	\end{enumerate}
\end{remark}

\medskip

\begin{proof}[Proof of \ref{3-thm:pseudo-blow-geo} in Theorem \ref{thm:pseudo-blow-geo}: \ref{CaseI}]

	It is clear that $\Oinf$ is an open set in \ref{CaseI} since $u\in C^\infty(\R^n)$.
	On the other hand, 	since $0\in\pOj$ for all $j$, by Lemma \ref{lm:cptHd} and modulo passing to a subsequence (which we relabel) we have that there exist non-empty closed sets $\Gamma_{\infty}, \Lambda_{\infty}$ such that
			$\overline{\Oj}\to \Gamma_\infty$ and $\pOj\to \Lambda_\infty$ as $j\to\infty$, where the convergence is in the sense of Definition \ref{def:cvsets}.
			
We are left with obtaining 
\begin{equation}\label{eq:claim-boundary--}
\Lambda_\infty = \pOinf
\qquad\mbox{and}\qquad\Gamma_\infty=\overline\Omega_\infty.
\end{equation}		
We first show that $\Lambda_\infty\subset \partial\Omega_\infty$. To that end we take $p\in\Lambda_\infty$, and there is a sequence $p_j\in\pOj$ such that $\lim_{j\to\infty} p_j = p$. Note that
			$ u_\infty(p) =\lim_{j\to \infty}\ \uj(p)$. On the other hand since the $\uj$'s are uniformly H\"older continuous on compact sets (see the Proof of  \ref{1-thm:pseudo-blow-geo} in Theorem \ref{thm:pseudo-blow-geo}) and $\uj(p_j)=0$ as $p_j\in\partial\Oj$ we have
			$$
			0\le u_\infty(p)
			\le
			|u_\infty(p)-u_j(p)|+
			|u_j(p)-u_j(p_j)|
			\lesssim
			|u_\infty(p)-u_j(p)|+|p-p_j|^\alpha\to 0,
			$$
as $j\to\infty$. Thus $u_{\infty}(p) = 0$, that is, $p\in \R^n\setminus \Omega_\infty$. 

Our goal is to show that $p\in\partial\Omega_\infty$. Suppose that $p\notin\partial\Omega_\infty$, then $p\in\R^n\setminus\overline{\Omega_\infty}$ and there exists $\epsilon\in(0,1)$ such that $B(p,\epsilon) \subset \R^n\setminus\overline{\Omega_\infty}$, that is, $u_\infty\equiv 0$ on $\overline{B(p,\epsilon)}$. In $\Oj$ we have
			$$
				\left|A_j\left(p_j,\frac{\epsilon}{2}\right) - A_j(0,1)\right|  \le \frac{\epsilon}{2} + | p_j | + 1 
				 \le 2\left(|p| + 1\right)
			$$
			and
			$$
			 \delta_j\left( A_j\left(p_j,\frac{\epsilon}{2}\right) \right) \geq \frac{1}{M} \frac{\epsilon }{2}, 
			\qquad \delta_j\left( A_j(0,1)\right) \geq \frac{1}{M}. 
			 $$
			Note also that
			$$
			\frac{\delta_j\left( A_j\left( p_j,\frac{\epsilon }{2} \right)\right) }{\delta_j(X_j)}
			+
			\frac{\delta_j\left( A_j(0,1)\right) }{\delta_j(X_j)}
			\sim
			\frac{1}{\diam(\Omega_j)}\to 0,
			\qquad\text{as }j \to\infty,
			$$
			hence for $j$ large enough, $A_j(0,1), A_j\left(p_j,\frac{\epsilon }{2} \right) \notin \overline{B(X_j,\delta_j(X_j)/2})$. 
						
						We can then apply Lemma \ref{lemm:NC-avoid} and Harnack's inequality along the constructed chain in $\Omega_j$ to obtain 
			$$
			G_j\left( X_j,A_j\left( p_j,\frac{\epsilon }{2} \right) \right) \sim G_j(X_j,A_j(0,1)),
			$$
			where the implicit constants depend on the allowable parameters, $\epsilon$ and $|p|$, but are uniform on $j$. Hence by \eqref{eq:pjzj},
\begin{equation}\label{eq:tmplb}
				\uj\left(\wAj\left(p_j,\frac{\epsilon}{2}\right)\right)  = \dfrac{ G_j \left( X_j,A_j\left( p_j,\frac{\epsilon }{2} \right)\right)}{\omega_j^{X_j}(B(0,1))} 
				\gtrsim C \dfrac{ G_j\left( X_j,A_j(0,1) \right)}{\omega_j^{X_j}(B(0,1))}  
				 = \uj(\wAj(0,1))
\ge C_0,
				\end{equation}
where $C_0$ is independent of $j$.  

Note that since $\uj\to u_\infty$ on compact sets it follows from our assumption that for $j$ large enough  depending on $C_0$ 
\begin{equation}\label{eq:contrad}
u_j(Z)=u_j(Z)-u_\infty(Z)<\frac{C_0}2,
\qquad
\forall\,Z\in \overline{B(p,\epsilon)}.
\end{equation}
However, for $j$ large enough $\wAj (p_j,\epsilon/2) \in B(p_j,\epsilon/2)  \subset B(p,\epsilon)$  and then \eqref{eq:contrad} contradicts \eqref{eq:tmplb}. Thus, we have shown that necessarily $p\in\partial\Omega_\infty$ and consequently $\Lambda_\infty\subset \partial\Omega_\infty$.

Let us next show that $\pOinf \subset \Lambda_{\infty}$. Assume that $p \notin \Lambda_{\infty}$. Since $\Lambda_{\infty}$ is a closed set, there exists $\epsilon>0$ such that $B(p,2\epsilon)\cap\Lambda_{\infty} = \emptyset$. Since $\Lambda_{\infty}$ is the limit of $\pOj$, by Definition \ref{def:cvsets} we have that for $j$ large enough $B(p,\epsilon) \cap \pOj=\emptyset$. Hence, by passing to a subsequence (and relabeling) either $B(p,\epsilon) \subset \Oj$ for all $j$ large enough or $B(p,\epsilon) \subset \R^n\setminus\overline{\Oj}$ for all $j$ large enough. 

We first consider the case $B(p,\epsilon) \subset \Oj$. Hence,  $\delta_j(p)  \geq \epsilon $ and $|A_j(0,1) - p| \leq 1+|p|$. Thus there exists a Harnack chain joining $A_j(0,1)$ and $p$ whose length is independent of $j$ and depends on $\epsilon$ and $|p|$. 
We next observe that for $j$ large enough $|p-\wXj|>\wdj(\wXj)/2$. Indeed, if we take $j$ large enough, using that $0\in\partial\Oj$ and \eqref{eq:temp1} we clearly have
$$
1\le \frac{|\wXj|}{\wdj(\wXj)}
\le
\frac{|\wXj-p|}{\wdj(\wXj)}
+
\frac{|p|}{\wdj(\wXj)}
<
\frac{|\wXj-p|}{\wdj(\wXj)}+\frac12,
$$
and we just need to hide to obtain the desired estimate. 
Once we know that $|p-\wXj|>\wdj(\wXj)/2$, we also note that $|\delta_j(A_j(0,1))|\le 1\ll \diam(\Omega_j)\sim\delta_j(X_j)$
and hence $A_j(0,1)\notin\overline{B(X_j,\delta(X_j)/2})$ for $j$ large enough. 

We can now invoke Lemma \ref{lemm:NC-avoid} and Harnack's inequality along the constructed chain in $\Omega_j$ to obtain that $G_j(X_j,p) \sim G_j(X_j,A_j(0,1))$, which combined with \eqref{def:ujcasei} and \eqref{eq:pjzj}, yields
\begin{equation}\label{eqn:3.5}
\uj(p) \sim \uj(\wAj(0,1)) \sim 1,
\end{equation}
where the implicit constants depend on the allowable parameters, $p$ and $\epsilon$, but are uniform on $j$. Letting $j\to\infty$ we obtain that $u_\infty(p)\sim 1$ which implies that $p\in\Omega_\infty$, and since we have already shown that  $\Omega_\infty$ is open, it follows that $p\notin\partial \Omega_\infty$.
			
We next consider now the case 	$B(p,\epsilon) \subset \R^n\setminus\overline{\Oj}$ for all $j$ large enough which implies that by construction $\uj(X) = 0$ for all $X\in B(p,\epsilon)$. By uniform convergence of $\uj$ in compact sets we have that $u_{\infty}(X) = 0$ for $X\in B(p,\epsilon/2)$, which implies $B(p,\epsilon/2) \subset \{u_{\infty}=0\}$ and therefore $p\notin \partial\Oinf$.

In both cases we have shown that if $p \notin \Lambda_{\infty}$ then  $p\notin \partial\Oinf$, or, equivalently, $\partial\Oinf\subset \Lambda_{\infty}$.  This together with the converse inclusion completes the proof of $\Lambda_{\infty} = \partial\Oinf$.

Our next goal is to show that $\Gamma_\infty=\overline\Omega_\infty$. Note that if $Z\in\Omega_\infty$, then $u_\infty(Z)>0$ and this implies that $\uj(Z)>0$ for $j$ large enough. The latter forces $Z\in\Omega_j$ 	for all $j$ large enough. This implies that $Z\in \Gamma_\infty$, and we have shown that $\Omega_\infty\subset \Gamma_\infty$. Moreover since $\Gamma_{\infty}$ is closed, we conclude that $\overline\Omega_\infty\subset \Gamma_\infty$. 

To obtain the converse inclusion we take $X\in \Gamma_\infty$. Assume that there is $\epsilon>0$ such that $\overline{B(X,2\epsilon)}\subset \R^n\setminus\overline{\Omega_\infty}$, in particular $B(X,2\epsilon) \cap \partial\Oinf = \emptyset$. Since we have already shown that $\partial\Oinf$ is the limit of $\pOj$'s, for $j$ large enough $B(X,\epsilon) \cap \pOj = \emptyset$. By the definition of $\Gamma_{\infty}$, there is a sequence $\{Y_j\}\subset \overline\Oj $ with $Y_j\to X$ as $j\to\infty$. Thus, for all $j$ large enough $B(X,\epsilon)$ is a neighborhood of $Y_j$; and in particular $\Oj\cap B(X,\epsilon)\not =\emptyset$ since $Y_j \in \overline{\Oj}$. On the other hand, since $B(X,\epsilon) \cap \pOj = \emptyset$
			we conclude that $B(X,\epsilon)\subset \Oj$. At this point we follow a similar argument to the one used to obtain \eqref{eqn:3.5} replacing $p$ by $X$ and obtain for all $j$ large enough
			$$
				\uj(X) \sim \uj(\wAj(0,1)) \sim 1,
			$$
			where the implicit constants depend on the allowable parameters, $|X|$ and $\epsilon$, but are uniform on $j$. Letting  $j\to\infty$ it follows that $u_{\infty}(X)>0$ and hence $X\in\Oinf$, contradicting the assumption that there is $\epsilon>0$ such that $\overline{B(X,2\epsilon)}\subset \R^n\setminus\overline{\Omega_\infty}$. In sort, we have shown that $\overline{B(X,2\epsilon)}\cap  \overline{\Omega_\infty}\neq\emptyset$ for every $\epsilon>0$, that is, $X\in  \overline{\Omega_\infty}$. We have eventually proved that  $\Gamma_\infty\subset \overline\Omega_\infty$ this completes the proof of  \eqref{eq:claim-boundary--} in the \ref{CaseI} scenario.

			Since $\diam(\Omega_j) \to \infty$ and $0\in \overline{\Oj} \to \overline{\Oinf}$ uniformly on compact set, $\Oinf$ is unbounded. Otherwise we would have $\overline{\Oinf} \subset B(0,R)$, and for sufficiently large $j$ one would see that $\overline{\Oj} \subset B(0,2R)$, which is a contradiction.

On the other hand, it is possible that $\diam(\pOj) \not\to \diam(\pOinf)$, hence we do not know whether $\diam(\pOinf)=\infty$. However, under the assumption that the $\pOj$'s are Ahlfors regular with uniform constant, we claim that $\pOinf$ is also unbounded. Assume not, then there is $R>0$ such that $\pOinf \subset B(0,R)$. Let $k$ be a large integer, and notice that $\pOj \to \pOinf$ uniformly on the compact set $\overline{B(0,kR)}$. Thus for $j$ sufficiently large (depending on $k$) 
			\begin{equation}\label{eq:nobdbd}
				\pOj \cap \overline{B(0,kR)} \subset B(0,2R).
			\end{equation}  
			Since $\diam(\pOj)\to \infty$ we can also guarantee that $\diam(\pOj) > kR $ for $j$ sufficiently large. Recalling that $0\in \pOj$, we can then consider the surface ball $\Delta_j(0,kR) = B(0,kR)\cap \pOj$. By \eqref{eq:nobdbd} and the Ahlfors regularity of $\pOj$,
			\begin{equation}
				C_{AR}^{-1}(kR)^{n-1}\leq \sj(\Delta_j(0,kR)) \le  \sj(B(0,2R)\cap\partial\Omega_j) \leq C_{AR} (2R)^{n-1}.
			\end{equation}
			Letting $k$ large readily leads to a contradiction. 
			\end{proof}
			
\begin{proof}[Proof of \ref{3-thm:pseudo-blow-geo} in Theorem \ref{thm:pseudo-blow-geo}: \ref{CaseII}]
Take $X\in \Omega_\infty$, that is, $u_\infty(X)>0$. If $X\neq X_0$ then $u_\infty$ is continuous at $X$ and hence $u_\infty(Z)>0$ for every $Z\in B(X,r_x)$ for some $r_x$ small enough. On the other hand, if $X=X_0$, by Remark \ref{rem:deawffr} part \ref{2-rem:deawffr} we have that $u_\infty(Z)>0$ for all $Z\in B(X_0, \rho)$ with $\rho$ sufficiently small (here we use the convention that $+\infty>0$). Note that this argument show in particular that $B(X_0,\rho)\subset \Omega_\infty$.

On the other hand, 	since $0\in\pOj$ for all $j$, by Lemma \ref{lm:cptHd} and modulo passing to a subsequence (which we relabel), there exist closed sets $\Gamma_{\infty}, \Lambda_{\infty}$ such that
			$\overline\Oj\to \Gamma_\infty$ and $\pOj\to \Lambda_\infty$ as $j\to\infty$ in the sense of Definition \ref{def:cvsets}. We are going to obtain that 
\begin{equation}\label{eq:claim-boundary--CaseII}
\Lambda_\infty = \pOinf
\qquad\mbox{and}\qquad\Gamma_\infty=\overline\Omega_\infty.
\end{equation}

Let $p\in\Lambda_\infty$, there is a sequence $\{p_j\}\subset \pOj$ such that $p_j\to p$ as $j\to\infty$. Note that by \eqref{eq:wXjcase2}
$$
c_1 R_0
\le 
\wdj(\wXj) 
\le 
|\wXj-p_j|
\le
|\wXj-p|+|p-p_j|.
$$
Thus, for $j$ large enough, $|\wXj-p|> \wdj(\wXj) /2>c_1R_0/2$. In particular, $X_0\neq p$
and $ \uj(p)\to u_\infty(p)$ as $j\to\infty$. 
On the other hand since the $\uj$'s are uniformly H\"older continuous on compact sets as observed above, $|\uj(p)|=|\uj(p)-\uj(p_j)| \leq C|p-p_j|^\alpha$, thus $\uj(p)\to 0$ as $j\to\infty$. Therefore $u_{\infty}(p) = 0$, that is, $p\in \R^n\setminus\Omega_\infty$. 

Suppose now that $p\notin\partial\Omega_\infty$. Then, there exists $0<\epsilon<\wdj(\wXj) /4$ such that $\overline{B(p,\epsilon)} \subset \R^n\setminus\Oinf$, or, equivalently, $u_\infty\equiv 0$ on $B(p,\epsilon)$. Note that 
			$$
				\left|\wAj\left( p_j,\frac{\epsilon}{2}\right) - \wAj\left(0,\frac{c_1 R_0}{2} \right)\right|  \leq \frac{\epsilon}{2} + |p_j | + \frac{c_1 R_0}{2} \leq C(\epsilon, |p|, R_0).
			$$
			Also, 
			$$
			\frac{\epsilon }{2M} 
			\le
			\wdj\left( \wAj\left( p_j,\frac{\epsilon}{2}\right) \right) 
			<
			\frac{\epsilon}{2}
			<
			\frac{\wdj(\wXj)}{2}
			 $$
and, by  \eqref{eq:wXjcase2},
\begin{equation}\label{eq:fwefgaserg}
\frac{c_1 R_0}{2M}\le \wdj\left( \wAj\left(0,\frac{c_1 R_0}{2} \right)\right)<\frac{c_1 R_0}{2}
\le
\frac{\wdj(\wXj)}{2}.
\end{equation}

Notice that in particular $\wAj\left( p_j,\frac{\epsilon}{2}\right)$, $\wAj\left(0,\frac{c_1 R_0}{2} \right) \notin\overline{B(\wXj,\wdj(\wXj)/2)}$. We can now invoke Lemma \ref{lemm:NC-avoid},  Harnack's inequality along the constructed chain in $\Oj$, and \eqref{eq:ujfixedvalue} to see that
			\begin{equation} \label{eqn103case2}
			\uj \left( \wAj\left( p_j,\frac{\epsilon }{2} \right) \right) 
			\sim  
			\uj\left(\wAj\left(0,\frac{c_1 R_0}{2}\right) \right) 
			\gtrsim 
			1, 
			\end{equation}
			with implicit constant depending on the allowable parameters, $\epsilon, |p|, R_0$ but independent of $j$. On the other hand, for all $j$ large enough
			\begin{equation} \label{eqn104case2}
			\wAj \left(p_j,\frac{\epsilon}{2}\right) 
			\in B\left(p_j,\frac{\epsilon}{2}\right) 
			\subset 
			\overline{B(p,\epsilon)}
			\subset
			\R^n\setminus B(\wXj,\wdj(\wXj)/4),
						 \end{equation}
hence $u_j\to u_\infty$ uniformly on $\overline{B(p,\epsilon)}$ with $u_\infty\equiv 0$ on $\overline{B(p,\epsilon)}$. This and \eqref{eqn104case2} contradict \eqref{eqn103case2} and therefore we conclude that $p\in\pOinf$, and we have eventually obtained that $\Lambda_\infty \subset \pOinf$.
			
			To show that $\pOinf \subset \Lambda_{\infty}$, we assume that $p \notin \Lambda_{\infty}$. If $p=X_0$, then since we observed above that $B(X_0,\rho) \subset \Oinf$ (see \eqref{eq:uinftynearpole}) then $X_0 \notin \pOinf$. 
			
			Assume next that $p\neq X_0$. Since $\Lambda_{\infty}$ is a closed set and since $\wXj\to X_0$ as $j\to\infty$, there exists $\epsilon>0$ such that $B(p,2\epsilon)\cap\Lambda_{\infty} = \emptyset$ and $X_0, \wXj \notin B(p, 2 \epsilon)$ for all $j$ large enough. Moreover, since $\Lambda_{\infty}$ is the limit of $\pOj$, by Definition \ref{def:cvsets} we have that for all $j$ large enough $B(p,\epsilon) \cap \pOj=\emptyset$.  Hence, passing to a subsequence (and relabeling) either $B(p,\epsilon) \subset \Oj$ for $j$ large enough or $B(p,\epsilon) \subset \R^n\setminus\overline{\Oj}$ for $j$ large enough. 
			
			Assume first that $B(p,\epsilon) \subset \Oj$ for all $j$ large enough. We consider two subcases. Assume first that $p\notin \overline{B(\wXj,\wdj(\wXj)/2}$. Then, proceeding as before, by \eqref{eq:fwefgaserg} we can apply Lemma \ref{lemm:NC-avoid} and Harnack's inequality along the constructed chain in $\Oj$ to get 
			\begin{equation}\label{eqn:3.5case2}
				\uj(p) \sim \uj\left(\wAj\left(0,\frac{c_1 R_0}{2} \right) \right) \gtrsim 1,
				\end{equation} 
				with implicit constant depending on the allowable parameters, $\epsilon, |p|, R_0$ but independent of $j$. Suppose next that $p\in \overline{B(\wXj,\wdj(\wXj)/2)}$. In that case we can use \eqref{def:ujcaseii}, \eqref{eqn:glb}, and   \eqref{eq:wXjcase2}  to see that for all $j$  large enough
			\begin{equation}\label{eqn:3.5case2-a}
				\uj(p) \gtrsim |p-\wXj|^{2-n} \gtrsim \wdj(\wXj)^{2-n}\gtrsim (c_2 R_0)^{2-n},			
												\end{equation} 
				with implicit constants which are uniform on $j$. Combining the two cases together we have shown that $\uj(p) \gtrsim 1$ uniformly on $j$. Letting $j\to\infty$ we conclude that $
			u_\infty(p)\gtrsim 1$ and hence $p\in \Omega_\infty$, and since we have already shown that $\Omega_\infty$ is an open set we conclude that $p\notin \partial\Omega_\infty$

We now tackle the second case on which $B(p,\epsilon) \subset \R^n\setminus\overline{\Oj}$ for all $j$ large enough. In this scenario $\uj(X) = 0$ for all $X\in B(p,\epsilon)$. Since $X_0\notin B(p,2\epsilon)$, by uniform convergence of $\uj$ in $\overline{B(p,\epsilon/2)}$ we have that $u_{\infty}(X) = 0$ for $X\in B(p,\epsilon/2)$, which implies $B(p,\epsilon/2) \subset \R^n\setminus\Omega_\infty$ and eventually $p\notin \partial\Oinf$.

In both cases we have shown that if $p \notin \Lambda_{\infty}$ then  $p\notin \partial\Oinf$, or, equivalently, $\partial\Oinf\subset \Lambda_{\infty}$.  This together with the converse inclusion completes the proof of $\Lambda_{\infty} = \partial\Oinf$.

Our next task is to show that $\Gamma_\infty=\overline\Omega_\infty$. 
Let $Z\in\Oinf$ and assume first that $Z=X_0$. By \eqref{poleawaycase2} and since $\wXj\to X_0$ as $j\to\infty$ we have that $X_0 \in B(\wXj,2\rho) \subset \Oj$ for all $j$ large enough, thus $Z=X_0\in \Gamma_\infty$. On the other hand, if $Z\neq X_0$ since $u_\infty(Z)>0$ we have that $\uj(Z)>0$ for all $j$ large enough. This forces as well that $Z\in\Oj$ for $j$ all large enough and again $Z\in \Gamma_\infty$. With this  we have shown that $\Omega_\infty\subset \Gamma_\infty$. Moreover, since $\Gamma_{\infty}$ is closed we conclude as well that $\overline\Omega_\infty\subset \Gamma_\infty$. 

Next we look at the converse inclusion and take $X\in \Gamma_\infty$. Assume that $X\in \R^n\setminus\overline{\Omega_\infty}$. Thus, there is $\epsilon>0$ such that $\overline{B(X,2\epsilon)}\subset \R^n\setminus\overline{\Omega_\infty}$. In particular $B(X,2\epsilon) \cap \partial\Oinf = \emptyset$ and $B(X_0,\rho) \cap B(X,2\epsilon) = \emptyset$ (recall that we showed that $B(X_0,\rho)\subset \Oinf$).  Since we have already shown that $\partial\Oinf$ is the limit of $\pOj$'s, for $j$ large enough $B(X,\epsilon) \cap \pOj = \emptyset$. By the definition of $\Gamma_{\infty}$, there is a sequence $\{Y_j\}\subset \overline\Oj$ so that $Y_j\to X$ as $j\to\infty$. Thus, for all $j$ large enough $B(X,\epsilon)$ is a neighborhood of $Y_j$, and, in particular, $\Oj\cap B(X,\epsilon)\not =\emptyset$ since $Y_j \in \overline{\Oj}$. Besides, since $B(X,\epsilon) \cap \pOj = \emptyset$
			we conclude that $B(X,\epsilon)\subset \Oj$. Using a similar argument to the one used to obtain \eqref{eqn:3.5case2} and \eqref{eqn:3.5case2-a} we have (replacing $p$ by $X$) that
			$$
				\uj(X) \gtrsim 1
			$$
			independently of $j$ and with constants that depend on the allowable parameters, $\epsilon, |X|, R_0$. Since $u_j(X)\to u_\infty(X)$ we conclude that $u_{\infty}(X)>0$ and thus $X\in\Oinf$, contradicting the assumption that $X\in \R^n\setminus\overline{\Omega_\infty}$. Eventually, $X\in\overline\Omega_\infty$ and we have obtained that $\Gamma_\infty\subset \overline\Omega_\infty$.

			Since $\diam(\Oj) \to R_0$ is finite and $0\in \pOj$, we have $\Oj, \Oinf \subset \overline{B(0,2R_0)}$ for $j$ sufficiently large. Hence $\overline{\Oj} \to \overline{\Oinf}$ uniformly, and thus $\diam(\Oinf) = \lim\limits_{j\to\infty} \diam(\Oj) = R_0 \geq 1$. 
\end{proof}

For later use let us remark that in the \ref{CaseII} scenario the fact that 
			$\overline{\Oj}\to\overline{\Omega_\infty}$ and $\pOj \to \pOinf$ as $j\to\infty$ in the sense of Definition \ref{def:cvsets} yields
			\begin{equation}\label{diamOinf}
				\diam(\Oinf) = \diam (\overline{\Oinf}) = \lim_{j\to\infty} \diam(\overline{\Oj}) 
				= 
				\lim_{j\to\infty} \diam(\Oj)= 
				R_0.
			\end{equation}
			\begin{equation}\label{diampOinf}
				\diam(\pOinf) = \lim_{j\to\infty}\diam(\pOj) = R_0
			\end{equation}

\begin{proof}[Proof of \ref{4-thm:pseudo-blow-geo} in Theorem \ref{thm:pseudo-blow-geo}] 
Notice that $\Omega_\infty\not =\emptyset$ since $0\in\partial\Omega_\infty$. Next we show that $\Omega$ satisfies the interior corkscrew and the Harnack chain. Let us sketch the argument. For the interior corkscrew condition, fixed $p\in\pOinf$ and $0<r< \diam(\pOinf)$, we take a sequence $p_j\in \partial\Omega_j$ so that $p_j\to p$ and for each $j$ we let $A_j$ be an interior corkscrew relative to $B(p_j,r/2)\cap\partial\Omega_j$ in $\Omega_j$. All the $A_j$'s are contained in $B(p, 3r/4)$, hence, passing to a subsequence, they converge to some point $A$. 
Using that the interior corkscrew condition holds for all $\Omega_j$ with the same constant $M$, it follows that each $A_j$ is uniformly away from $\partial\Omega_j$ and so will be $A$ from $\partial \Omega_\infty$ since $\pOj \to \pOinf$. I turn, this means that $A$ is an interior corkscrew relative to $B(p,r)\cap\partial\Omega_\infty$ in $\Omega_\infty$. Regarding the Harnack chain condition we proceed in a similar fashion. Fixed $X, Y\in\Omega_\infty$ for some fixed $j$ large enough we will have that $X,Y\in\Omega_j$ with $\delta_j(X)\approx \delta_\infty(X)$ and $\delta_j(Y)\approx \delta_\infty(Y)$. We can then construct a Harnack chain to join $X$ and $Y$ within  $\Omega_j$ (whose implicit constants are independent of $j$). Again, since each ball in the constructed Harnack chain is uniformly away from $\partial \Omega_j$, it will also be uniformly away from $\partial\Omega_\infty$  allowing us to conclude that this chain of balls is indeed a Harnack chain within $\Omega_\infty$.

\noindent\textbf{Interior corkscrew condition.} Recall that each $\Omega_j$ is a uniform domain with constants $M, C_1>1$. Hence, for all $q\in \pOj$ and $r\in (0,\diam(\pOj))$ there is a point $A_j(q,r)\in \Omega_j$ such that
		\begin{equation}\label{eqn106}
B\left(A_j(q,r), \frac{r}{M} \right) \subset B(q,r)\cap\Omega_j.
	\end{equation} 
			 
Let $p\in\pOinf$ and $0<r< \diam(\pOinf)$. In \ref{CaseII}, by \eqref{diampOinf} we get that $r<\diam(\pOj)$ for all $j$ sufficiently large. In \ref{CaseI}, either $\diam(\pOinf) = \infty$ or $\diam(\pOinf) < \infty$, but we still have $r<\diam(\pOj)$ for all $j$ sufficiently large (note that in the latter case $\diam(\pOj) \not\to \diam(\pOinf)$).	
	Since $\pOj\to\pOinf$, we can find $p_j\in \pOj$ converging to $ p$. For each $j$ there exists $\wAj(p_j,r/2)$ such that 
	\begin{equation}\label{eqn107}
		B\left(\wAj\left(p_j, \frac{r}{2}\right), \frac{r}{2M}\right) \subset B\left(p_j, \frac{r}{2}\right) \cap \Oj.
	\end{equation}
	In particular we deduce that
	\begin{equation}\label{saegt}
		\overline{B\left(\wAj\left(p_j, \frac{r}{2}\right), \frac{r}{3M}\right)} \subset \Oj
		\qquad\mbox{and}\qquad 
		\dist\left( B\left(\wAj\left(p_j, \frac{r}{2}\right), \frac{r}{2M}\right), \pOj \right) \geq \frac{r}{6M}.
	\end{equation}
Note that for $j$ large enough 
\begin{equation} \label{eqn108}
\wAj\left(p_j, \frac{r}{2}\right) \in B\left( p_j, \frac{r}{2} \right) \subset \overline{B\left(p,\frac{3r}{4}\right)}. 
\end{equation}
Modulo passing to a subsequence (which we relabel) $\wAj\left(p_j, r/2\right)$ converges to some point, which we denote by $A(p,r)$, and  for all $j$ sufficiently large (depending on $r$)
\begin{equation} \label{eqn109}
B\left(A(p,r), \frac{r}{4M} \right) \subset B\left(\wAj\left(p_j, \frac{r}{2}\right), \frac{r}{3M}\right)  \subset B(p,r)\cap\Oj.
 \end{equation}
The fact that $\overline{\Oj} \to \overline{\Oinf}$, the first inclusion in \eqref{eqn109}, and \eqref{saegt} give for all $j$ large enough
 \begin{equation}\label{eq:temp107}
 	B\left(A(p,r),\frac{r}{4M}\right) \subset \overline{\Oinf}
 	\qquad\mbox{and}\qquad
 	\dist\left( B\left(A(p, r), \frac{r}{4M}\right), \pOj \right) \geq 
 	\frac{r}{6M}.
 \end{equation} 
This and the fact that $\pOj \to \pOinf$ yield that $\dist(B(A(p,r),r/4M), \pOinf)\ge r/6M$, hence $B(A(p,r),r/4M)$ misses $\pOinf$. Combining this with \eqref{eq:temp107} and the second inclusion in \eqref{eqn109}, we conclude that 
\begin{equation}\label{eq:NTOinf}
 			B\left(A(p,r),\frac{r}{4M} \right) \subset \Oinf \cap B(p,r).
 		\end{equation}
Hence, $\Omega_\infty$ satisfies the interior corkscrew condition with constant $4M$.
\medskip

\noindent \textbf{Harnack chain condition.} Fix $X,Y\in \Oinf$ and pick $q_X, q_Y\in\pOinf$ such that $|X-q_X| = \delta_{\infty}(X), |Y-q_Y|=\delta_\infty(Y)$. 	Without loss of generality we may assume that $\delta(X)\ge \delta(Y)$ (otherwise we switch the roles of $X$ and $Y$). Let us recall that every $\Omega_j$ satisfies the Harnack chain condition with constants $M, C_1>1$.
Set 
\begin{equation}\label{Def:Theta}
\Theta:=M\left(2+\log_2^+\left( \frac{|X-Y|}{\min\{\delta_{\infty}(X),\delta_{\infty}(Y)\}} \right)\right)
=
M\left(2+\log_2^+\left( \frac{|X-Y|}{\delta_{\infty}(Y)} \right)\right).
\end{equation}
  Choose $R\ge $ large enough (depending on $X,Y$) so that
\begin{equation}\label{BqX}
B(q_X,\delta_\infty(X)/2), B\big(X, (2C_1^2)^{4\Theta}\delta_\infty(X)\big)
\subset B(0,R)
\end{equation}
and
\begin{equation}\label{BqY}
B(q_Y,\delta_\infty(Y)/2),  B\big(Y, (2C_1^2)^{4\Theta}\delta_\infty(Y)\big)\subset B(0,R)
\end{equation}
Take also $d=2^{-1} C_1^{-2\Theta}\le 1$ which also depends on $X,Y$. Then, by \ref{3-thm:pseudo-blow-geo} in Theorem \ref{thm:pseudo-blow-geo} we can take $j$ large enough (depending on $R$ and $d$) so that
	\begin{equation}\label{eq:temp109}
	    D\big[\pOj \cap \overline{B(0,R)}, \pOinf\cap \overline{B(0,R)}\big], D\big[\overline\Oj \cap \overline{B(0,R)}, \overline\Oinf\cap \overline{B(0,R)}\big] \leq  \frac{d}{2} \delta_{\infty}(Y)\le \frac{d}{2}  \delta_{\infty}(X),
	\end{equation}
By \eqref{eq:temp109}, \eqref{BqX}, and \eqref{BqY} we have that $X,Y\in \Oj$, and 
	\begin{equation}\label{eqn110}
	 \frac{\delta_{\infty}(X)}{2} \leq \wdj(X) \leq \frac{3\delta_{\infty}(X)}{2}\qquad\mbox{and}\qquad\frac{\delta_{\infty}(Y)}{2} \leq \wdj(Y) \leq \frac{3\delta_{\infty}(Y)}{2}. 
	 \end{equation}
	Since $\Oj$ satisfies the Harnack chain condition with constants $M, C_1>1$, there exists a collection  of balls $B_1,\dots, B_K$ (the choice of balls  depend on the fixed $j$) connecting $X$ to $Y$ in $\Oj$ and such that
	\begin{equation} \label{eq:HBj}
		C_1^{-1} \dist(B_k,\partial\Omega_j) \leq \diam(B_k) \leq C_1 \dist(B_k,\partial\Omega_j),  
	\end{equation} 
	for $k=1,2,\dots,K$ where
	\begin{equation}\label{eqn111}
	 K 
	 \leq 
	 M\left(2+\log_2^+\left( \frac{|X-Y|}{\min \{ \wdj(X), \wdj(Y) \}} \right)\right)
	 \leq 
	 2\Theta.
	  \end{equation}
	  Combining \eqref{eq:HBj} and \eqref{eqn111}, one can see that for every $k = 1,2,\dots,K$
	  \begin{equation}\label{aaa1}
	  \dist(B_k,\partial\Omega_j)  \geq d \delta_{\infty}(X), \qquad \diam(B_k)\le (2C_1^2)^{2\Theta}\delta_\infty(Y)
	  \end{equation}
	  and
	  \begin{equation}\label{aaa2}
\dist(X,B_k)\le 2(2C_1^2)^{2\Theta}\delta_{\infty}(X),\qquad \dist(Y,B_k) \le 2(2C_1^2)^{2\Theta} \delta_{\infty}(Y).
	  \end{equation}

Given an arbitrary $q_j\in\partial\Omega_j\setminus \overline{B(0,R)}$, by \eqref{BqX}, \eqref{aaa1},	and \eqref{aaa2}	 it follows that 
\begin{multline}\label{rfarefer}
(2C_1^2)^{4\Theta}\delta_\infty(X)\le |q_j-X|\le \dist(q_j,B_k)+\diam(B_k)+\dist(X,B_k)
\\
\le
\dist(q_j,B_k)+
3(2C_1^2)^{2\Theta}\delta_{\infty}(X).
\end{multline}
Hiding the last term, using that $\Theta>2$ and taking the infimum over the $q_j$ as above we conclude that
\begin{equation}\label{gvasf}
4C_1 (2C_1^2)^{2\Theta}\delta_\infty(X)< \dist(B_k,\partial\Omega_j\setminus \overline{B(0,R)}).
\end{equation}
On the other hand, by \eqref{eq:HBj} and \eqref{aaa1}
$$
\dist(B_k,\partial\Omega_j)\le C_1\diam(B_k)
\le C_1(2C_1^2)^{2\Theta}\delta_\infty(Y)
\le
C_1(2C_1^2)^{2\Theta}\delta_\infty(X),
$$
which eventually leads to $\dist(B_k,\partial\Omega_j)=\dist(B_k,\partial\Omega_j\cap \overline{B(0,R)})$.
Analogously, replacing $q_j$ by $q\in\partial\Omega_\infty\setminus \overline{B(0,R)}$ in \eqref{rfarefer} we can easily obtain that \eqref{gvasf} also holds for $\Omega_\infty$: 
\begin{equation}\label{gvasserf}
4C_1 (2C_1^2)^{2\Theta}\delta_\infty(X)< \dist(B_k,\partial\Omega_\infty\setminus \overline{B(0,R)}).
\end{equation}
But, \eqref{aaa2} yields
$$
\dist(B_k,\partial\Omega_\infty)
\le
\delta_\infty(X)+\dist(X,B_k)
\le
\delta_\infty(X)+
 2(2C_1^2)^{2\Theta}\delta_\infty(Y)
 \le 3(2C_1^2)^{2\Theta}\delta_\infty(Y)
 ,
$$
which eventually leads to $\dist(B_k,\partial\Omega_\infty)=\dist(B_k,\partial\Omega_\infty\cap \overline{B(0,R)})$.
Using all these, \eqref{eq:temp109}, the triangular inequality and \eqref{eq:temp109} we can obtain
\begin{multline*}
\big|\dist(B_k,\partial\Omega_j) -\dist(B_k,\partial\Omega_\infty)\big| 
=
\big|\dist(B_k,\partial\Omega_j\cap \overline{B(0,R)})-\dist(B_k,\partial\Omega_\infty\cap \overline{B(0,R)})\big|
\\
\leq D\big[\pOj \cap \overline{B(0,R)},\pOinf\cap \overline{B(0,R)}\big] \leq \frac{d}{2} \delta_{\infty}(X) 
\leq \frac12 \dist(B_k,\Omega_j).
\end{multline*}
Thus,
\begin{equation}\label{y6e5}
\frac23 \dist(B_k,\partial\Omega_\infty) \leq \dist(B_k,\partial\Omega_j) \leq 2 \dist(B_k,\partial\Omega_\infty).
\end{equation}
and moreover $B_k\cap\partial\Omega_\infty=\emptyset$. Note that the latter happens for all $k=1,\dots, K$. Recall also that $X\in B_1\cap\Omega_\infty$ and that $B_k\cap B_{k+1}\neq\emptyset$. Consequently, we necessarily have that $B_k \subset \Oinf$ for all $k=1,\dots, K$. Furthermore, \eqref{y6e5} and \eqref{eq:HBj} give
	  \begin{equation} \label{eq:HBj-infty}
		    \frac{2}{3} C_1^{-1} \dist(B_k,\partial\Omega_\infty)\leq \diam(B_k) \leq 2C_1 \dist(B_k,\partial\Omega_\infty). 
	\end{equation}
To summarize, we have found a chain of balls $B_1, \dots, B_K$, all contained in $\Oinf$, which  verify \eqref{eq:HBj-infty}, and connect $X$ to $Y$. Also,
$K$ satisfies \eqref{eqn111} with $\Theta$ given in \eqref{Def:Theta}. Therefore $\Omega_\infty$ satisfies the Harnack chain condition with constants $2M$ and $2C_1$. This completes the proof of  \ref{4-thm:pseudo-blow-geo} in Theorem \ref{thm:pseudo-blow-geo}.
\end{proof}

\begin{proof}[Proof of \ref{5-thm:pseudo-blow-geo} in Theorem \ref{thm:pseudo-blow-geo}]
We first recall that for every $j$, $\sj=\HH^{n-1}|_{\partial\Omega_j}$ is an Ahlfors regular measure with constant $C_{AR}$ and hence $\spt\sj=\partial\Omega_j$. 
In particular the sequence $\{\sj\}$ satisfies conditions \ref{1-lm:sptcv} and \ref{2-lm:sptcv} of Lemma \ref{lm:sptcv}. 

 On the other hand, the fact that $\partial\Omega_j$ is Ahlfors regular easily yields, via a standard covering argument, that $\HH^{n-1}(\partial\Omega_j) \le 2^{n-1}C_{AR} \diam(\Omega_j)^n$. Hence, using again that $\partial\Omega_j$ is Ahlfors regular we conclude that for every $R>0$
\[ 
\sup_j \sj(B(0,R)) = \sup_j \HH^{n-1}(\partial\Omega_j \cap B(0,R)) \leq 2^{n-1}C_{AR}  R^{n-1}.
\]
Therefore modulo passing to a subsequence (which we relabel), there exists a Radon measure $\sinf$ such that $\sj \rightharpoonup \sinf$ as $j\to\infty$.
Using Lemma \ref{lm:sptcv}, $\partial\Omega_j=\spt\sj \to \spt\sinf$  as $j\to\infty$ in the sense of Definition \ref{def:cvsets}. This and  \ref{3-thm:pseudo-blow-geo} in Theorem \ref{thm:pseudo-blow-geo}  lead to $\spt\sinf = \pOinf$. 

 To show that $\mu_\infty$ is Ahlfors regular take $q\in\partial\Omega_\infty$. Let $q_j\in\pOj$ be such that $q_j\to q$ as $j\to\infty$. For  any $r>0$, using
 \cite[Theorem 1.24]{Ma}  and that $\sigma_j$ is Ahlfors regular with constant $C_{AR}$ we conclude that
	\begin{equation}\label{eq:upperAR}
\mu_\infty(B(q,r)) \leq \liminf_{j\to\infty}  \sj(B(q,r)) \leq \liminf_{j\to\infty} \sj(B(q_j,2r)) 
\leq 2^{n-1} C_{AR}r^{n-1}.
\end{equation}
On the other hand, let $0<r<\diam(\partial\Omega_\infty)$. In \ref{CaseII}, by \eqref{diampOinf} we get that $r<\diam(\pOj)$ for all $j$ sufficiently large. In \ref{CaseI}, either $\diam(\pOinf) = \infty$ or $\diam(\pOinf) < \infty$, but we still have $r<\diam(\pOj)$ for all $j$ sufficiently large. Hence, using again \cite[Theorem 1.24]{Ma} and that $\sigma_j$ is Ahlfors regular with constant $C_{AR}$ we obtain 
  	\begin{multline} \label{eq:lowerAR}
		\mu_\infty(B(q,r)) \geq \mu_\infty\left(\overline{B\left(q,\frac{r}{2}\right)}\right)  
		\geq 
		\limsup_{j\to\infty} \sj\left(\overline{B\left(q,\frac{r}{2}\right)}\right) 
		\\
		\geq 
		\limsup_{j\to\infty} \sj\left(B\left(q_j,\frac{r}{4}\right) \right) 
		 \geq 4^{-(n-1)}C_{R}^{-1}r^{n-1}. 
	\end{multline} 
 These estimates guarantee that $\minf$ is Ahlfors regular with constant $2^{2(n-1)}C_{AR}$. Moreover by \cite[Theorem 6.9]{Ma}, 
	\begin{equation}  \label{eqn115}
	2^{-2(n-1)}C_{AR}^{-1} \minf \leq \mathcal{H}^{n-1}|_{\pOinf} \leq 2^{3(n-1)}C_{AR} \minf. 
	\end{equation}
	and consequently $\partial\Omega_\infty$ is Ahlfors regular with constant $2^{5(n-1)}C_{AR}^2$. This completes the proof of \ref{5-thm:pseudo-blow-geo} and hence that of Theorem \ref{thm:pseudo-blow-geo}. 
\end{proof}

\subsection{Convergence of elliptic matrices}
Our next goal is to show that there exists a constant coefficient real 
symmetric elliptic matrix $\wcalA^*$ with 
ellipticity constants 
$1=\lambda\le \Lambda<\infty$
(as in \eqref{def:UE}) so that for any $0<R<\diam(\pOinf)$ and for any $1\le p<\infty$.
\begin{equation}\label{eq:temp118}
\int_{ B(0,R) \cap \Oj } |\wcalA_{j}(Z) - \wcalA^*|^p dZ \to 0,\qquad \text{ as } j\to \infty.
\end{equation}

Fix $Z_0 \in \Oinf$ and set $B_0 = B(Z_0, 3\delta_{\infty}(Z_0)/8)$. Since $\pOj \to \pOinf$ and $\overline\Oj \to \overline\Oinf$ as $j\to\infty$,  for all sufficiently large $j$, we can see that $Z_0 \in\Oj$, 
\begin{equation}\label{eq:djdinfty}
	\frac{3}{4} \delta_{\infty}(Z_0) \leq  \wdj(Z_0) \leq \frac{5}{4} \delta_{\infty}(Z_0),
\end{equation}  
and 
\begin{equation}
B_0 \subset B\left(Z_0,\frac{\wdj(Z_0)}{2} \right) \subset \frac{5}{3} B_0 \subset \Oj \quad \text{for all } j.
\end{equation}
All these, an the oscillation assumption \eqref{def:oscA}, yield 
\begin{multline}\label{esti-conve-Aj}
	\fint_{B_0} |\wcalA_{j}(Z) - \langle\wcalA_j\rangle_{B_0}| dZ
	\lesssim
	\fint_{B(Z_0,\delta_j(Z_0)/2)} |\mathcal{A}_j(Z) - \langle\mathcal{A}_j\rangle_{B(Z_0,\delta_j(Z_0)/2)}| dZ 
	\\
	\le
	\osc(\Omega_j, \mathcal{A}_j)
		< \epsilon_j.
\end{multline}
Note that all the matrices $\wcalA_{j}$ are 
uniformly elliptic (i.e., all of them satisfy \eqref{def:UE}), with the 
same constants 
$1=\lambda\le\Lambda<\infty$,
and in particular $\{\langle\wcalA_{j} \rangle_{B_0}\}_j$ is a bounded sequence of constant real matrices. Hence, passing to a subsequence and relabeling $\langle\wcalA_{j} \rangle_{B_0}$ converges to some constant elliptic matrix, denoted by $\wcalA^*(B_0)$. Combining this with \eqref{esti-conve-Aj}, the dominated convergence theorem yields 
\begin{equation}\label{conv-B0}
	\fint_{B_0} |\wcalA_{j} (Z) - \wcalA^*(B_0)| dZ \to 0\quad \text{ as } j\to\infty,
\end{equation}
that is,  $\wcalA_{j}$ converges in $L^1(B_0)$ to a constant elliptic matrix $\wcalA^*(B_0)$.  Moreover, passing to a further subsequence an relabeling $\wcalA_{j}\to \wcalA^*(B_0)$ almost everywhere in $B_0$. In 
particular, $\wcalA^*(B_0)$ is a real symmetric elliptic matrix
(i.e., it satisfies \eqref{def:UE}), with ellipticity constants 
$1=\lambda\le \Lambda<\infty$. 
It is important to highlight that all the previous subsequences and relabeling only depends 
on the choice of $Z_0 \in \Oinf$.  In any case, since $\wcalA^*(B_0)$ is a constant coefficient matrix we set 
$\wcalA^*:=\wcalA^*(B_0)$.

Let us pick a countable collection of points $\{Z_k\}\subset \Omega_\infty$ so that $\Omega_\infty=\cup_k B_k$ with $B_k = B(Z_k, 3\delta_{\infty}(Z_k)/8)$. 
We can repeat the previous argument with any $Z_k$ and define $\wcalA^*(B_k)$, a constant  real symmetric elliptic matrix satisfying \eqref{def:UE} so that for some subsequence depending on $k$, we obtain that  $\wcalA_{j}\to \wcalA^*(B_k)$ in $L^1(B_k)$ and a.e in $B_k$ as $j\to\infty$. In particular, $\wcalA^*(B_{k_1})=\wcalA^*(B_{k_2})$ a.e. in  $B_{k_1}\cap B_{k_2}$ (in case it is non-empty). Note that $\Oinf$ is path connected (since it satisfies the Harnack chain condition), hence for any $k$ we can find a path joining $Z_k$ and $Z_0$ and cover this path with a finite collection of the previous balls to easily see that $\wcalA^*(B_{k})=\wcalA^*=\wcalA^*(B_0)$. Moreover, using a diagonalization argument, we can show that there exists a subsequence, which we relabel, so that for all $k$, we have that  $\wcalA_{j}\to \wcalA^*$ in $L^1(B_k)$ and a.e in $B_k$ as $j\to\infty$. From this, and since the matrices concerned are all uniformly bounded, one can prove that for any $1\leq p<\infty$ and for all $Z\in\Omega_\infty$
\begin{equation}\label{Ajcvinp}
\fint_{B_Z} |\wcalA_{j}(Y) - \wcalA^*|^p dY \to 0\quad \text{ as } j\to\infty,
\end{equation}
where $B_Z=B(Z,\delta(Z)/2)$.

We are now ready to start proving our claim \eqref{eq:temp118}. Recalling that $\overline{\Oj} \to \overline{\Oinf}$, $\pOj \to \pOinf$ in the sense of Definition \ref{def:cvsets}, and that $\pOj, \pOinf$ have zero Lebesgue measure since they are Ahlfors regular sets, one can see that 
\begin{align}
	B(0,R) \cap \left( \Oj \triangle \Oinf \right) \subset B(0,R) \cap \left( \left(\overline\Oj \triangle \overline\Oinf \right) \cup \left(\overline\Oj \cap \pOinf \right) \cup \left( \overline\Oinf \cap \pOj \right) \right)
\end{align}
and hence the Lebesgue measure of the set on the left hand side  tends to zero as $j\to\infty$. This and the fact that $\|\wcalA_{j}\|_\infty, \|\wcalA^*\|_\infty\le\Lambda$ give
\begin{equation}\label{eq:temp119}
	\int_{B(0,R) \cap \left(\Oj \triangle \Oinf \right)} |\wcalA_{j}(Z) - \wcalA^*|^p dZ \to 0,\qquad \text{ as } j \to \infty.
\end{equation}

On the other hand, let $\varrho>0$ be arbitrarily small and let $\epsilon = \epsilon(\varrho)>0$ be a small constant to be determined later. Set 
\[
\Oinf^{\epsilon,1} := B(0,R) \cap \{Z\in\Oinf: \delta_{\infty}(Z) < \epsilon\}
\quad\mbox{and}\quad
\Oinf^{\epsilon,2} := B(0,R) \cap \{Z\in\Oinf: \delta_{\infty}(Z) \geq \epsilon\}.
\]
Using the notation $\Delta(q,r):=B(q,r)\cap \pOinf$ with $q\in\pOinf$ and $r>0$, Vitali's covering lemma allows us to find a finite collection of balls $B(q_i,\epsilon)$ with $q_i\in \Delta(0,R+\epsilon)$, such that
\begin{equation}\label{est:Omega-near}
\Oinf^{\epsilon,1} \subset \bigcup_{i} B(q_i,5\epsilon). 
\end{equation}
Calling the number of balls $L_1$ we get the following estimate 
\begin{align}\label{eq:countL2}
L_1\epsilon^{n-1} 
\lesssim
\sum_{i} \sigma_\infty \left(  \Delta(q_i,\epsilon) \right) 
= 
\sigma_\infty \Big( \bigcup_i\Delta(q_i,\epsilon) \Big) \leq \sigma_\infty \left( \Delta(0,R+2\epsilon) \right)
\lesssim 
(R+2\epsilon)^{n-1},
\end{align}
where we have used that $\pOinf$ is Ahlfors regular and also that $\Delta(q_i, \epsilon) \subset \Delta(0, R+2\epsilon)$ since $q_i \in \Delta(0,R+\epsilon)$. If we assume that $0<\epsilon<R$ we conclude that $L_1\lesssim (R/\epsilon)^{n-1}$ and moreover by \eqref{est:Omega-near} we conclude that $|\Oinf^{\epsilon,1}|\lesssim \epsilon$ (here the implicit constant depend on $R$). This and $\|\wcalA_{j}\|_\infty, \|\wcalA^*\|_\infty\le\Lambda$ give at once that for every $j$
\begin{equation}\label{eq:e1}
	\int_{\Oinf^{\epsilon,1} \cap \Oj} |\wcalA_{j}(Z) - \wcalA^*|^p dZ \lesssim \Lambda^p \epsilon<\frac{\varrho}2, 
\end{equation} 
provided $\epsilon$ is taken small enough which is fixed from now on. 

On the other hand, note that $\overline{\Oinf^{\epsilon,2}}$ is compact, hence we can find $Z_1,\dots, Z_{L_2}\in \overline{\Oinf^{\epsilon,2}}$ so that 
$\Oinf^{\epsilon,2} \subset \bigcup_{i=1}^{L_2} B_{Z_i}$ where $L_2$ depends on $\epsilon$ and $R$ which have been fixed already. Hence, by \eqref{Ajcvinp}
\[
	\int_{\Oinf^{\epsilon,2} \cap \Oj } |\wcalA_{j}(Z) - \wcalA^*|^p dZ 
	\leq \sum_{i=1}^{L_2} \int_{B_{Z_i}} |\wcalA_{j}(Z) - \wcalA^*|^p dZ 
	\to 0,\quad \text{ as } j\to\infty.
\]
In particular, we can find an integer $j_0 = j_0(R,\epsilon)$ such that
\begin{equation}\label{eq:e2}
	\int_{\Oinf^{\epsilon,2} \cap \Oj } |\wcalA_{j}(Z) - \wcalA^*|^p dZ < \frac{\varrho}{2}, \quad \text{ for any } j\geq j_0.
\end{equation}
Combining \eqref{eq:e1} and \eqref{eq:e2}, we conclude that
\begin{equation}\label{eq:temp120}
	\int_{B(0,R) \cap \left(\Oj \cap\Oinf\right)} |\wcalA_{j}(Z) - \wcalA^*|^p dZ <\varrho, \quad \text{ for any } j\geq j_0.
\end{equation}
This combined with \eqref{eq:temp119} proves the claim \eqref{eq:temp118}.

\subsection{Convergence of operator}
\begin{theorem}\label{thm:AW11}
The function $u_{\infty}$ solves the Dirichlet problem 
\begin{equation}\label{limit-eq2}
	 \left\{ \begin{array}{rl}
		-\divg(\mathcal{A}^* \nabla u_\infty) = 0 & \text{in } \Oinf, \\
		u_\infty > 0 & \text{in }\Oinf, \\
		u_\infty = 0 & \text{on } \pOinf,
	\end{array} \right.  
\end{equation}
in \ref{CaseI}, and solves the Dirichlet problem
\begin{equation}\label{limit-eq2case2}
	 \left\{ \begin{array}{rl}
		-\divg(\mathcal{A}^* \nabla u_\infty) = \delta_{\{X_0\}} & \text{in } \Oinf, \\
		u_\infty > 0 & \text{in }\Oinf, \\
		u_\infty = 0 & \text{on } \pOinf,
	\end{array} \right.
\end{equation}
in \ref{CaseII}. Hence, $u_\infty$ is a Green function in $\Oinf$ for a constant-coefficient elliptic operator $L_\infty=-\divg(\mathcal{A}^* \nabla)$ with pole at $\infty$ in \ref{CaseI} or at $X_0\in\Oinf$ in \ref{CaseII}.
\end{theorem}

\begin{proof}
Let $\psi\in C^\infty_c(\Oinf)$. Since $\overline{\Oj}\to\overline{\Oinf}$ and $\pOj\to\pOinf$, it follows that $\psi\in C^\infty_c(\Oj)$  for $j$ sufficiently large. 
In \ref{CaseI}, using \eqref{def:ujcasei} and \eqref{eqn:int-parts}  we have
	\begin{equation}\label{eqn116}
		\int_{\Oj} \langle \wcalA_{j}\nabla \uj,\nabla \psi\rangle dZ 
		=
	 \frac{1}{\oj^{X_j}(B(0,1)) }\int_{\Omega_j} \langle\wcalA_j\nabla G_j(X_j,\cdot), \nabla \psi\rangle  dZ
	 =  
	 \frac{\psi(X_j)} {\omega_j^{X_j}(B(0,1))} \to 0,
	 \end{equation}
	 as $j\to\infty$ since $\wXj \to \infty$ by \eqref{eq:temp1}. 
	 Analogously, in \ref{CaseII}, by \eqref{def:ujcaseii} and \eqref{eqn:int-parts} we obtain 
	 \begin{equation}\label{eqn116case2}
	\int_{\Oj} \langle \wcalA_{j}\nabla \uj,\nabla \psi\rangle dZ 
	=
	\int_{\Omega_j} \langle\wcalA_j\nabla G_j(X_j,\cdot), \nabla \psi\rangle  dZ
	=
	\psi(X_j) \to \psi(X_0).
	 \end{equation}
	 as $j\to\infty$ since $\wXj \to X_0$.
	
	Suppose next that $\spt \psi \subset B(0,R)$. Let $r=2$ for \ref{CaseI}, and pick $r\in [1,n/(n-1))$ for \ref{CaseII}. 	
	By \ref{1-thm:pseudo-blow-geo} in Theorem \ref{thm:pseudo-blow-geo} in \ref{CaseI} and \ref{3-rem:deawffr} in Remark \ref{rem:deawffr} in \ref{CaseII} it follows that $\nabla \uj \rightharpoonup \nabla u_{\infty}$ in $L^r(B(0,R))$. On the other hand, 
	\begin{align}\label{eqn:3.8}
		& \left| \int_{\Oj}  \langle \wcalA_{j}\nabla \uj,\nabla \psi\rangle dZ - \int_{\Oinf}  \langle \wcalA^*\nabla u_\infty,\nabla \psi\rangle
		dZ \right|  
		\\
				& \qquad \leq \|\nabla \psi\|_{L^{\infty}} \left( \int_{\Oj\cap B(0,R)} |\wcalA_{j} - \wcalA^* |^{r'} dZ \right)^{\frac{1}{r'}} \left( \int_{\Oj\cap B(0,R)} |\nabla \uj|^r \right)^{\frac{1}{r}} 
		\nonumber \\
		& \qquad \qquad \qquad + \left| \int_{\Oj \cap B(0,R)} \langle \wcalA^* \nabla \uj, \nabla \psi \rangle dZ - \int_{\Oinf \cap B(0,R)} \langle \wcalA^* \nabla u_{\infty}, \nabla \psi \rangle dZ \right|. \nonumber 
	\end{align}
	Using \eqref{eqn:3.4} in \ref{CaseI} or \eqref{eq:upperr} in \ref{CaseII},  and \eqref{eq:temp118} with $p=r'$, the term in the second line of \eqref{eqn:3.8} tends to zero as $j\to\infty$.
	Concerning the last term, since $\wcalA^*$ is a constant-coefficient matrix, it follows that $\wcalA^* \nabla \uj \rightharpoonup \wcalA^* \nabla u_{\infty}$ in $L^r(B(0,R))$. Moreover $\overline\Oj=\overline{\{\uj > 0 \}} \to \overline\Oinf=\overline{\{u_{\infty} > 0 \}}$, thus
\[
		\lim_{j\to\infty} \int_{\Oj} \langle \wcalA^* \nabla \uj, \nabla \psi\rangle = \int_{\Oinf} \langle \wcalA^* \nabla u_{\infty}, \nabla \psi \rangle.
\]
	Combining these with \eqref{eqn116}--\eqref{eqn:3.8} we eventually conclude that
	\begin{equation}
		\int_{\Oinf} \wcalA^* \nabla u_{\infty} \cdot \nabla \psi = 0 \quad \text{for all } \psi \in C_c^\infty(\Oinf)
	\end{equation}
	in \ref{CaseI}, i.e., $-\divg(\mathcal{A}^*\nabla u_{\infty}) = 0$ in $\Oinf$; and in \ref{CaseII},
	\begin{equation}
		\int_{\Oinf} \wcalA^* \nabla u_{\infty} \cdot \nabla \psi = \psi(X_0) \quad \text{for all } \psi \in C_c^\infty(\Oinf),
	\end{equation}
	 i.e., $-\divg(\mathcal{A}^*\nabla u_{\infty}) = \delta_{\{X_0\}}$ in $\Oinf$.
\end{proof}

\subsection{Analytic properties of the limiting domains}
As mentioned in Section \ref{sect:blowup}, in order to apply Theorem \ref{thm:hmu} we need to study the elliptic measures of the limiting domain with finite poles. In this section we construct these measures by a limiting procedure which is compatible with the procedure used to produce the limiting domain $\Oinf$.

\begin{theorem}\label{thm:blow-ana-pole}
	Under  \ref{1-assump},  \ref{2-assump},  \ref{3-assump}, and using the notation from Theorems \ref{thm:pseudo-blow-geo} and \ref{thm:AW11}, 
	the elliptic measure $\omega_{L_{\infty}}\in A_\infty(\sigma_\infty)$ (see Definition \ref{def:AinftyHMU}) with constants $\widetilde{C}_0=C_2 C_{AR}^{4\theta} 2^{8(n-1)\theta}$ and $\widetilde{\theta}=\theta$, here $C_2$ is the constant in Remark \ref{rem:doubling:needed}.
\end{theorem}

\begin{proof}
	Our goal is to show that the elliptic measures of $L_{\infty}$ with finite poles can be recovered as a limit of the elliptic measures of $\wLj = -\divg(\wcalA_{j}(Z)\nabla)$, and the $A_{\infty}$ property of elliptic measures is preserved when passing to a limit.

To set the stage we start with $0\le f\in \Lip(\partial\Omega_\infty)$ with compact support. Let $R_0>0$ be large enough so that $\spt f\subset B(0,R_0/2)$. We are going to take a particular solution to the following Dirichlet problem 
	\begin{equation}\label{D:Oinf}
\left\{ \begin{array}{ll}
L_\infty v = 0, & \text{in }\Oinf \\
v= f, & \text{on }\pOinf,
\end{array} \right.
\end{equation} 

In \ref{CaseII}, where the domain $\Oinf$ is bounded, the Dirichlet problem \eqref{D:Oinf} has a unique solution satisfying the maximum principle,
then we let $v_\infty$ be that unique solution. 

In \ref{CaseI}, where $\Oinf$ is unbounded, we follows the construction in \cite{HM} using Perron's method (see \cite[pg. 588]{HM} for details, which is done the Laplacian but holds for any constant coefficient operator, for the general case see also \cite{HMT2}). We denote the solution constructed in \cite{HM} by 
\[
v_\infty(Z)=
\int_{\pOinf} f(q) d\omega_{L_\infty}^Z(q).
\]
For later use we need to sketch how it is constructed. For every $R>4R_0$ define $f_R=f \eta(\cdot/R)$, where $\eta\in C_c^\infty(B(0,2R)$ verifies $0\leq \eta\leq 1$, $\eta = 1$ for $|Z| < 1$. Let $v_R$ be the unique solution to $L_{\infty} v_R = 0$ in the bounded open set $\Omega_R = \Oinf \cap B(0,2R)$ with boundary value $f_R$.  Then one shows that $v_R\to v_\infty$ uniformly on compacta as $R\to\infty$,  and also that $v_\infty\in C(\overline{\Omega_\infty})$ satisfies the maximum principle $0\le \max_{\Omega_\infty}v_\infty\le \max_{\partial\Omega_\infty} f$.

Once the solution $v_\infty$ is defined we observe that since $\partial\Omega_\infty$ is Ahlfors regular we can use the 
Jonsson-Wallin trace/extension theory \cite{JW} to extend $f$ (abusing the notation we call the extension $f$) so that $0\le f\in C_c(\R^n)\cap W^{1,2}(\R^n)$ with $\spt f\subset B(0,R_0)$. For every $j$ we let $h_j\in W^{1,2}_0(\Omega_j)$ be the unique Lax-Milgram solution to the problem $L_j h_j= L_j f$. Initially, $h_j$ is only defined in $\Omega_j$ but we can clearly extend it by $0$ outside so that the resulting function, which we call again  $h_j$, belongs to $W^{1,2}(\R^n)$. If we next set $v_j=f-h_j\in W^{1,2}(\R^n)$ we obtain that $L_jv_j=0$ in $\Omega_j$ and indeed 
	\begin{equation}\label{eqn:102A}
		v_j(Z) = \int_{\pOj} f d\oj^Z, \qquad Z\in\Oj,
	\end{equation}
	see \cite{HMT2}. Here $\oj^Z$ is the elliptic measure of $\wLj$ in $\Oj$ with pole $Z$ and, as observed above, the fact that $\partial\Omega_j$ is Ahlfors regular implies in particular that $v_j\in C(\overline{\Omega_j})$ with $v_j|_{\partial\Omega_j}=f$. Note also that $v_j=f\in C(\R^n)$ on $\R^n\setminus\Omega_j$, hence $v_j\in\ C(\R^n)$. Moreover, by the maximum principle 
\begin{equation}\label{eqn:max-pple}
0\le \sup_{\Oj} v_j \leq \|f\|_{L^{\infty}(\pOj)} \leq \|f\|_{L^{\infty}(\RR^n)},
\end{equation}
thus the sequence $\{v_j\}$ is uniformly bounded. 
	
Our next goal is to show that $\{v_j\}$ is equicontinuous. Given an arbitrary $\varrho>0$ let $0<\gamma<\frac1{32}$ to be chosen. Since $f\in C_c(\R^n)$, it is uniformly continuous, hence letting $\gamma$ small enough (depending on $f$) we can guarantee that
	\begin{equation}\label{f-uc}
		|f(X)-f(Y)|<\frac{\varrho}8,
		\qquad
		\mbox{provided \,} |X-Y|<\gamma^{\frac14}
	\end{equation}

Our first claim is that if $\gamma$ is small enough depending on $n$, $C_{AR}$, $\Lambda$ (recall that
$\lambda = 1$), and $\|f\|_{L^{\infty}(\R^n)}$, there holds 
\begin{equation}\label{claim-equicont}
|v_j(X)-v_j(Y)|<\frac{\rho}{2},
\qquad
\forall\,X\in\Omega_j,\ Y\in\partial\Omega_j,\ |X-Y|<\sqrt{\gamma}.
\end{equation}
To see this we recall that $\pOj$ is Ahlfors regular with a uniform constant (independent of $j$), it satisfies the CDC with a uniform constant and \cite[Theorem 6.18]{HKM} (see also \cite{HMT2}) yields that for some $\beta>0$ and $C$ depending on $n$, $C_{AR}$, and $\Lambda$, but independent of $j$ (indeed this is the same $\beta$ as in Lemma \ref{lem:vanishing}), the following estimate holds:
\[
	\underset{B(Y_j,\sqrt{\gamma})\cap \Oj }{\osc} v_j
	\le 
	\underset{B(Y_j,\gamma^{1/4})\cap \pOj }{\osc} f  
	+ 
	C\|f\|_{L^{\infty}(\R^n)} \eta^{\beta}<\frac\varrho2,
\]
	where in the last estimate we have used \eqref{f-uc} and $\gamma$ has been chosen small enough so that $C\|f\|_{L^{\infty}(\R^n)} \eta^{\beta}<\varrho/4$.

We now fix $X,
 Y\in\R^n$ so that $|X-Y|<\gamma$ and consider several cases. 

\noindent\textbf{Case 1}: $X,Y\in\Omega_j$ with $\max\{\delta_j(X),\delta_j(Y)\}<\sqrt{\gamma}/2$.

In this case, we take $\widehat{x}\in\partial\Omega_j$ so that $|X-\hat{x}|=\delta_j(X)$. Note that $|Y-\widehat{x}|<\sqrt{\gamma}$ and we can use \eqref{claim-equicont} to obtain
\[
|v_j(X)-v_j(Y)|
\le
|v_j(X)-v_j(\widehat{x})|
+
|v_j(\widehat{x})-v_j(Y)|<\rho.
\]

\noindent\textbf{Case 2}: $X,Y\in \Omega_j$ with $\max\{\delta_j(X),\delta_j(Y)\}\ge \sqrt{\gamma}/2$.

Assuming without loss of generality that $\delta_j(X)\ge \sqrt{\gamma}/2$, necessarily $Y\in B(X,\delta_j(X)/2)\subset\Omega_j$. Then, by the interior H\"older regularity of $v_j$ in $\Oj$ (here $\alpha$ and $C$  depend only on $\Lambda$ and are independent of $j$) we conclude that
\[
|v_j(X) - v_j(Y)| 
\le 
C\left( \frac{|X-Y|}{\wdj(X)} \right)^{\alpha} \|v_j\|_{L^{\infty}(\Oj)} 
\le
C 2^\alpha \gamma^{\frac{\alpha}2}\|f\|_{L^{\infty}(\R^n)}
<
\varrho,
\]
provided $\varrho$ is taken small enough (again independently of $j$).

\noindent\textbf{Case 3}: $X,Y\notin \Omega_j$.

Here we just need to use \eqref{f-uc} and the fact that $v_j=f$ on $\R^n\setminus \Omega_j$:
\[
|v_j(X) - v_j(Y)| 
=
|f(X)-f(Y)|<
\rho.
\]

\noindent\textbf{Case 4}: $X\in \Omega_j$ and $Y\notin \Omega_j$.

Pick $Z\in\partial\Omega_j$ in the line segment joining $X$ and $Y$ (if $Y\in\partial\Omega_j$ we just take $Z=Y$) so that $|X-Z|,|Y-Z|\le |X-Y|<\gamma$. Using \eqref{claim-equicont}, the fact that $v_j=f$ on $\R^n\setminus \Omega_j$,  and \eqref{f-uc} we obtain
\[
|v_j(X) - v_j(Y)| 
\le
|v_j(X) - v_j(Z)| 
+
|v_j(Z) - v_j(Y)| 
<
\frac{\varrho}2+|f(Z) - f(Y)| 
<\varrho.
\]
	
If we now put all the cases together we have shown that, as desired,  $\{v_j\}$ is equicontinuous.

On the other hand, recalling that $h_j\in W_0^{1,2}(\Omega_j)$ satisfies $L_j h_j=L_j f$ in the weak sense in 
$\Omega_j$, that $f\in W^{1,2}(\R^n)$, and that $\lambda =1$, we see that
\begin{multline*}
 \|\nabla h_j\|_{L^2(\Omega_j)}^2
\le
\int_{\Omega_j} \langle \wcalA_{j}\nabla h_j,\nabla h_j\rangle dX
=
\int_{\Omega_j} \langle \wcalA_{j}\nabla f,\nabla h_j\rangle dX
\le
\Lambda
\|\nabla f\|_{L^2(\Omega_j)}
\|\nabla h_j\|_{L^2(\Omega_j)}.
\end{multline*}
We next absorb the  last term, use that $v_j=f-h_j$ and that $h_j$ has been extended as $0$ outside of $\Omega_j$:
\[
\|\nabla v_j\|_{L^2(\R^n)}
\le
\|\nabla f\|_{L^2(\R^n)}
+
\|\nabla h_j\|_{L^2(\R^n)}
=
\|\nabla f\|_{L^2(\R^n)}
+
\|\nabla h_j\|_{L^2(\Omega_j)}
\le (1+\Lambda)
\|\nabla f\|_{L^2(\R^n)}.
\]
This along with \eqref{eqn:max-pple} yield
\begin{equation}\label{unfi-vj}
 \sup_j \|\nabla v_j \|_{L^2(\RR^n)} \leq  (1+\Lambda) \|\nabla f\|_{L^2(\RR^n)}, \quad \text{and } \sup_j \|v_j \|_{L^2(B(0,R))} \leq C_R. 
\end{equation}
We notice that all these estimates hold for the whole sequence and therefore, so it does for any subsequence.

Let us now fix an arbitrary subsequence $\{v_{j_k}\}_k$. By \eqref{unfi-vj} there are a further subsequence and $v\in C(\R^n)\cap W^{1,2}_{\rm loc}(\RR^n)$ with $\nabla v \in L^2(\RR^n)$, such that $v_{j_{k_l}} \to v$ uniformly on compact sets of $\RR^n$ (hence $v\ge 0$) and $\nabla v_{j_{k_l}}\rightharpoonup \nabla v$ in $L^2(\RR^n)$ as $l\to\infty$. Here it is important to emphasize that the choice of the subsequence may depend on the boundary data $f$ and the fixed subsequence, and the same happens with $v$ , and this could be problematic, later we will see that  this is not the case.

To proceed we next see that $v$ agrees with $f$ in $\partial\Omega_\infty$. Given $p\in \pOinf$, there exist $p_{j_{k_l}} \in \partial\Omega_{j_{k_l}}$ with $p_{j_{k_l}} \to p$ as $l\to\infty$. Using the continuity of $v$ and $f$ at $p$, the uniform convergence of $v_{j_{k_l}} $ to $v$ on  $\overline{B(p,1)}$  and the fact that $v_{j_{k_l}} =f$ on $\partial\Omega_{j_{k_l}} $, we have
\begin{multline*}\label{eqn:105A}
|v(p) - f(p)| 
\leq |v(p) - v(p_{j_{k_l}})| + |v(p_{j_{k_l}}) - v_{j_{k_l}}(p_{j_{k_l}})| +|f(p_{j_{k_l}} ) - f(p)| 
\\
\leq |v(p) - v( p_{j_{k_l}} )| + \|v-v_{j_{k_l}} \|_{L^{\infty}(\overline{B(p,1)})} + |f( p_{j_{k_l}}) - f(p)|
\to 0,
\quad\mbox{as }l\to\infty,
\end{multline*}
thus $v(p)=f(p)$ as desired.

Next, we claim the function $v$ solves the Dirichlet problem \eqref{D:Oinf}. We know that $v\in C(\R^n)$ with $v=f$ in $\partial\Omega_\infty$. Hence, we only need to show that $L_\infty v = 0$ in $\Oinf$. To this aim, let us  take $\psi \in C_c^1(\Oinf)$ and let $R>0$ be large enough so that $\spt \psi\subset B(0,R)$. Since $\overline{\Oj} \to \overline{\Oinf}$, for all $l$ large enough we have that $\psi \in C_c^1(\Omega_{j_{k_l}})$ in which case 
 	\begin{equation}\label{eqn:106A}
		\int_{\R^n} \langle \wcalA_{j_{k_l}}\nabla v_{j_{k_l}},\nabla \psi\rangle dZ= 0,
	\end{equation}
	since $L_{j_{k_l}} v_{j_{k_l}}=0$ in $\Omega_{j_{k_l}}$ in the weak sense. Then, by \eqref{unfi-vj} and the fact that $\spt \psi\subset \Omega_\infty\cap\Omega_{j_{k_l}}\cap B(0,R)$,
	\begin{align*}
	& \left| 	\int_{\R^n} \langle \wcalA^* \nabla v, \nabla \psi \rangle dZ \right|=	
	\left| \int_{\Omega_{j_k}}  \langle \wcalA_{j_k}\nabla v_{j_{k_l}},\nabla \psi\rangle dZ - \int_{\Oinf}  \langle \wcalA^*\nabla v,\nabla \psi\rangle
	dZ \right|  
	\\
	& \qquad \leq  (1+\Lambda)\|\nabla f\|_{L^2(\R^n)}\|\nabla \psi\|_{L^{\infty}} \left( \int_{\Oj\cap B(0,R)} |\wcalA_{j_{k_l}} - \wcalA^* |^2 dZ \right)^{\frac{1}{2}}
	\nonumber \\
	& \qquad \qquad \qquad + \left| \int_{\R^n} \langle \wcalA^* \nabla v_{j_{k_l}}, \nabla \psi \rangle dZ - \int_{\R^n}  \langle \wcalA^* \nabla v, \nabla \psi \rangle dZ \right|\to 0,\quad\mbox{as }l\to\infty,
	\end{align*}
where we have used \eqref{eq:temp118} with $p=2$ for the term in the second line, and the fact that since $\wcalA^*$ is a constant-coefficient matrix, it follows that $\wcalA^* \nabla v_{j_{k_l}} \rightharpoonup \wcalA^* \nabla v$ in $L^2(\R^n)$ as $l\to\infty$. This eventually shows that $L_\infty v=0$ in $\Omega_\infty$.

In \ref{CaseII} when the domain $\Oinf$ is bounded, the Dirichlet problem \eqref{D:Oinf} has a unique solution, and it satisfies the maximum principle, hence we must have that $v=v_\infty$. Therefore, we have  shown that given any subsequence $\{v_{j_k}\}_k$ there is a further subsequence $\{v_{j_{k_l}}\}_l$ so that $v_{j_{k_l}}\to v_\infty$ uniformly on compact sets of $\RR^n$ and $\nabla v_{j_{k_l}} \rightharpoonup \nabla v_\infty$ in $L^2(\RR^n)$ as $l\to\infty$. This eventually shows that entire sequence $\{v_j\}$ satisfies $v_j\to v_\infty$ uniformly on compact sets of $\RR^n$ and $\nabla v_{j} \rightharpoonup \nabla v_\infty$ in $L^2(\RR^n)$ as $j\to\infty$.

In \ref{CaseI} where the limiting domain $\Oinf$ is unbounded, we need more work to show the solution $v$ is indeed $v_\infty$. 
Recall that $f\in C_c(\R^n)$ with $\spt f\subset B(0,R_0)$. Given $\epsilon>0$, there is an integer $j_0 = j_0(\epsilon, R_0)\in \NN$ such that for $j\geq j_{0}$
and for any $p_j' \in \pOj \cap B(0,4R_0)$, there is $p' \in \pOinf \cap B(0,5R_0) $ close enough to $p_j'$ so that $|f(p') -f(p'_j)|<\epsilon$. Consequently, 
	\begin{equation}\label{eqn:106C}
	\sup_{\pOj} |f| = \sup_{\pOj \cap B(0,4R_0)} |f| \leq \sup_{\pOinf \cap  B(0,5R_0)} |f| + \epsilon = \sup_{\pOinf} |f| + \epsilon.  
	\end{equation}

For any $Z\in \Oinf$ there exists a sequence $Z_j \in \Oj$ such that $Z_j \to Z$ and $Z_j\in \overline{B(Z,\delta_\infty(Z)/2)}$ for all $j$ large enough. Since $v\in C(\R^n)$ it follows that for $j$ large enough $|v(Z)-v(Z_j)|<\epsilon$.
All these together with \eqref{eqn:max-pple}	and the fact that $v_{j_{k_l}} \to v$ uniformly on compact sets of $\RR^n$ as $l\to\infty$ give that for all $l$ large enough 
	\begin{equation}\label{eqn:106D}
0\le v(Z) \leq |v(Z) - v(Z_{j_{k_l}} )| + |v(Z_{j_{k_l}} ) - v_{j_{k_l}} (Z_{j_{k_l}} )| + |v_{j_{k_l}} (Z_{j_{k_l}} )| \leq 2\epsilon + \sup_{\pOj}|f|\le 3\epsilon+ \sup_{\pOinf} |f|,
\end{equation}
Letting $\epsilon\to 0$ we  get $0\le \sup_{\Oinf}v \leq \sup_{\pOinf} |f|$.

Let us recall that $\Omega_R=\Omega_\infty\cap B(0,2R)\subset \Omega_\infty$.  Since $v\in C(\R^n)$ with 
$v|_{\partial\Omega_\infty}=f$, and since $\spt f\subset B(0,R_0)$, for every $R>4R_0$ we have that $f_R|_{\partial\Omega_\infty}=f \eta(\cdot/R)\le v|_{\partial\Omega_\infty}$. Hence the maximum principle implies that $v_R \leq v$ in $\Omega_R$, and taking limits  we conclude that $ v_\infty \leq v$ on $\Oinf$. Write $0\le \widetilde{v}=v- v_\infty\in C(\overline{\Omega_\infty})$ so that $L_{\infty} \widetilde{v}=0$ in $\Oinf$ and $\widetilde{v}|_{\partial\Omega_\infty}=0$. For any $Z\in \Oinf$, since $\Oinf$ is a uniform domain with Ahlfors regular boundary, by Lemma \ref{lem:vanishing}
	  for any $\delta_\infty(Z)<R' < \diam(\pOinf)=\infty$ (see \ref{3-thm:pseudo-blow-geo} in Theorem \ref{thm:pseudo-blow-geo}) 
\begin{equation}\label{eqn:Holderinfty}
		0\le \widetilde{v}(Z) \lesssim \left( \frac{\delta_{\infty}(Z)}{R'} \right)^{\beta} \sup_{\Oinf} \widetilde{v} \leq 2 \left( \frac{\delta_{\infty}(Z)}{R'} \right)^{\beta} \sup_{\pOinf} f,    	\end{equation}
Letting $R'\to\infty$	we conclude that $\widetilde{v}(Z) = 0$ and hence $v=v_\infty$. Therefore, we have  shown that given a subsequence $\{v_{j_k}\}_k$ there is a further subsequence $\{v_{j_{k_l}}\}_l$ so that $v_{j_{k_l}}\to v_\infty$ uniformly on compact sets of $\RR^n$ and $\nabla v_{j_{k_l}} \rightharpoonup \nabla v_\infty$ in $L^2(\RR^n)$ as $l\to\infty$. This eventually shows that entire sequence $\{v_j\}$ satisfies $v_j\to v_\infty$ uniformly on compact sets of $\RR^n$ and $\nabla v_{j} \rightharpoonup \nabla v_\infty$ in $L^2(\RR^n)$ as $j\to\infty$.

Hence, in both \ref{CaseI} and \ref{CaseII}, if $0\le f\in \Lip(\partial\Omega_\infty)$ has compact support  then 
\begin{equation}\label{eqn:110A}
\lim_{j\to\infty} \int_{\pOj} f(q) d\oj^Z(q) =\lim_{j\to\infty} v_j(Z) = v_\infty(Z)=\int_{\pOinf} f(q) d\omega_{L_\infty}^Z(q),
\end{equation}
for any $Z\in\Omega_\infty$. A standard approximation argument and splitting each function on its positive and negative parts lead to shows that \eqref{eqn:110A} holds for all $f\in C_c(\RR^n)$, hence 
$\omega_j^Z \rightharpoonup \omega_{L_\infty}^Z$ as Radon measures for any $Z\in\Oinf$.

\medskip

Our next goal is to see that $\omega_{L_\infty}\in A_{\infty}(\sigma_{\infty})$ (where $\sigma_{\infty} = \mathcal{H}^{n-1}|_{\pOinf}$). Fix $p\in \pOinf$ and $0<r<\diam(\pOinf)$. Recall that whether $\diam(\pOinf)$ is finite or infinite, we always have $r< \diam(\Oj)$ for all $j$ sufficiently large. Let $\Delta' = B(m,s)\cap \pOinf$ with $m\in\pOinf$ and $B(m,s)\subset B(p,r) \cap \pOinf$. 
Let $A(p,r)\in \Omega_\infty$ be a corkscrew point relative to $\Delta(p,r)$ (whose existence is guaranteed by \ref{4-thm:pseudo-blow-geo} in Theorem \ref{thm:pseudo-blow-geo}). We can then find $p_j \in \pOj$ such that $p_j \to p$. Thus, for all $j$ large enough $B(p,r)\subset B(p_j,2r)$ and $\delta_j(A(p,r))\ge r/(2M)$. Hence, 
$A(p,r) $ is also a corkscrew point relative to $B(p_j,2r)\cap\partial\Omega_j$ in $\Oj$ with constant $4M$.  Since $m\in\pOinf$, we can also find $m_j\in\pOj$ such that $m_j\to m$. In particular, for $j$ sufficiently large
	\begin{equation}\label{eq:mjm}
		|m_j - m| < \frac{s}{5}.
	\end{equation}
	Note also that since all the $\Oj$'s are uniform and satisfy the CDC with the same constants, and all the operators $L_j$'s have ellipticity constants bounded below and above by $\lambda=1$ and $\Lambda$, we can conclude from Remark \ref{rem:doubling:needed} that there is a uniform constant $C_2$ depending on 
	$M, C_1, C_{AR}>1$, and
$\Lambda$, 
such that \eqref{doubling:needed} holds for all $\omega_j$ with the appropriate changes. Using this and  \cite[Theorem 1.24]{Ma} we obtain 
		\begin{multline}\label{eqn:113A}
			\oinf^{A(p,r)} (\Delta(m,s)) \geq \oinf^{A(p,r)} \left( \overline{B\left(m,\frac{4}{5}s\right)} \right)  \geq \limsup_{j\to\infty} \oj^{A(p,r)} \left( \overline{B\left(m,\frac{4}{5}s\right)} \right) 
			\\
			\geq \limsup_{j\to\infty} \oj^{A(p,r)} \left( \overline{B\left(m_j,\frac{3}{5}s\right)} \right) 
			\geq C_2^{-1} \limsup_{j\to\infty} \oj^{A(p,r)} \left(B\left(m_j,\frac{6}{5}s\right)\right),
		\end{multline}
		where we have used that $\delta_j(A(p,r))\ge r/(2M)\ge \frac35s/(2M)$.

Let $V$ be an arbitrary open set in $B(m,s)$, and note that by \eqref{eq:mjm}
$$
			V\subset B(m,s) \subset B\left(m_j, \frac{6}{5}s \right).
$$
	 Using again \cite[Theorem 1.24]{Ma}, we see that \eqref{eqn:113A} yields
		\begin{multline}
			\frac{\oinf^{A(p,r)}(V)}{\oinf^{A(p,r)}(\Delta(m,s))} 
			\leq 
			C_2 \dfrac{\liminf_{j\to\infty} \ojA(V)}{\limsup_{j\to\infty} \ojA\left(B\left(m_j,\frac{6}{5}s\right)\right)} 
			\\
			  \leq C_2 \liminf_{j\to\infty} \left( \dfrac{\ojA(V) }{\ojA\left(B\left(m_j,\frac{6}{5}s\right)\right) }\right). 
			\label{eqn:114A}
		\end{multline}
		The assumption $B(m,s) \subset B(p,r)$ implies $|m-p|\leq r-s$. Using this and that $m_j\to m$, $p_j\to p$ as $j\to\infty$ one can easily see that 
		$|m_j - p_j|<r-\frac{s}{5}$ for all $j$ large enough and hence
		\begin{equation}\label{tempincl}
			B\left(m_j, \frac{6}{5}s \right)\cap\partial\Omega_j \subset B(p_j,2r)\cap\partial\Omega_j.
		\end{equation}
		As mentioned above $A(p,r)$ is a corkscrew point relative to $B(p_j,2r)\cap\partial\Omega_j$ in $\Oj$. This, \eqref{tempincl} and the fact that by assumption, $\omega_j\in A_\infty(\sigma_j)$ with uniform constants $C_0, \theta$ allow us to conclude that
		\begin{equation}\label{eq:Ainftyoj}
			\dfrac{\ojA(V) }{\ojA\left(B\left(m_j,\frac{6}{5}s\right)\right) } \leq C_0 \left( \dfrac{\sj(V)}{\sj\left(B\left(m_j, \frac{6}{5}s \right) \right)} \right)^\theta
			\le C_0 C_{AR}^\theta \left( \dfrac{\sj(V)}{s^{n-1}} \right)^\theta,
		\end{equation}
		where in the last estimate we have used that $\partial\Omega_j$ is Ahlfors regular with constants $C_{AR}$.
		Combining \eqref{eqn:114A}, \eqref{eq:Ainftyoj}, the fact that $\sj \rightharpoonup \sinf$, \cite[Theorem 1.24]{Ma}, and \ref{5-thm:pseudo-blow-geo} in Theorem \ref{thm:pseudo-blow-geo}, we finally arrive at 
		\begin{multline*}
			\frac{\oinf^{A(p,r)}(V)}{\oinf^{A(p,r)}(\Delta(m,s))} 
			\leq C_0 C_{AR}^\theta \left(\liminf_{j\to\infty} \dfrac{\sj(V)}{s^{n-1}} \right)^\theta 
			\\
			\leq C_0 C_{AR}^\theta \left( \frac{\sinf(\overline V)}{s^{n-1}} \right)^\theta 
			\le 
			C_0 C_{AR}^{4\theta}  2^{8(n-1)\theta} \left( \frac{\sigma_\infty(\overline V)}{\sigma_\infty(\Delta(m,s))} \right)^\theta.
		\end{multline*}
and therefore we have shown that  for any open set $V\subset B(m,s)$ there holds 
\begin{equation}\label{eq:Ainftyoinf}
			\frac{\oinf^{A(p,r)}(V)}{\oinf^{A(p,r)}(\Delta(m,s))} 
			\le 
			C_0 C_{AR}^{4\theta}  2^{8(n-1)\theta} \left( \frac{\sigma_\infty(\overline V)}{\sigma_\infty(\Delta(m,s))} \right)^\theta.			
\end{equation}

Consider next an arbitrary Borel set $E\subset B(m,s)$. Since $\sigma_{\infty}$ and  $\oinf^{A(p,r)}$ are
		Borel regular, given any $\epsilon>0$ there is an open set $U$ and a compact set $F$ so that  $F\subset E\subset U \subset B(m,s)$ and $\oinf^{A(p,r)}(U\setminus F)+ \sigma_{\infty}(U\setminus F) <\epsilon$. Note that for any $x\in F$, there is $r_x >0 $ such that $B(x,2r_x) \subset U$. Using that $F$ is compact we can then show there exists a finite collection of points $\{x_i\}_{i=1}^m\subset F$ such that $F\subset \bigcup_{i=1}^m B(x_i, r_i) =:  V$ and $B(x_i, 2r_i) \subset U$ for $i= 1,\dots,m$. Consequently, $F\subset V \subset \overline V \subset U$ and $\sigma_{\infty}(\overline V \setminus F) \leq \sigma_{\infty}(U\setminus F) <\epsilon$. We next use \eqref{eq:Ainftyoinf} with $V$ to see that
\begin{multline*}
\frac{\oinf^{A(p,r)}(E)}{\oinf^{A(p,r)}(\Delta(m,s))} 
\leq 
\frac{\epsilon +\oinf^{A(p,r)}(F)}{\oinf^{A(p,r)}(\Delta(m,s))} 
\leq 
\frac{\epsilon +\oinf^{A(p,r)}(V)}{\oinf^{A(p,r)}(\Delta(m,s))} 
\\
\leq 
\frac{\epsilon}{\oinf^{A(p,r)}(\Delta(m,s))} 
+
C_0 C_{AR}^{4\theta}  2^{8(n-1)\theta} \left( \frac{\sigma_{\infty}(\overline V)}{\sigma_{\infty}(\Delta(m,s))} \right)^\theta 
\\
\leq 
\frac{\epsilon}{\oinf^{A(p,r)}(\Delta(m,s))} 
+
C_0 C_{AR}^{4\theta}  2^{8(n-1)\theta} \left( \frac{\sigma_{\infty}(E)+\epsilon}{\sigma_{\infty}(\Delta(m,s))} \right)^\theta. \label{eqn:119A}
\end{multline*}
Letting $\epsilon \to 0$ we obtain  as desired that $\oinf\in A_\infty(\sigma_\infty)$ with constants $C_0 C_{AR}^{4\theta}  2^{8(n-1)\theta} $ and $\theta$ and the proof is  complete. 
\end{proof}

\section{Proof of Theorem~\ref{thm:main:I}}\label{section:proof-main}

Applying Theorem \ref{thm:pseudo-blow-geo}, we obtain that $\Omega_\infty$ is a uniform domain with constants $4M$ and $2C_1$, whose boundary is Ahlfors regular with constant $2^{5(n-1)}C_{AR}^2$. Moreover, Theorem \ref{thm:blow-ana-pole} gives that $\omega_{L_{\infty}}\in A_\infty(\sigma_\infty)$ with constants 
$\widetilde{C}_0=C_2 C_{AR}^{4\theta} 2^{8(n-1)\theta}$ and $\widetilde{\theta}=\theta$. 
Here $L_\infty=-\divg(\mathcal{A}^* \nabla)$ with $\mathcal{A}^*$  a constant-coefficient real
 symmetric uniformly elliptic matrix with ellipticity constants 
$1=\lambda\le \Lambda<\infty$.
We can then invoke Theorem \ref{thm:hmu}, 
to see that $\Oinf$ satisfies the exterior corkscrew condition with 
constant
\[N_0=N_0(4M,2C_1, 2^{5(n-1)}C_{AR}^2,\Lambda, C_0  C_2 C_{AR}^{4\theta} 2^{8(n-1)\theta},\theta)\] 
(see introduction to Section \ref{comp-tt}). Therefore,  since $0\in\pOinf$,  $0<\frac12<\diam(\pOinf)$ (recall that $\diam(\pOinf) = \infty$ in \ref{CaseI}, and $\diam(\pOinf) = \diam(\Oinf) = R_0 \geq 1$) there exists $A_0=A^-(0,\frac12)$ so that 
\begin{equation}\label{eqn:201A}
	B\left( A_0,\frac{1}{2N_0} \right) \subset B\left(0,\frac12\right)\setminus\overline{\Oinf}.
\end{equation}
Hence
\begin{equation}
	\dist\left(B\left(A_0,\frac{1}{4N_0}\right), \R^n\setminus \overline{\Oinf}\right) 
	\geq \frac{1}{4N_0}.
\end{equation} 
Since $\overline{\Oj} \to \overline{\Oinf}$, it follows that for all $j$ large enough
\begin{equation}\label{eqn:203A}
	B\left( A_0,\frac{1}{4N_0} \right) 
	\subset  B\left(0,\frac12\right)\setminus\overline{\Oj} \subset B(0,1)\setminus\overline{\Oj}.
\end{equation}
Hence for all $j$ large enough $A_0$ is a corkscrew point relative to $B(0,1)\cap\partial\Omega_j$ for $\R^n\setminus \overline{\Oj}$ with constant $4N_0$. This contradicts our assumption that  $\Omega_j$ has no exterior corkscrew point with constant $N=4N_0$ for the surface ball $B(0,1)\cap\partial\Omega_j$ and the proof is complete.

\part{Large constant case}\label{part:large}

The extrapolation argument to augment small Carleson norm is in essence an inductive process combined with delicate stopping time arguments. It requires us to construct the so-called \textit{sawtooth domains} on which we have better control of the Carleson measure norm. We explain the construction of these sawtooth domains and auxiliary definitions in Section \ref{section:sawtooth}, to set the stage for the extrapolation argument.
In Section \ref{S:proof-by-extrapolation} 
we first state several key ingredients to be used in the proof, namely: the framework of the extrapolation argument in Theorem \ref{thm:extrapolation} proved in  \cite{HMM}; and Theorem \ref{thm:bhc}, a criterion for uniform rectifiability
proven as a combination of \cite{HMM}, \cite{GMT}. 
We then outline the proof of Theorem \ref{thm:main} using these ingredients and the small constant result proven in Part \ref{part:small} (i.e. Corollary \ref{corol:DKP-small}), and in the process reduce matters to two main steps, which are then carried out in two separate sections.
In Section \ref{s:small-extrapolation}, we prove a technical estimate showing that
a continuous parameter Carleson measure, restricted to a sawtooth subdomain, 
may be controlled quantitatively by a discretized version of itself.
Section \ref{sect:trsf} contains the most delicate technical part of the proof, involving transference of the
$A_\infty$ property to sawtooth subdomains. Finally in Section \ref{optimal} we discuss the optimality of Theorem \ref{thm:main} and present an important corollary.

\section{Construction of sawtooth domains and Discrete Carleson measures}\label{section:sawtooth}

\begin{lemma}[Dyadic decomposition of Ahlfors regular set, \cite{DS1, DS2, Ch}]\label{lm:ddAR}
	Let $E\subset \RR^n$ be an Ahlfors regular set. Then there exist constants $a_0, A_1, \gamma>0$, depending only on $n$ and the constants of Ahlfors regularity, such that for each $k\in\mathbb{Z}$, there is a collection of Borel sets (``dyadic cubes'')
	\[ \mathbb{D}_k : = \{Q_j^k \subset E: j\in \mathscr{J}_k \}, \]
	where $\mathscr{J}_k$ denotes some index set depending on $k$, satisfying the following properties.
\begin{enumerate}[label= \textup{(\roman*)}, itemsep=0.2cm] 		
		\item\label{1-lm:ddAR} $E = \bigcup_{j\in\mathscr{J}_k} Q_j^k$ for each $k\in\mathbb{Z}$.
		
		\item\label{2-lm:ddAR} If $m\geq k$ then either $Q_i^m \subset Q_j^k$ or $Q_i^m \cap Q_j^k = \emptyset$.
		
		\item\label{3-lm:ddAR} For each pair $(j,k)$ and each $m<k$, there is a unique $i\in\mathscr{J}_m$ such that $Q_j^k \subset Q_i^m$.
		
		\item\label{4-lm:ddAR} $\diam Q_j^k \leq A_1 2^{-k}$.
		
		\item\label{5-lm:ddAR} Each $Q_j^k$ contains some surface ball $\Delta(x_j^k, a_0 2^{-k}) := B(x_j^k, a_0 2^{-k}) \cap E$.
		
		\item\label{6-lm:ddAR} For all $(j, k)$ and all $\rho\in (0,1)$ 
		\begin{multline}\label{thin-boundary}
		\mathcal{H}^{n-1} \left( \left\{ q\in Q_j^k: \dist(q, E\setminus Q_j^k) \leq \rho 2^{-k} \right\} \right)
		\\
		 + \mathcal{H}^{n-1} \left( \left\{ q\in E\setminus Q_j^k: \dist(q, Q_j^k) \leq \rho 2^{-k} \right\} \right) \leq A_1 \rho^{\gamma} \mathcal{H}^{n-1} (Q_j^k).
\end{multline}
\end{enumerate}
	We shall denote by $\mathbb{D} = \mathbb{D}(E)$ the collection of all relevant $Q_j^k$, i.e.,
	\begin{equation}\label{all-cubes}
	\DD = \bigcup_k \DD_k, 
	\end{equation}
	where, if $\diam (E)$ is finite, the union runs 	over those $k$ such that $2^{-k} \lesssim  {\rm diam}(E)$.
\end{lemma}

\begin{remark}
	For a dyadic cube $Q\in \mathbb{D}_k$, we shall
set $\ell(Q) = 2^{-k}$, and we shall refer to this quantity as the ``length''
of $Q$.  Evidently, $\ell(Q)\approx \diam(Q).$	We will also write $x_Q$ for the ``center'' of $Q$, that is, the center of the ball appearing in \ref{5-lm:ddAR}. 
\end{remark}

Assume from now on that $\Omega$ is a uniform domain with Ahlfors regular boundary and set $\sigma=\mathcal{H}^{n-1}|_{\pO}$. Let $\dd=\dd(\pom)$ be the associated dyadic grid from the previous result. We first make a simple observation:

\begin{corollary}[Doubling property of the kernel]\label{cor:doublingPk}
		Let $Q\in \DD$ be a dyadic cube, and $\widetilde{Q}\in\dd$ be such that $C_1^{-1}\ell(Q)\le \ell(\widetilde{Q})\le C_1\ell(Q)$ and $\dist(Q,\widetilde{Q})\le C_1\ell(Q)$ for some $C_1\ge 1$. Suppose $\omega\in RH_p(\sigma)$ for some $p>1$, then for $X_Q$ the corkscrew relative to $Q$ we have 
		\begin{equation}
			\int_{\widetilde{Q}} \left( \textup{\textbf{k}}^{X_Q} \right)^p d\sigma \leq C \int_Q \left( \textup{\textbf{k}}^{X_Q} \right)^p d\sigma,
		\end{equation}
		with a constant $C$ depending on $C_1$ and the allowable constants and where $\textup{\textbf{k}}=d\omega/d\sigma$.
\end{corollary}
The proof is a simple corollary of the doubling property of the elliptic measure (see Lemma \ref{lm:doubling}):
		\begin{align*}
			\Big(\fint_{\widetilde Q} \left(\textbf{k}^{X_Q} \right)^p d\sigma\Big)^\frac1p
			& \lesssim 
			\frac{\omega^{X_Q}(\widetilde Q)}{\sigma(\widetilde Q)}
			\approx 
			\frac{\omega^{X_Q}( Q)}{\sigma( Q)}
			=
			\fint_{Q} \textbf{k}^{X_Q} d\sigma 
			\le
			\Big(\fint_{Q} \left(\textbf{k}^{X_Q} \right)^p d\sigma\Big)^\frac1p. 
		\end{align*}

Let $\mathcal{W}=\W(\Omega)$ denote a collection
of (closed) dyadic Whitney cubes   of $\Omega$ (just dyadically divide the standard Whitney cubes from \cite[Chapter VI]{St} into cubes with side length 1/8 as large),  so that the boxes  in $\mathcal{W}$
form a covering of $\Omega$ with non-overlapping interiors, and  which satisfy
\begin{equation}\label{eqWh1} 4\, {\rm{diam}}\,(I)\leq \dist(4 I,\pom) \leq  \dist(I,\pom) \leq 40 \, {\rm{diam}}\,(I).
\end{equation}
Let $X(I)$ denote the center of $I$, let $\ell(I)$ denote the side length of $I$,
and write $k=k_I$ if $\ell(I) = 2^{-k}$. We will use ``boxes'' to refer to the Whitney cubes as just constructed, and ``cubes'' for the dyadic cubes on $\pO$. 
Then for each pair $I, J \in \W$,
\begin{equation}\label{eq:Iequivsize}
\text{ if } I\cap J \neq \emptyset, \text{ then } 4^{-1} \leq \frac{\ell(I)}{\ell(J)} \leq 4.
\end{equation}
Since $I, J$ are dyadic boxes, then $I\cap J$ is either contained in a face of $I$, or contained in a face of $J$. 
By choosing $\tau_0<2^{-10}$  sufficiently small (depending on $n$), we may also suppose that there is $t\in(\frac12,1)$ so that if $0<\tau<\tau_0$, for every distinct pair $I, J\in \W(\Omega)$,
\begin{equation}\label{eq:tausmall1}
(1+4\tau) I \cap (1+4\tau) J \neq \emptyset \iff I\cap J \neq \emptyset;
\end{equation}
and
\begin{equation}\label{eq:tausmall2}
	t J \cap (1+ 4\,\tau) I = \emptyset.
\end{equation}
Also, $J\cap (1+\tau )I$ contains an $(n-1)$-dimensional cube with side length of the order of $\min\{\ell(I), \ell(J)\}$. This observation will become useful in Section \ref{sect:trsf}.  For such $\tau\in (0,\tau_0)$ fixed, we write $I^*=(1+\tau)I$, $I^{**}= (1+2\tau) I$, and $I^{**}= (1+4\tau) I$ for the ``fattening'' of $I\in \W$.

Following \cite[Section 3]{HM} we next introduce the notion of \textit{\bf Carleson region} and \textit{\bf discretized sawtooth}.  Given a cube $Q\in\dd$, the \textit{\bf discretized Carleson region $\dd_{Q}$} relative to $Q$ is defined by
\[
\dd_{Q}=\{Q'\in\dd:\, \, Q'\subset Q\}.
\]
Let $\F$ be family of disjoint cubes $\{Q_{j}\}\subset\dd$. The \textit{\bf global discretized sawtooth region} relative to $\F$ is the collection of cubes $Q\in\dd$  that are not contained in any $Q_{j}\in\F$;
\[
\dd_{\F}:=\dd\setminus \bigcup\limits_{Q_{j}\in\F}\dd_{Q_{j}}.
\]
For a given $Q\in\dd$ the {\bf local discretized sawtooth region} relative to $\F$ is the collection of cubes in $\dd_{Q}$ that are not in contained in any $Q_{j}\in\F$;
\begin{equation}\label{dfq}
\dd_{\F,Q}:=\dd_{Q}\setminus \bigcup\limits_{Q_{j}\in \F} \dd_{Q_{j}}=\dd_{\F}\cap \dd_{Q}.
\end{equation}
We also introduce the ``geometric'' Carleson and sawtooth regions. For any dyadic cube $Q\in\DD$, pick two parameters $\eta\ll 1$ and $K\gg 1$, and define
\begin{equation}\label{def:WQ0}
	\WW_Q^0 : = \{I\in\WW: \eta^{\frac{1}{4}} \ell(Q) \leq \ell(I) \leq K^{\frac{1}{2}} \ell(Q),\  \dist(I,Q) \leq K^{\frac{1}{2}} \ell(Q)\}.
\end{equation} 
Taking $K\ge 40^2 n$, if $I\in\W$ and we pick $Q_I\in \dd$ so that $\ell(Q_I)=\ell(I)$ and $\dist(I,\pom)=\dist(I,Q_I)$, then $I\in \W^0_{Q_I}$. 
Let $X_Q$ denote a corkscrew point for the surface ball $\Delta(x_Q, r_Q/2)$. We can guarantee that $X_Q $ is in some $I\in\WW_Q^0$ provided we choose $\eta$ small enough and $K$ large enough. For each $I\in\WW_Q^0$, there is a Harnack chain connecting $X(I)$ to $X_Q$, we call it $\mathcal{H}_I$. By the definition of $\WW_Q^0$ we may construct this Harnack chain so that it consists of a bounded number of balls (depending on the values of $\eta, K$). We let $\WW_Q$ denote the set of all $J\in\WW$ which meet at least one of the Harnack chains $\mathcal{H}_I$, with $I\in \WW_Q^0$, i.e.
\begin{equation}
	\WW_Q := \{J\in\WW: \text{there exists } I\in\WW_Q^0 \text{ for which } \mathcal{H}_I \cap J \neq \emptyset\}.
\end{equation}
Clearly $\WW_Q^0 \subset \WW_Q$. Besides, it follows from the construction of the augmented collections $\WW_Q$ and the properties of the Harnack chains that there are uniform constants $c$ and $C$ such that
\begin{equation}\label{def:WQ}
	c\eta^{\frac12 } \ell(Q) \leq \ell(I) \leq CK^{\frac{1}{2}} \ell(Q), \quad \dist(I,Q) \leq CK^{\frac{1}{2}} \ell(Q) 
\end{equation}
for any $I\in \WW_Q$. In particular once $\eta, K$ are fixed, for any $Q\in \DD$ the cardinality of $\WW_Q$ is uniformly bounded.
Finally, for every $Q$ we define its associated Whitney region
\begin{equation}\label{eq2.whitney3}
U_Q := \bigcup_{I\in\,\mathcal{W}_Q} I^*.
\end{equation}

We refer the reader to \cite[Section 3]{HM} or \cite[Section 2]{HMM2} for additional details.

For a given $Q\in\dd$, the {\bf Carleson box} relative to $Q$ is defined by
\begin{equation}\label{carleson-box}
T_{Q}:=\mbox{int}\left(\bigcup\limits_{Q'\in\dd_{Q}} U_{Q'}\right).
\end{equation}
For a given family $\F$ of disjoint cubes $\{Q_{j}\}\subset\dd$ and a given $Q\in\dd$ we define the {\bf local sawtooth region} relative to $\F$ by
\[
\Omega_{\F,Q}:=\mbox{int}\left(\bigcup\limits_{Q'\in\dd_{\F,Q}}U_{Q'}\right)
=
{\rm int }\,\left(\bigcup_{I\in\,\W_{\F,Q}} I^*\right),
\]
where $\W_{\F,Q}:=\bigcup_{Q'\in\dd_{\F,Q}}\W^*_{Q'}$.
Analogously, we can slightly fatten the Whitney boxes and use $I^{**}$ to define new fattened Whitney regions and sawtooth domains. More precisely,
\begin{equation}
T_{Q}^*:=\mbox{int}\left(\bigcup\limits_{Q'\in\dd_{Q}} U_{Q'}^*\right),
\qquad
\Omega_{\F,Q}^*:=\mbox{int}\left(\bigcup\limits_{Q'\in\dd_{\F,Q}}U_{Q'}^*\right)
,
\qquad
U_{Q'}^*:=\bigcup_{I\in\,\mathcal{W}^*_{Q'}} I^{**}.
\label{eq:fatten-objects}
\end{equation}
Similarly, we can define $T_Q^{**}$, $\Omega_{\F,Q}^{**}$ and $U_{Q}^{**}$ by using $I^{***}$ in place of $I^{**}$.

One can easily see that there is a constant $\kappa_0>0$ (depending only on the allowable parameters, $\eta$, and $K$)  so that 
\begin{equation}\label{tent-Q}
T_Q\subset T_{Q}^*\subset T_{Q}^{**}\subset\overline{T_Q^{**}}\subset \kappa_0 B_Q\cap\overline{\Omega}
=: B_Q^*\cap \overline{\Omega}, \qquad\forall\,Q\in\dd.
\end{equation}

Given a pairwise disjoint family $\F\subset\dd$ (we also allow $\F$ to be the null set) and a constant $\rho>0$, we derive another family $\F({\rho})\subset\dd$  from $\F$ as follows. Augment $\F$ by adding cubes $Q\in\dd$ whose side length $\ell(Q)\leq \rho$ and let $\F(\rho)$ denote the corresponding collection of maximal cubes with respect to the inclusion. Note that the corresponding discrete sawtooth region $\dd_{\F(\rho)}$ is the union of all cubes $Q\in\dd_{\F}$ such that $\ell(Q)>\rho$.
For a given constant $\rho$ and a cube $Q\in \dd$, let $\dd_{\F(\rho),Q}$ denote the local discrete sawtooth region and let $\Omega_{\F(\rho),Q}$ denote the local geometric sawtooth region relative to disjoint family $\F(\rho)$.

Given $Q\in\dd$ and $0<\epsilon<1$, if we take $\F_0=\emptyset$, one has that \
$\F_0(\epsilon \,\ell(Q))$ is the collection of $Q'\in \dd$ such that $\epsilon\,\ell(Q)/2<\ell(Q')\le \epsilon\,\ell(Q)$. We then introduce $U_{Q,\epsilon}=\Omega_{\F_0(\epsilon \,\ell(Q)),Q}$, which is a Whitney region relative to $Q$ whose distance to $\pom$ is of the order of $\epsilon\,\ell(Q)$.  For later use, we observe that given $Q_0\in\dd$, the sets $\{U_{Q,\epsilon}\}_{Q\in\dd_{Q_0}}$ have bounded overlap with constant that may depend on $\epsilon$. Indeed, suppose that there is $X\in U_{Q,\epsilon}\cap U_{Q',\epsilon}$ with $Q$, $Q'\in \dd_{Q_0}$. By construction $\ell(Q)\approx_\epsilon \delta(X)\approx_\epsilon \ell(Q')$ and
$$
\dist(Q,Q')\le \dist(X,Q)+\dist(X,Q')\lesssim_\epsilon \ell(Q)+\ell(Q')\approx_\epsilon \ell(Q).
$$
The bounded overlap property follows then at once.

\begin{lemma}[{\cite[Lemma 3.61]{HM}}]\label{lemma:sawtooth-inherit}
	Let $\Omega\subset \RR^n$ be a uniform domain with Ahlfors regular boundary. Then all of its Carleson boxes $T_Q$ and sawtooth domains $\Omega_{\F, Q}$, $\Omega_{\F, Q}^*$  are uniform domains with Ahlfors regular boundaries. In all the cases the implicit constants are uniform, and depend only on
	dimension and on the corresponding constants for $\Omega$. 
\end{lemma}

We say that  $\PP$ is a \textbf{fundamental chord-arc subdomain} of $\Omega$ if there is $I\in \W$ and $m_1$ such that 
\begin{equation}\label{fundamental-casd}
\PP=\interior\left( \bigcup_{j=1}^{m_1} I_j^\ast\right) \hbox{ where } I_j\in\W \hbox{  and  } I\cap I_j \not =\emptyset.
\end{equation}
Note that the fact that $I\cap I_j \not =\emptyset$ ensures that $\ell(I)\approx \ell(I_j)$.  Moreover $\PP$ is a chord-arc domain with constants that only depend on $n$, $\tau$ and the constants used in the  construction of $\DD$ and $\W$ (see \cite[Lemma 2.47]{HMU} for a similar argument).

Given a sequence of non-negative numbers 
$\alpha=\{\alpha_Q\}_{Q\in\dd}$ we define the associated discrete ``measure'' $m=m_\alpha$: 
\begin{equation}\label{def:mQ}
	m(\DD') := \sum_{Q\in \DD'} \alpha_Q,
	\qquad \dd'\subset\dd
\end{equation}

\begin{defn}\label{discrete-carleson}
 Let $E\subset \RR^n$ be an Ahlfors regular set, and let  $\sigma$ be a dyadically doubling Borel measure on $E$ (not necessarily equal to $\mathcal{H}^{n-1}|_{\pO}$). We say that $m$ as defined in \eqref{def:mQ} is a discrete Carleson measure with respect to $\sigma$, if 
	\begin{equation}\label{mCarleson}
		\|m\|_{\C}: = \sup_{Q\in \DD} \frac{ m(\DD_Q)}{ \sigma(Q) } <\infty.
	\end{equation}
Also, fixed $Q_0\in\dd$ we say that $m$ is a discrete Carleson measure with respect to $\sigma$ in $Q_0$ if 
\begin{equation}\label{mCarleson-local}
\|m\|_{\C(Q_0)}: = \sup_{Q\in \DD_{Q_0}} \frac{ m(\DD_Q)}{ \sigma(Q) } <\infty.
\end{equation}
\end{defn}

\section{Proof by extrapolation}\label{S:proof-by-extrapolation}

In this section we present some powerful tools which will be key in the proof of our main result. After that we outline the roadmap to pass from the small constant case to the large constant case and thus finish proving Theorem \ref{thm:main}.

We start with \cite[Lemma 4.5]{HMM}, an extrapolation for Carleson measure result which in a nutshell describes how the relationship between a discrete Carleson measure
$m$ and another discrete measure $\widetilde m$ yields information about $\widetilde m$.

\begin{theorem}[{Extrapolation, \cite[Lemma 4.5]{HMM}}]\label{thm:extrapolation}
	Let $\sigma$ be a dyadically doubling Borel measure on $\pO$ (not necessarily equal to $\mathcal{H}^{n-1}|_{\pO}$), and let $m$ be a discrete Carleson measure 
with respect to $\sigma$ (defined as in \eqref{def:mQ} and Definition \ref{discrete-carleson}),
with constant $M_0$, that is
		\begin{equation}\label{m:extrapol}
		\|m\|_{\C}: = \sup_{Q\in \DD} \frac{ m(\DD_Q)}{ \sigma(Q) } \leq M_0.
	\end{equation}
	Let $\widetilde m$ be another discrete non-negative measure on $\DD$ defined as in \eqref{def:mQ}, by
	\begin{equation*}
		\widetilde m(\DD') := \sum_{Q\in \DD'} \beta_Q, \qquad \DD'\subset\DD. 
		\end{equation*}
Assume there is a constant $M_1$ such that
\begin{equation}\label{eq:betaQ}
\quad 0\le \beta_Q \le M_1 \sigma(Q)\  \hbox{    for any   }Q\in  \DD
\end{equation}
and that there is a positive constant $\gamma$ such that for every $Q\in \DD$ and every family of pairwise disjoint dyadic subcubes $\F=\{Q_j\}\subset \DD_Q$ verifying
	\begin{equation}\label{H:small}
		\|m_\F\|_{\C(Q)}:= \sup_{Q'\in \DD_Q} \frac{m \left( \DD_{\F,Q'}  \right) }{\sigma(Q')} \leq \gamma,
	\end{equation}
	we have that $\widetilde m$ satisfies
	\begin{equation}\label{H:bd}
		\widetilde m (\DD_{\F,Q}) \leq M_1 \sigma(Q).
	\end{equation}
	Then $\widetilde m$ is a discrete Carleson measure, with
	\begin{equation}
		\| \widetilde m\|_{\C}:= \sup_{Q\in \DD} \frac{\widetilde m(\DD_Q)}{\sigma(Q)} \leq M_2,
	\end{equation}
	for some $M_2 <\infty$ depending on $n, M_0, M_1, \gamma$ and the doubling constant of $\sigma$.
\end{theorem}

\begin{theorem}[\cite{HMM}, \cite{GMT}]\label{thm:bhc}
	Let $\mathcal{D}$ be an open set satisfying an interior corkscrew condition  with Ahlfors regular boundary. Then the following are equivalent:
\begin{enumerate}[label=\textup{(\alph*)}, itemsep=0.2cm] 	
		\item\label{1-thm:bhc} $\partial\mathcal{D}$ is uniformly rectifiable.
			
		\item\label{2-thm:bhc} There exists a constant $C$ such that for every bounded harmonic function  $u$ in $\mathcal{D}$, i.e. $-\Delta u = 0$ in $\mathcal{D}$, and for any $x\in \partial\mathcal{D}$ and $0<r\lesssim \diam(\mathcal{D})$, there hold
			\begin{equation}\label{eq:bhc}
				\frac{1}{r^{n-1}} \iint_{B(x,r) \cap \mathcal{D}} |\nabla u(Y)|^2 \dist(Y, \partial\mathcal{D})\, dY \leq C \|u\|_\infty^2.
			\end{equation}
	\end{enumerate}			
\end{theorem}

\begin{remark}
	Condition \eqref{eq:bhc} is sometimes referred to as the \textit{Carleson measure estimate} (CME)
	for bounded, harmonic functions.	
\end{remark}

The direction \ref{1-thm:bhc} $\implies$ \ref{2-thm:bhc} is proved by the first three authors of the present paper \cite[Theorem 1.1]{HMM}, and the converse direction is proved by Garnett, Mourgoglou, and Tolsa \cite[Theorem 1.1]{GMT}.  As we have noted above, see Theorem \ref{thm:hmu}, under the uniform domain assumption, the statements \ref{1-thm:bhc} and/or \ref{2-thm:bhc} are equivalent to the fact that $\mathcal{D}$ is a chord-arc domain.

%
%

The proof of Theorem \ref{thm:main} is rather involved thus we sketch below the plan of the proof. 

\begin{proof}[Proof of Theorem \ref{thm:main}]
We first reduce matters to the case on which $\wcalA$ is symmetric. To do so we observe that by \cite[Theorem 1.6]{CHMT}, under the assumptions \ref{H1} and \ref{H2}, if $\omega_L\in A_\infty(\sigma)$ then $\omega_{L^{\rm sym}}\in A_\infty(\sigma)$ where $L^{\rm sym}=-\divg(\wcalA^{\rm sym}\nabla)$ and $\wcalA^{\rm sym}=(\frac{a_{ij}+a_{ji}}2)_{i,j=1}^n$ is the symmetric part of $\wcalA$. Note that, clearly, $\wcalA^{\rm sym}$ is a symmetric uniformly elliptic matrix in $\Omega$ with the same ellipticity constants as $\wcalA$. It also satisfies  \ref{H1} and \ref{H2} with constants which are controlled by those of $\wcalA$. Hence we only need to show 
\ref{1-thm-main} $\implies$ \ref{3-thm-main} for $L^{\rm sym}$ which is associated to the symmetric matrix $\wcalA^{\rm sym}$.  That is, we may assume to begin with, and we do so, that $\wcalA$ is symmetric.

Our main goal is to use the above extrapolation theorem with $m$ and $\widetilde{m}$ two discrete measures associated respectively with the sequences $\alpha=\{\alpha_Q\}_{Q\in\dd}$ and $\beta=\{\beta_Q\}_{Q\in\dd}$ defined by
\begin{equation}\label{def-alpha-beta}
\alpha_Q: = \iint_{U_Q^\ast} |\nabla \wcalA(Y)|^2 \delta(Y) dY,\qquad 
\beta_Q ; = \iint_{U_Q} |\nabla u(Y)|^2 \delta(Y) dY,
\end{equation}
where $u$ is an arbitrary bounded, harmonic function in $\Omega$, such that $\|u\|_{L^\infty(\Omega)} \leq 1$; and $U_Q$ and $U_Q^\ast$ are as defined in \eqref{eq2.whitney3} and \eqref{eq:fatten-objects} respectively.
We would like to observe that by the interior Caccioppoli inequality, $\beta_Q$ clearly satisfies the assumption \eqref{eq:betaQ}:
\begin{multline*}
\beta_Q 
= 
\iint_{U_Q} |\nabla u(Y)|^2 \delta(Y)\, dY
\lesssim
\sum_{I\in \W_Q} \ell(I)\,\iint_{I^*} |\nabla u(Y)|^2\,dY
\lesssim
\sum_{I\in \W_Q} \ell(I)^{-1}\,\iint_{I^{**}} |u(Y)|^2\,dY
\\
\lesssim
\ell(Q)^{-1}\,\iint_{U_Q^*} |u(Y)|^2 \delta(Y) dY
\le
\ell(Q)^n
\approx \sigma(Q),
\end{multline*}
where we have used that \eqref{def:WQ}, the bounded overlap of the family $\{I^{**}\}_{I\in\W}$, and \eqref{tent-Q}. We will take any family of pairwise disjoint dyadic subcubes $\F=\{Q_j\}\subset \DD_Q$ so that \eqref{H:small} holds for sufficiently small $\gamma\in (0,1)$ to be chosen and the goal is to obtain \eqref{H:bd}. To achieve this  we will carry out the following steps:

\begin{enumerate}[label=\textup{Step \arabic*}:, ref=\textup{Step \arabic*}, itemsep=0.2cm]

\item\label{1-proof-main} We first observe that \eqref{m:extrapol} is equivalent to the Carleson measure assumption \ref{H2}. This is a simple calculation which uses the fact that the Whitney boxes  $I^{\ast\ast}$ which form $U_Q^\ast$ have finite overlap and the definition of $T_Q$ in \eqref{carleson-box}, details are left to the reader.

\item\label{2-proof-main} Given $\epsilon>0$ we verify that the small Carleson hypothesis \eqref{H:small} implies that if $\gamma=\gamma(\epsilon)$ is small enough $\wcalA$ satisfies the small Carleson assumption (with constant $\epsilon$) in the sawtooth domain $\Omega_{\F,Q}^\ast$. This is done in Section \ref{s:small-extrapolation}.

\item\label{3-proof-main} We verify that under the hypotheses \ref{H1} and \ref{H2}, the assumption $\omega \in A_\infty(\sigma)$ in $\Omega$ is transferable to any sawtooth domain, in particular, if we write $\omega_{*}$ for the elliptic measure associated with $L$ in $\Omega_{\F,Q}^\ast$ then  $\omega_{*}\in A_\infty(\HH^{n-1}|_{\partial{\Omega_{\F,Q}^\ast}})$ and the implicit constants are uniformly controlled by the allowable constants. See Theorem \ref{thm:trsf} and Corollary \ref{corol:KP-transfer}.

\item\label{4-proof-main}  We combine \ref{2-proof-main} and \ref{3-proof-main}  with Corollary \ref{corol:DKP-small} applied to the domain $\Omega_{\F,Q}^\ast$ and obtain that $\Omega_{\F,Q}^\ast$ is a chord-arc domain. More precisely, note first that $\Omega_{\F,Q}^\ast$ is a \textbf{bounded} uniform domain with Ahlfors regular boundary (see Lemma \ref{lemma:sawtooth-inherit}) and all the implicit constants are uniformly controlled by those of $\Omega$, that is, they do not depend on $Q$ or the family $\F$. Also, \ref{3-proof-main} says that $\omega_{*}\in A_\infty(\HH^{n-1}|_{\partial{\Omega_{\F,Q}^\ast}})$ and the implicit constants are uniformly controlled by the allowable constants. Hence for the parameter $\epsilon$ given by Corollary \ref{corol:DKP-small} (recall that we have assumed that $\wcalA$ is \textbf{symmetric}), which only depends on the allowable constants and is independent of $Q$ or the family $\F$, we can find the corresponding $\gamma=\gamma(\epsilon)$ from \ref{2-proof-main} so that $\wcalA$ satisfies the small Carleson assumption (with constant $\epsilon$) in the sawtooth domain $\Omega_{\F,Q}^\ast$. Thus Corollary \ref{corol:DKP-small} applied to the domain $\Omega_{\F,Q}^\ast$ yields that $\Omega_{\F,Q}^\ast$ is a chord-arc domain with constants that only depend on the allowable constants.

\item\label{5-proof-main} We next apply Theorem \ref{thm:bhc} with $\mathcal{D}=\Omega_{\F,Q}^\ast$ to obtain that \eqref{eq:bhc} holds with $\mathcal{D}=\Omega_{\F,Q}^\ast$. Seeing that the latter implies 	\eqref{H:bd} is not difficult. Indeed, note that any $Y \in \Omega_{\F,Q}$ satisfies $\delta_*(Y):= \dist(Y, \partial\Omega_{\F,Q}^\ast ) \approx_\tau \delta(Y)$ (here we would like to remind the reader that $\Omega_{\F,Q}$ is comprised of fattened Whiney boxes $I^*=(1+\tau)I$ while for $\Omega_{\F,Q}^*$ we use the fatter versions $I^{**}=(1+2\tau)I$).  Thus by \eqref{eq:bhc}, the fact that $u$ is harmonic and bounded by $1$ in $\Omega$, and so in $\Omega_{\F,Q}^*$, and a simple covering argument, we can conclude that
\begin{multline*}
\null\hskip1.2cm	\widetilde m_{}( \DD_{\F,Q}) \lesssim \iint_{\Omega_{\F,Q}} |\nabla u|^2 \delta(Y) dY  \approx \iint_{\Omega_{\F,Q}} |\nabla u|^2 \delta_* (Y) dY \\
\lesssim	 \iint_{\Omega_{\F,Q}^\ast}  |\nabla u|^2 \delta_* (Y) dY  \lesssim \diam(\Omega_{\F,Q}^\ast)^{n-1} \approx \ell(Q)^{n-1} \approx \sigma(Q),
\end{multline*}
which is \eqref{H:bd}.
\end{enumerate} 

\medskip

After all these steps have been carried out the extrapolation for Carleson measures in Theorem \ref{thm:extrapolation} allows us to conclude that $\widetilde m$ is a discrete Carleson measure. In other words, we have proved that any bounded harmonic function in $\Omega$ satisfies \eqref{eq:bhc} with $\mathcal{D}=\Omega$. As a result, and by another use of Theorem \ref{thm:bhc} this time with $\mathcal{D}=\Omega$, we derive that $\partial\Omega$ is uniformly rectifiable. This completes the proof of Theorem \ref{thm:main} modulo establishing \ref{2-proof-main} and \ref{3-proof-main} and this will be done in the following sections.
\end{proof}

\begin{remark}
For convenience, we augment $\F$ by adding all subcubes of $Q$ of length $2^{-N}\ell(Q)$, 
and  let $\F_N$ denote the maximal cubes
in the resulting augmented collection. Note that for each $N\ge 2$, the sawtooth domain $\Omega_{\F_N,Q}$ is compactly contained in $\Omega$ (indeed is $2^{-N}\ell(Q)$-away from $\pom$). Note that $\DD_{\F_{N},Q} \subset \DD_{\F_{N'},Q} \subset \DD_{\F,Q} $ for every $2\le N\le N'$. In particular, $m_{\F_{N}} \leq m_{\F_{N'}} \leq m_{\F}$ and thus
\[ m_{\F} \text{ satisfies  \eqref{H:small}} \implies m_{\F_N} \text{ also satisfies the \eqref{H:small} with a constant independent of }N. \] 
We are going to prove \ref{2-proof-main} and \ref{3-proof-main} for the sawtooth domain $\Omega_{\F_N,Q}$, with constants independent of $N$. Then by  
\ref{4-proof-main} and \ref{5-proof-main}, we will have 
\begin{equation}
	\widetilde m(\DD_{\F_N,Q}) \leq M_1 \sigma(Q),
\end{equation}
with a constant $M_1$ independent of $N$, and thus \eqref{H:bd} follows from monotone convergence theorem by letting $N\to \infty$. To simplify the notations we drop from  the index $N$ from now on and write $\F=\F_N$ but we keep in mind that the corresponding sawtooth domain $\Omega_{\F,Q}$ is compactly contained in $\Omega$.
\end{remark}

\section{Consequences of the small Carleson hypothesis in the extrapolation theorem.}\label{s:small-extrapolation}

Set $\Omega_*:= \Omega_{\F,Q}^\ast$ and let $\epsilon$ be given. The goal is to see that we can find $\gamma=\gamma(\epsilon)\in (0,1)$ so that  \eqref{H:small} 
implies
\begin{equation}\label{sC:alt}
\iint_{B(x,r) \cap \Omega_*} |\nabla \wcalA(Y)|^2 \delta_*(Y) dY \leq \epsilon r^{n-1},
\end{equation}
for any $x\in\pom_*$ and any $0<r<\diam(\pom_*)$. To see this we fix $x\in\pom_*$ and $0<r<\diam(\pom_*)\approx\ell(Q)$. Using that $\Omega_* \subset \Omega$ one has that  $\delta_*(Y) \leq \delta(Y)$ and therefore \eqref{sC:alt} follows at once from 
\begin{equation}\label{sC}
\iint_{B(x,r) \cap \Omega_*} |\nabla \wcalA(Y)|^2 \delta(Y) dY \leq \epsilon r^{n-1}.
\end{equation}

To show \eqref{sC}, we let $c\in (0,1)$ be a small constant and $\widetilde{M}\ge 1$ be a large constant to be determined later, depending on the values of $\eta, K$ used in the definition of $\WW_Q^0 $ in \eqref{def:WQ0}. We consider two cases depending on the size of $r$ with respect to $\delta(x)$ for $x\in\pom^\ast$. Recall that $\Omega_*$ is compactly contained in $\Omega$, thus $\delta(x)>0$ for any $x\in \pO_*$.

\noindent
\textbf{Case 1}. $r\leq c\delta(x)$. Since $x\in\pom_*= \partial\Omega_{\F,Q}^\ast$ there exist $Q_x\in \DD_{\F,Q}$ and $I_x\in \W_{Q_x}$ such that $x\in \partial I_x^{**}$.  We choose and fix $c$ sufficiently small (depending just on dimension), so that $B(x,r)$ is contained in $2I_x$. We consider two sub-cases. First if $r\le \gamma^\frac1n\delta(x)$ then we can invoke \ref{H1} to obtain 
\begin{multline}\label{Case1a}
\iint_{B(x,r) \cap \Omega_*} |\nabla \wcalA(Y)|^2 \delta(Y) dY
\le
\iint_{B(x,r) \cap 2I_x} |\nabla \wcalA(Y)|^2 \delta(Y) dY
\lesssim
\iint_{B(x,r) \cap 2I_x} \delta(Y)^{-1} dY
\\
\approx
\ell(I_x)^{-1} \,r^{n} 
\approx
\delta(x)^{-1} \,r^{n} 
\lesssim
\gamma^\frac1n\,r^{n-1}.
\end{multline}
On the other hand, if $\gamma^\frac1n\delta(x)\le r$ we note that 
\[
B(x,r) \cap \Omega_* \subset 2I_x\cap \Omega_* 
\subset
\bigcup_{Q'\in\dd_{\F,Q}} (U_Q^{*}\cap 2I_x).
\]
It is clear that from construction if $U_{Q'}^{*}\cap 2I_x\neq\emptyset$ then $\ell(Q')\approx\ell(I_x)\approx \delta(x)$. Note also that $\#\{I\in\W: I\cap 2I_x\neq\emptyset\}\lesssim C_n$ hence $\#\{Q'\in\dd: U_{Q'}^{*}\cap 2I_x\neq\emptyset\}\lesssim C_{n,\eta,K}$. Thus, observing that $Q'\in\dd_{\F,Q'}$ for every $Q'\in \dd_{\F,Q}$ we obtain from \eqref{H:small} 
\begin{multline}\label{Case1b}
\iint_{B(x,r) \cap \Omega_*} |\nabla \wcalA(Y)|^2 \delta(Y) dY
\le
\sum_{\substack{Q'\in\dd_{\F,Q}\\ U_{Q'}^*\cap 2I_x\neq\varnothing}}
\iint_{U_Q^*} |\nabla \wcalA(Y)|^2 \delta(Y) dY
=
\sum_{\substack{Q'\in\dd_{\F,Q}\\ U_{Q'}^*\cap 2I_x\neq\varnothing}} \alpha_{Q'}
\\
\le
\sum_{\substack{Q'\in\dd_{\F,Q}\\ U_{Q'}^*\cap 2I_x\neq\varnothing}} m(\dd_{\F,Q'})
\le
\gamma\,\sum_{\substack{Q'\in\dd_{\F,Q}\\ U_{Q'}^*\cap 2I_x\neq\varnothing}}  \sigma(Q')
\lesssim
\gamma\,\,\delta(x)^{n-1}
\lesssim
\gamma^\frac1n\,r^{n-1}.
\end{multline}

\noindent\textbf{Case 2}. $\widetilde{M}^{-1}\,\ell(Q)<r<\diam(\pom_*)\approx\ell(Q)$. This is a trivial case since by construction and \eqref{H:small}  we obtain
\begin{multline}\label{Case2}
\iint_{B(x,r) \cap \Omega_*} |\nabla \wcalA(Y)|^2 \delta(Y) dY 
\le
\sum_{Q'\in\dd_{\F,Q}} \iint_{U_{Q'}^*} |\nabla \wcalA(Y)|^2 \delta(Y) dY 
\\
=
\sum_{Q'\in\dd_{\F,Q}} \alpha_{Q'}
=
m(\dd_{\F,Q})
\le
\gamma\sigma(Q)
\approx
\gamma \ell(Q)^{n-1}
\approx
\gamma r^{n-1}.
\end{multline}

\noindent\textbf{Case 3}. $c\delta(x)<r \le \widetilde{M}^{-1}\,\ell(Q)$. Pick $\hat x\in \pom$ such that $|x-\hat x| = \delta(x)$ and note that $B(x,r)\subset B(\hat{x}, (1+ c^{-1})r)$. Note also that if $Q'\in\dd_{\F,Q}$ is so that $U_{Q'}^*\cap B(x,r)\neq\emptyset$ then we can find $I\in \W_{Q'}$ and $Y\in I^{**}\cap B(x,r)$ so that by \eqref{def:WQ}
\[
\eta^{\frac12}\ell(Q')
\lesssim 
\ell(I)
\approx
\delta(Y)
\le
|Y-x|+\delta(x)
<
(1+c^{-1})\,r.
\]
and for every $y\in Q'$
\begin{multline*}
|y-\hat{x}|
\le
\diam(Q')+\dist(Q',I)+\diam(I)+|Y-x|+|x-\hat{x}|
\\
\lesssim
K^{\frac12}\ell(Q')+r+\delta(x)
\lesssim
K^{\frac12}\,\eta^{-\frac12}(1+ c^{-1})\,r.
\end{multline*}
Consequently, if we write $\widetilde{M}'= C\,K^{\frac12}\,\eta^{-\frac12}\,(1+c^{-1})$ and choose $\widetilde{M}>\widetilde{M}'$, it follows that $\ell(Q')<\widetilde{M}'\,r<\ell(Q)$
and $Q'\subset \Delta(\hat{x}, \widetilde{M}'r)=:\Delta'$. 

We can then find a pairwise disjoint family of dyadic cubes $\{Q_k\}_{k=1}^{\widetilde{N}}$ with uniform cardinality $\widetilde{N}$ (depending on $C_{AR}$ and $n$) so that $2^{-1}\,\widetilde{M}'\,r\le  \ell(Q_k)< \widetilde{M}'\,r $, $Q_k\cap \Delta'\neq\emptyset$ for every $1\le k\le \widetilde{N} $, and $\Delta'\subset \bigcup_{k=1}^{\widetilde{N}} Q_k$. 
Relabeling if necessary, we can assume that there exists $\widetilde{N}'\le \widetilde{N}$ so that $\Delta'\cap Q\subset \bigcup_{k=1}^{\widetilde{N}'} Q_k$ and each $Q_k$ meets $\Delta'\cap Q$ for $1\le k\le \widetilde{N}'$. We would like to observe that necessarily $\widetilde{N}'\ge 1$ since we have shown that $Q'\subset \Delta'$ for every $Q'\in\dd_{\F,Q}$ so that $U_{Q'}^*\cap B(x,r)\neq\emptyset$. Also $Q_k\subset Q$ for $1\le k\le \widetilde{N}'$ since $\ell(Q_k)<\widetilde{M}'\,r<\ell(Q)$ and $Q_k$ meets $Q$.
Moreover, for every such a $Q'$ we necessarily have $Q'\subset Q_k$ for some $1\le k\le \widetilde{N}'$ since
$Q'\subset \Delta'\cap Q$, hence $Q'$ meets some $Q_k$ and also $\ell(Q')<\widetilde{M}'\,r\le 2\ell(Q_k)$ which forces $Q'\subset Q_k$. All these and \eqref{H:small} readily imply
\begin{multline}\label{Case3}
\iint_{B(x,r) \cap \Omega_*} |\nabla \wcalA(Y)|^2 \delta(Y) dY 
\le
\sum_{\substack{Q'\in\dd_{\F,Q}\\ U_{Q'}^*\cap B(x,r)\neq\varnothing}} \iint_{U_{Q'}^*} |\nabla \wcalA(Y)|^2 \delta(Y) dY 
=
\sum_{\substack{Q'\in\dd_{\F,Q}\\ U_{Q'}^*\cap B(x,r)\neq\varnothing}} \alpha_{Q'}
\\
\le
\sum_{k=1}^{\widetilde{N}'}
\sum_{\substack{Q'\in\dd_{\F,Q}\\ Q'\subset Q_k\subset Q}} \alpha_{Q'}
=
\sum_{k=1}^{\widetilde{N}'}
\sum_{Q'\in\dd_{\F,Q_k}} \alpha_{Q'}
=
\sum_{k=1}^{\widetilde{N}'} m(\dd_{\F,Q_k})
\le
\gamma \sum_{k=1}^{\widetilde{N}'} \sigma(Q_k)
\lesssim
\gamma r^{n-1}.
\end{multline}

\medskip

Combining what we have obtained in all the cases we see that \eqref{Case1a}, \eqref{Case1b}, \eqref{Case2}, and \eqref{Case3} give, since $0<\gamma<1$, that 
\[
\frac1{r^{n-1}}\iint_{B(x,r) \cap \Omega_*} |\nabla \wcalA(Y)|^2 \delta(Y) dY
\le
C_0\gamma^{\frac1n},
\]
for some constant $C_0\ge 1$ depending on the allowable constants and where we recall that $\gamma$ is at our choice. Hence we just need to pick $\gamma< (C_0^{-1}\epsilon)^n$ to conclude as desired \eqref{sC}.

\section{Transference of the \texorpdfstring{$A_\infty$}{A-infinity} property to sawtooth domains}\label{sect:trsf}

In this section we show that the $A_\infty$ property for the elliptic operator $L$ in $\Omega$ can be transferred to sawtooth subdomains with constants that only depend on the allowable constants. We first work with sawtooth subdomains which are compactly contained in $\Omega$ and then we consider the general case using that interior sawtooth subdomains exhaust general sawtooth domains.

\begin{theorem}\label{thm:trsf}
Let $\Omega\subset \RR^n$ be a uniform domain with Ahlfors regular boundary. Let $\wcalA$ be a \textbf{symmetric} uniformly elliptic matrix on $\Omega$ and $L=-\divg(\wcalA\nabla)$ . Assume the following two properties: 
\begin{enumerate}[label=\textup{(\arabic*)}, itemsep=0.2cm] 
		\item\label{1-thm:trsf} The elliptic measure $\omega_L$ associated with the operator $L$ relative to the domain $\Omega$ is of class $A_\infty$ with respect to the surface measure. 
		
		\item\label{2-thm:trsf} For every fundamental chord-arc subdomain $\PP$ of $\Omega$, see \eqref{fundamental-casd}, the elliptic measure associated with $L$ relative to the domain $\PP$ is also of class $A_\infty$ with respect to the surface measure of $\PP$, with uniform $A_\infty$ constants.
\end{enumerate}
	For every $Q\in \DD$ and every family of pairwise disjoint dyadic subcubes $\F = \{Q_j\} \subset \DD_Q$, let $\Omega_*=\Omega_{\F,Q}$ (or $\Omega_*=\Omega_{\F,Q}^*$) be the associated sawtooth domain, and $\omega_*$ and $\sigma_* = \HH^{n-1}|_{\partial\Omega_*}$ be the elliptic measure for $L$ and the surface measure of $\Omega_*$. Then $\omega_* \in A_\infty(\sigma_*)$, with the $A_\infty$ constants independent of $Q$ and $\F$.
\end{theorem}

Note that if $\wcalA$ is a non necessarily symmetric matrix satisfying hypotheses \ref{H1} and \ref{H2} in $\Omega$, we can easily verify it also satisfies the Kenig-Pipher condition relative to every fundamental chord-arc subdomain. Indeed, since $\PP\subset \Omega$ then $\delta_{\PP}(\cdot)\le \delta(\cdot)$ and \ref{H1} in $\PP$ is automatic. On the other hand, let $\PP=\interior \left( \bigcup_{j=1}^{m_1} I_j^\ast\right)$  with  $I_j\in\W $ and $ I\cap I_j \not =\emptyset$ and take $x\in\partial\PP$ and $r\le \diam\PP\lesssim \ell(I)$. Note that  \ref{H1} implies  $|\nabla \wcalA(Y)|^2\lesssim \ell(I)^{-2} $ for every $Y\in\PP$ since $\delta(Y)\approx\ell(I)$ , hence 
\begin{equation}\label{carleson-in-p}
\frac{1}{r^{n-1}}\iint_{B(x,r)\cap\PP}|\nabla \wcalA(Y)|^2 \delta_{\PP}(Y) dY\lesssim\frac{1}{r^{n-1}\ell(I)^2}\iint_{B(x,r)\cap\PP} \delta_{\PP}(Y) dY\lesssim 
\frac{r^{n+1}}{r^{n-1}\ell(I)^2}\lesssim 1.
\end{equation}
That is, \ref{H1} in $\PP$ holds as well. Thus by \cite{KP} (and the slight improvement in \cite{HMT1}), and the fact that chord-arc domains can be approximated by Lipschitz domains, one obtains that the elliptic measure for $L$ relative to $\PP$ is also of class $A_\infty$ with respect to the surface measure of $\PP$ and \ref{2-thm:trsf} in the previous result holds.  On the other hand, \cite[Theorem 1.6]{CHMT} asserts that for any uniform domain $\Omega$, and under the assumptions   \ref{H1} and \ref{H2}, one has that $\omega_L\in A_\infty(\sigma)$ if and only if $\omega_{L^{\rm sym}}\in A_\infty(\sigma)$ where $L^{\rm sym}$ is the operator associated with the symmetric matrix $\wcalA^{\rm sym}=(\frac{a_{ij}+a_{ji}}2)_{i,j=1}^n$. Note that $\wcalA^{\rm sym}$ is also a uniformly elliptic matrix in $\Omega$ with the same ellipticity constants as $\wcalA$ and satisfies  \ref{H1} and \ref{H2} with constants which are controlled by those of $\wcalA$. With all these observations we immediately get the following corollary:  
\begin{corollary}\label{corol:KP-transfer}
	Let $\Omega\subset \RR^n$ be a uniform domain with Ahlfors regular boundary. Suppose that $\wcalA$ is a (non necessarily symmetric) uniformly elliptic matrix on $\Omega$ 
	satisfying the hypotheses \textup{\ref{H1}} and \textup{\ref{H2}}, and that the elliptic measure $\omega_L$ associated with the operator $L$ relative to the domain $\Omega$ is of class $A_\infty$ with respect to the surface measure. Then the elliptic measure associated with $L$ relative to any sawtooth domain is of class $A_\infty$, with uniform constants.
\end{corollary}

\begin{proof}[Proof of Theorem \ref{thm:trsf}]

The proof Theorem \ref{thm:trsf} has several steps. We work with $\Omega_*=\Omega_{\F, Q}$ as the proof with $\Omega_{\F, Q}^*$ is identical. We first assume that the sawtooth domain $\Omega_\ast=\Omega_{\F, Q}$ is \textbf{compactly contained} in $\Omega$ and show that $\omega_\ast \in A_\infty(\sigma_*)$, with the $A_\infty$ constants independent of $Q$ and $\F$. Here we  use $\omega_*$ to denote the elliptic measure associated with $L$ relative to $\Omega_*$. 

Under the assumption that $\Omega_\ast$ is compactly contained in $\Omega$, for $Q$ fixed, let $N$ be an integer such that $\dist(\Omega_*,\pom) \approx 2^{-N} \ell(Q)$. 
Then $\Omega_*$ if formed by a union of fattened Whitney boxes of side length  controlled from below by $c\,2^{-N} \ell(Q)$ hence $\Omega_*$ clearly 
satisfies a \textit{qualitative exterior corkscrew condition}, that is, it satisfies the exterior corkscrew condition for surface balls up to a scale of the order of $2^{-N}\ell(Q)$. In the case of  the Kenig-Pipher operators, this information alone does not suffice to derive the desired $A_\infty$ property, with constant independent of $N$; however this does give us the qualitative absolute continuity $\omega_*^X \ll \sigma_*$ for any $X\in \Omega_*$ (since $\Omega_*$ is a chord-arc domain with constants depending on $N$). Note that Theorem \ref{thm:trsf} is nonetheless written for a more general class and it is not obvious whether we can automatically have the desired absolute continuity. This will be shown in the course of the proof.

Our main task is to then show that $\omega_* \in A_\infty(\sigma_*)$ with constants that depend only on the allowable constants. If we write $\textbf{k}_*:=d\omega_*/d\sigma_*$ for the Radon-Nikodym derivative, by the change of pole  formula Lemma \ref{lm:cop}, obtaining 
$\omega_* \in A_\infty(\sigma_*)$, it is equivalent to prove the following: 
there exists an exponent $p\in (1,\infty)$ and a constant $C$ depending only on the allowable constants such that
for any surface ball 
$\Delta_*=B_*\cap\partial\Omega_*$ centered at $\pom_*$, with radius smaller than the diameter of $\pom_*$, and for $X=X_{\Delta_*}\in \Omega_*\cap B_*$, a corkscrew point relative to $\Delta_*$, the following holds
\begin{equation}\label{RH}
	\int_{\Delta_*} \left(\textbf{k}_*^X\right)^p d\sigma_* \leq C \sigma_*(\Delta_*)^{1-p}.
\end{equation}
Since $\diam(\Omega_*) \approx \ell(Q)$, it is easy to see by a standard covering argument and Harnack's inequality that it suffices to prove \eqref{RH} for 
$r_* \le M_1^{-1} \ell(Q)$, 
where $M_1$ is a suitably large fixed constant.  By hypothesis \ref{1-thm:trsf}, $\omega_L\in A_\infty(\sigma)$, hence it belongs to the reverse H\"older class with some exponent $p_1>1$ (see \eqref{eq1.wRH}). Also, by hypothesis \ref{2-thm:trsf} we know that the elliptic measure relative to any fundamental chord-arc subdomain $\PP$ satisfies an $A_\infty$ condition with respect to the corresponding surface measure with uniform bounds. In turn, there exists $p_2>1$ and a uniform constant so that any of these elliptic measures belong to the reverse Hölder class with this exponent $p_2$ and with the same uniform constant  (see \eqref{eq1.wRH}).  We shall henceforth set $p:= \min \{p_1,p_2\}$, and it 
is for this $p$ that we shall prove \eqref{RH}. 

To start with the proof, recall that as observed above, since $\dist(\Omega_*, \pom) \approx 2^{-N}\ell(Q)$, it follows that all the dyadic cubes $Q'\in \DD_{\F,Q}$ have length $\ell(Q') \gtrsim 2^{-N}\ell(Q)$, and the cardinality of $\DD_{\F,Q}$ is bounded by a constant $C(N)$. Hence $\Omega_*=\Omega_{\F,Q}$ is formed by the finite union of Whitney regions $U_{Q'}$ with $Q' \in \DD_{\F,Q}$ satisfying $\ell(Q') \gtrsim 2^{-N}\ell(Q)$. In turn each $U_{Q'}$ is a polyhedral domain consisting of a finite number of fattened Whitney boxes with side length of the order of $\ell(Q')$. In particular there exists a finite index set $\N_*$ so that 
\begin{equation}\label{bdry-Omega*}
	\pO_* \subset \bigcup_{i\in \N_*} \partial \left((I^i)^* \right) \cap \pO_* = : \bigcup_{i\in \N_*} S_*^i.
\end{equation}
where $S_*^i \neq \emptyset$ for each $i\in \N_*$, $\interior((I^i)^*)\subset \Omega_*$, and $\ell(I^i)\gtrsim 2^{-N}$. 
For each $I^i$, with $i\in \N_*$, we pick $Q^i\in\DD_{\F,Q}$ such that $\W_{Q^i} \ni I^i$ (there could be more than one such a $Q^i$ in which case we just select one).  Note that different $I^i$'s may correspond to the same $Q^i$, but each $Q^i$ may only repeat up to a finitely many times, depending only on the allowable constants.  
Since $S_*^i$ is contained in the boundary of a fattened Whitney box $(I^i)^*$,
\begin{equation}\label{Sitmp}
	\diam(S_*^i) \lesssim \ell(I^i) \approx \ell(Q^i), \text{ and } \dist(S_*^i, \pO) \geq \dist((I^i)^*, \pO) \approx \ell(Q^i). 
\end{equation} 
On the other hand, the fact that $S_*^i \subset \pO_*$ means that $I^i$ intersects some $J^i\in \W$ so that if $J^i\in \W_{Q''}$ then  $Q'' \notin \DD_{\F,Q}$.
If we pick $\widetilde{Q}^i\in\dd$ so that $\ell(\widetilde{Q}^i)=\ell(J^i)$ and $\dist(J^i,\pom)=\dist(J^i, \widetilde{Q}^i)$ then as mentioned right below \eqref{def:WQ0} we have that $J^i\in \W^0_{\widetilde{Q}^i}\subset \W_{\widetilde{Q}^i}$, therefore $\widetilde{Q}^i\notin \DD_{\F,Q}$. Recalling \eqref{eq:tausmall2} and the comments after it, we know that $tJ^i \subset \Omega\setminus\Omega_*$ and  $\partial \left( (I^i)^* \right) \cap \pO_*$ contains  
an $(n-1)$-dimensional ball  with radius of the order of $\min\{\ell(I^i), \ell(J^i)\}\approx \ell(I^i)$. Denote that $(n-1)$-dimensional ball by  $\Delta_*^i \subset S_*^i$. 
This implies, combined with \eqref{Sitmp}, that
\begin{equation}\label{eq:sizeSi}
	r(\Delta_*^i) \approx \diam(S_*^i) \approx \ell(I^i) \approx \ell(Q^i)\ \text{\ and\ } \dist(S_*^i, \pO) \approx \dist(\Delta_*^i, \pO) \approx \ell(I^i) \approx \ell(Q^i).
\end{equation}

At this stage we consider several cases. In the \textbf{Base case}, see Lemma \ref{lemma:base-case}, we treat surface balls $\Delta_*$ with small radii so that $\Delta_*$ is contained in a uniformly bounded union of Whitney cubes of comparable sides. In the case when $\Delta_*$ is large we decompose the intersection of $\Delta_*$ in small pieces to which the base case can be applied (\textbf{Step 1}). We then put all the local estimates together to obtain a global one (\textbf{Step 2}). This requires Lemma \ref{lm:thinbd} and to consider several cases to account for all the small pieces.

Let $\Delta_* = B^*\cap  \pO_* \subset \Omega$, with $B_*=B(x_*,r_*)$, $x_*\in\pO_*$ and $0<r_*<\diam(\pom_*)$. Since $\Omega_*$ is a uniform domain (see Lemma \ref{lemma:sawtooth-inherit}), we can pick $X_{\Delta_*}\subset B_*\cap\Omega_*$, a Corkscrew point relative to $\Delta_*$ in $\Omega_*$, so that
$\delta_*(X):=\dist(X,\pom_*) \approx r_*$. 
 Write
\begin{equation}\label{eq:cover}
\Delta_* \subset \bigcup_{i\in \N_{\Delta_*}} S_*^i, \quad \text{ where }  \N_{\Delta_*}:=\{i:\in \N_{*}: \Delta_* \cap S_*^i \neq \emptyset\}.
\end{equation}

\begin{lemma}[Base case]\label{lemma:base-case}
Using the notation above we have that $\omega_*\ll \sigma_*$ in $\pom_*$.
Moreover if there exists  $i\in \N_{\Delta_*}$ such that $r_*\leq \frac{\tau}{8} \ell(I^i)$ then
\begin{equation}\label{conc-base-case}
\int_{\Delta_*} \left( \textup{\textbf{k}}_*^{X_{\Delta_*}} \right)^p d\sigma_*
\lesssim
\sigma_*(\Delta_*)^{1-p}\,,
\end{equation}
where $\textup{\textbf{k}}_* := d\hm_*/d\sigma_*$, $p$ is as above, and the implicit constant only depends on the allowable constants.
\end{lemma}

\begin{proof}
We first claim that 
\begin{equation}\label{eq:coversmall}
	2\Delta_* \subset \bigcup_{\substack{i'\in \N_*\\ I^{i'} \cap I^i \neq \varnothing} } S_*^{i'} \quad {\rm and}
	\quad 2B_*\cap\Omega_* \subset \bigcup_{\substack{i'\in \N_* \\ I^{i'} \cap I^i \neq \varnothing} } (I^{i'})^*.
\end{equation}
In fact, for any $i' \in \N_*$, if $S_*^{i'}$ intersects $2\Delta_*$, or if
$2B_*\cap\Omega_*$ intersects $ (I^{i'})^*$, then our current assumption gives 
\[
	\dist\big( (I^i)^*, (I^{i'})^* \big) \leq \dist \big((I^i)^* \cap 2B_*, (I^{i'})^* \cap 2B_* \big)
	 \leq \diam(2B_*) = 4r_*\leq \frac{\tau}{2} \ell(I^i),
\]
and thus 
\[ (I^i)^{**} \cap (I^{i'})^{**}\,\supset\, (I^i)^{**} \cap (I^{i'})^*\, \neq\, \emptyset. \]
By the choice of $\tau$, i.e., by \eqref{eq:tausmall1},  we then have $I^i \cap I^{i'} \neq \emptyset$ and the claim is proved. 
Next, let $m_1$ denote the maximal number of Whitney boxes intersecting $I^i$. Note that $m_1$ only depends on the constructions of the Whitney cubes, hence just on dimension. By relabeling \eqref{eq:coversmall} we write
\begin{equation}\label{eq:coversmall2}
	2\Delta_* \subset \bigcup_{i'=1}^{m_1} S_*^{i'}\quad {\rm and} \quad 2B_*\cap\Omega_* 
	\subset \bigcup_{i'=1}^{m_1}(I^{i'})^*.
\end{equation}
Moreover by \eqref{eq:sizeSi}, for each $i'=1, \dots, m_1$, we have $
	\diam(S_*^{i'}) \approx \ell(I^{i'}) \approx \ell(I^i)$.
Set  then 
$$\PP:= \interior\left( \bigcup_{i'=1}^{m_1} (I^{i'})^*\right)\subset\Omega_*,$$
which by construction is a fundamental chord-arc subdomain $\PP$ of $\Omega$, see \eqref{fundamental-casd}. 
Note that since $2B_*\cap\Omega_* $ is open then \eqref{eq:coversmall2} says that $2B_*\cap\Omega_* \subset \PP$ and hence $2B_*\cap\Omega_* =2B_*\cap\PP$. This and the fact that $\Omega_*$ and $\PP$ are open readily implies that $2B_*\cap\partial\Omega_* =2B_*\cap\partial\PP$. Moreover, $X_{\Delta_*}\in\PP$ (since $X\in B_*\cap\Omega_*$)
and 
\begin{equation}\label{eq5.13}
\dist(X_{\Delta_*},\partial\PP) \approx \delta_*(X_{\Delta_*}) \approx r_*\leq \frac{\tau}{8} \ell(I^i) \leq
\frac{\tau}{8}\diam(\PP)
\le 
\,\diam(\PP).
\end{equation}
Let $X_{\PP}$ be a Corkscrew point for the domain $\PP$, at the scale $ \ell(I^i) \approx\diam(\PP)$, i.e.,
$X_\PP$ is a Corkscrew point in $\PP$ relative to the surface ball consisting of the entire
boundary of $\PP$.  Thus in particular, $\dist(X_\PP,\partial\PP)\approx \diam(\PP)\ge \frac8{\tau} r_*$, hence $\dist(X_\PP,\partial\PP)\ge 2\,c_0\,r_*$ for some uniform $0<c_0<1/4$.  

Set $u_1(\cdot):=G_*(X_{\PP}, \cdot)$ and $u_2(\cdot):=G_{\PP}(X_{\PP},\cdot)$ in $2B_*\cap\Omega_* =2B_*\cap\PP$ where $G_*$ and $G_\PP$ are the Green functions for the operator $L$ and for the domains $\Omega_*$ and $\PP$ respectively, and where as observed above $X_\PP\in\PP\subset\Omega_*$. Fix $y\in\frac32B_*\cap\partial\Omega_*=\frac32B_*\cap\partial\PP$ and  note that $B(y,c_0 r_*)\subset 2B_*$. Note that if $Z\in B(y,c_0 r_*)$ then 
\[
2\,c_0r_*\le \dist(X_\PP,\partial\PP)
\le
|X_\PP-y|
\le
|X_\PP-Z|+|Z-y|
<
|X_\PP-Z|+c_0 r_*,
\]
and $|X_\PP-Z|> c_0 r_*$. As a consequence, $B(y,c_0 r_*)\subset 2B_*\setminus B(X_{\PP}, c_0 r_*)$. Hence, $L u_1=0$ and $L u_2=0$ in we weak sense in $B(y,c_0 r_*)\cap\Omega_* =B(y,c_0 r_*)\cap\PP$ and both are continuous in $B(y,c_0 r_*)\cap\overline{\Omega_*} =B(y,c_0 r_*)\cap\overline{\PP}$. In particular both vanish continuously in 
$B(y,c_0 r_*)\cap\partial\Omega_* =B(y,c_0 r_*)\cap\partial\PP$. This means that we can  use Lemma \ref{lm:comp} in $\mathcal{D}=\PP$ to obtain that for every $Z\in B(y, c_0\,r_*/8)$
\begin{equation}\label{comp-u1-u2}
\frac{u_1(Z)}{u_2(Z)}
\approx
\frac{u_1(X_{\Delta_\PP(y,c_0 r_*/2)}^\PP)}{u_2(X_{\Delta_\PP(y,c_0 r_*/2)}^\PP)}
\end{equation}
where $X_{\Delta_\PP(y,c_0 r_*/2)}^\PP$ is a corkscrew relative to $B(y,c_0 r_*/2)\cap\overline{\PP}$ for the fundamental chord-arc domain $\PP$. On the other hand Lemma \ref{lm:CFMS} applied in $\Omega_*$ (which is uniform with Ahlfors regular boundary and the implicit constants are uniformly controlled, see Lemma \ref{lemma:sawtooth-inherit}) and $\PP$ (a fundamental chord-arc domain) gives for any $0<s\le c_0 r_*/2$
\begin{equation}\label{CFMS-u1-u2}
u_1(X_{\Delta_*(y,s)}^*)\approx \omega_*^{X_{\PP}}(\Delta_*(y,s))\,s^{n-2},
\qquad 
u_2(X_{\Delta_\PP(y,s)}^\PP)\approx \omega_\PP^{X_{\PP}}(\Delta_\PP(y,s))\,s^{n-2},
\end{equation}
where $X_{\Delta_*(y,s)}^*$ is the corkscrew point relative to  $\Delta_*(y,s)= B(y,s)\cap\pom_*$ for the uniform domain $\Omega_*$,
$X_{\Delta_\PP(y,s)}^\PP$ is the corkscrew point relative to  $\Delta_\PP(y,s)= B(y,s)\cap\partial\PP$ for the fundamental chord-arc uniform domain $\PP$, and $\omega_\PP$ stands for the elliptic measure associated with the operator $L$ relative to $\PP$. Note that from the definition of corkscrew condition and the fact that $B(y,s)\cap\Omega_*=B(y,s)\cap\PP$ it follows that $X_{\Delta_*(y,s)}^*, X_{\Delta_\PP(y,s)}^*\in B(y,s)\cap\Omega_*=B(y,s)\cap\PP$ and also 
\[
\dist(X_{\Delta_*(y,s)}^*, \pom_*)\approx \dist(X_{\Delta_*(y,s)}^*, \partial\PP)\approx \dist(X_{\Delta_\PP(y,s)}^\PP, \pom_*)\approx
\dist(X_{\Delta_\PP(y,s)}^\PP, \partial\PP)\approx s.
\]
Consequently $u_1(X_{\Delta_*(y,s)}^*)\approx u_1(X_{\Delta_\PP(y,s)}^\PP)$ and $u_1(X_{\Delta_\PP(y,c_0 r_*/2)}^\PP)\approx u_1(X_{\Delta_*(y,c_0 r_*/2)}^*)$.
All these, together with \eqref{comp-u1-u2}, \eqref{CFMS-u1-u2}, and Lemma \ref{lm:Bourgain}, give for every  $0<s\le c_0 r_*/8$
\begin{multline}\label{comp-w*-wPP}
\frac{\omega_*^{X_{\PP}}(\Delta_*(y,s))}{\omega_\PP^{X_{\PP}}(\Delta_\PP(y,s))}
\approx
\frac{u_1(X_{\Delta_*(y,s)}^*)}{u_2(X_{\Delta_\PP(y,s)}^\PP)}
\approx
\frac{u_1(X_{\Delta_\PP(y,s)}^\PP)}{u_2(X_{\Delta_\PP(y,s)}^\PP)}
\approx
\frac{u_1(X_{\Delta_\PP(y,c_0 r_*/2)}^\PP)}{u_2(X_{\Delta_\PP(y,c_0 r_*/2)}^\PP))}
\\
\approx
\frac{u_1(X_{\Delta_*(y,c_0 r_*/2)}^*)}{u_2(X_{\Delta_\PP(y,c_0 r_*/2)}^\PP))}
\approx
\frac{\omega_*^{X_{\Delta_*(y,c_0 r_*/2)}}(\Delta_*(y,c_0 r_*/2))}{\omega_\PP^{X_{\Delta_\PP(y,c_0 r_*/2)}}(\Delta_\PP(y,c_0r_*/2))}
\approx 1
.
\end{multline}
With this in hand, we note that since $y\in\Delta_*(x_*,\frac32r_*)=\frac32B_*\cap\partial\Omega_*=\frac32B_*\cap\partial\PP=\ \Delta_\PP(x_*,\frac32r_*)$ and $0<s\le c_0 r_*/8$ are arbitrary we can easily conclude, using a Vitali covering argument and the fact that both $\omega_*^{X_{\PP}}$ and $\omega_\PP^{X_{\PP}}$ are outer regular  and doubling in $\Delta_*(x_*,\frac32r_*)=\Delta_\PP(x_*,\frac32r_*)$, that $\omega_*^{X_{\PP}}(F)\approx \omega_\PP^{X_{\PP}}(F)$ for any Borel set $F\subset \Delta_*(x_*,\frac32r_*)=\Delta_\PP(x_*,\frac32r_*)$. Hence $\omega_*^{X_{\PP}}\ll \omega_\PP^{X_{\PP}} \ll\omega_*^{X_{\PP}}$ in $\Delta_*(x_*,\frac32r_*)=\Delta_\PP(x_*,\frac32r_*)$. From hypothesis  \ref{2-thm:trsf} in Theorem \ref{thm:trsf} we know that  $\omega_\PP\ll \sigma_\PP:=\HH^{n-1}|_{\partial\PP}$, hence in particular  $\omega_*\ll\sigma_*$ in $\Delta_*(x_*,\frac32r_*)$. This, \eqref{comp-w*-wPP}, and Lebesgue's differentiation theorem readily imply that 
\begin{equation}\label{tmp1}
	\textbf{k}_*^{X_{\PP}} (y) \approx \textbf{k}_{\PP}^{X_{\PP}}(y), \quad \text{for $\HH^{n-1}$-almost all } y\in \Delta_*(x_*,r_*)=\Delta_\PP(x_*,r_*),
\end{equation} 
where $\textbf{k}_\PP := d\hm_\PP/d\sigma_\PP$ and $\textbf{k}_* := d\hm_*/d\sigma_\PP$. 

We next observe that Lemma \ref{lm:cop} applied with $\mathcal{D}=\Omega_*$ (along with Harnack's inequality  for the case $r_*\approx \ell(I^i)$) and Lebesgue's differentiation theorem  yield 
	\begin{equation}\label{tmp2}
		\textbf{k}_*^{X_{\Delta_\ast}}(y) \approx \frac1{\omega_*^{X_\PP}(\Delta_*)} \, \textbf{k}_*^{X_{\PP}}(y)\quad \text{for $\sigma_*$-almost all } y\in \Delta_*. 
	\end{equation} 
Since $\PP$ is a fundamental chord-arc subdomain $\PP$ of $\Omega$, see \eqref{fundamental-casd}, as observed above $\omega_\PP$ belongs to the reverse Hölder class 
with exponent $p_2>1$ and so with exponent $p= \min \{p_1, p_2\}$. We find that since ${\sigma_\ast} ={\sigma_\PP}$ in $\Delta_\ast=\Delta_*(x_*,r_*)=\Delta_{\PP}(x_*,r_*)$
\begin{multline*}
		\int_{\Delta_*} \left( \textbf{k}_*^{X_{\Delta_\ast}} \right)^p d\sigma_*
		\lesssim
		\frac{1}{\big(\omega_*^{X_\PP}(\Delta_*)\big)^{p}} \,
		\int_{\Delta_*} \left( \textbf{k}_*^{X_\PP} \right)^p d\sigma_*
		\approx
		\frac{\sigma_\PP(\Delta_\PP(x_*,r_*))}{\big(\omega_*^{X_\PP}(\Delta_*)\big)^{p}} \,
		\fint_{\Delta_\PP(x_*,r_*)} \left( \textbf{k}_\PP^{X_\PP} \right)^p d\sigma_\PP\\
\lesssim 
\frac{\sigma_\PP(\Delta_\PP(x_*,r_*))}{\big(\omega_*^{X_\PP}(\Delta_*)\big)^{p}} 
\left(\frac{\omega_\PP^{X_\PP}(\Delta_\PP(x_*,r_*))}{\sigma_\PP(\Delta_\PP(x_*,r_*))}\right)^p
		\approx\, \sigma_*(\Delta_*)^{1-p}\,,
	\end{multline*}
where we have used \eqref{tmp2}, \eqref{tmp1}, that ${\sigma_\ast} ={\sigma_\PP}$,  the reverse Hölder estimate with exponent $p$ for $\textbf{k}_\PP$, and that both $\pom_*$ and $\partial\PP$ are Ahlfors regular sets with uniform bounds.

To complete our proof we need to see that $\omega_*\ll\sigma_*$ in $\pom_*$. Let us observe that we have already obtained that $\omega_*\ll\sigma_*$ in $\Delta_*(x_*,\frac32r_*)$ where $x_*\in \pom_*$ is arbitrary and $r_*\le\frac{\tau}{8}\ell(I^i)$ for some $i\in \mathcal{N}_{\Delta_*}$. We may cover $\pom_*$ by a finite union of surface balls $\Delta_*(x_j, r_j)$, with $r_j=\frac{2^{-N}}{\tilde{M}}\ell(Q)$, where $\widetilde{M}$ is large enough to be chosen, whose cardinality may depend on $N$ and $\widetilde{M}$. Note that for every $i\in \mathcal{N}_*$ we have, as observed before, that $\ell(I^i)\gtrsim 2^{-N}\ell(Q)> \frac{8}{\tau} \frac{2^{-N}}{\tilde{M}}\ell(Q)$ if we pick $\widetilde{M}$ large enough. Hence, for every $j$, it follows that $r_j<\frac{\tau}{8}\ell(I^i)$ for every $i\in\mathcal{N}_*$ and in particular for every $i\in \mathcal{N}_{\Delta_*(x_j,r_j)}$. Hence the previous argument yields that
$\omega_*\ll\sigma_*$ in $\Delta_*(x_j,\frac32r_j)$ for every $j$ and consequently $\omega_*\ll\sigma_*$ in $\pom_*$. 
\end{proof}

\begin{remark}
	We would like to emphasize that the fact that  $\omega_*\ll\sigma_*$ in $\pom_*$ is automatic for the Kenig-Pipher operators. In fact as observed above $\Omega_*$ is a chord-arc domain and hence $\omega_*\in A_\infty(\sigma_*)$ (albeit with constants which may depend on $N$). The previous argument proves that the more general hypothesis  \ref{2-thm:trsf} in Theorem \ref{thm:trsf} also yields $\omega_*\ll\sigma_*$ in $\pom_*$.
	 \end{remark}

Once the \textbf{Base case} has been established we can focus on proving the $A_\infty$ property for the sawtooth. With this goal in mind we fix a surface ball $\Delta_* = B_*\cap  \pO_* \subset \Omega$, with $B_*=B(x_*,r_*)$, $x_*\in\pO_*$ and $0<r_*<\diam(\pom_*)$. Let $X:=X_{\Delta_*}\subset B_*\cap\Omega_*$ be a Corkscrew point relative to $\Delta_*$ in $\Omega_*$, so that
$\delta_*(X) \approx r_*$. Our goal is to show \eqref{RH}. As explained above we may assume that $r_* \le M_1^{-1} \ell(Q)$, for some $M_1$ large enough to be chosen. The \textbf{Base case} (Lemma \ref{lemma:base-case})
yields \eqref{RH} when $r_* <\frac{\tau}{8} \ell(I^i)$ for some $i\in \N_{\Delta_*}$. Hence we may assume from now on that $r_* \geq \frac{\tau}{8} \ell(I^i)$ for every $i\in \N_{\Delta_*}$.

\noindent\textbf{Step 1.} Show that 
\begin{align}\label{eq:tosumN}
\int_{\Delta_*} \left(\textbf{k}_*^X \right)^p d\sigma_* \lesssim \sum_{i\in \N_{\Delta_*}} \int_{Q^i} \left(\textbf{k}^X \right)^p d\sigma,
\end{align}
where we recall that $Q^i\in\dd_{\F,Q}$ is so that $I^i\in\W_{Q^i}$ for every $i\in \mathcal{N}_{*}$, and where $\textbf{k}=d\omega_L/d\sigma$.

To see this, by \eqref{eq:cover}, it suffices to obtain 
	\begin{equation}\label{eq:Si} 
\int_{S_*^i} \left( \textbf{k}_*^X \right)^p d\sigma_* 
\lesssim \int_{Q^i} \left(\textbf{k}^X \right)^p d\sigma.
\end{equation}
for each $i\in \N_{\Delta_*}$. Fix then such an $i$ and cover $S_*^i$ by a uniformly bounded number of
surface balls centered at $\pom_*$ with small radius $\Delta_*^{i,l}=B_*^{i,l}\cap \pO_*$ where $S_*^i \cap \Delta_*^{i,l} \neq \emptyset$ and $r(\Delta_*^{i,l}) \approx c \diam(S_*^i)\approx c\,\ell(I^i) $, the constant $c$ is chosen sufficiently small 
(depending on $\tau$), so that $r(\Delta_*^{i,l})\ll (\tau/8)\ell(I^i) $. Hence in the present scenario,
\begin{equation}\label{eq5.20}
\delta_*(X)\approx r_* \gg r(\Delta_*^{i,l})\,.
\end{equation}
We further choose $c$ small enough so that
\begin{equation}\label{eq5.21}
		2\Delta_*^{i,l} \subset \bigcup_{\substack{i'\in \N_* \\ I^{i'} \cap I^i\neq \varnothing} } S_*^{i'}\quad 
 {\rm and}
	\quad 2B^{i,l}_*\cap\Omega_* \subset 
	\bigcup_{\substack{i'\in \N_* \\ I^{i'} \cap I^i \neq \varnothing} } (I^{i'})^*.
	\end{equation} 
Note that there are at most a uniformly bounded number of such $i'$, for each $l$.  In each $\Delta_*^{i,l}$ we can use the \textbf{Base Case}, Lemma \ref{lemma:base-case}, since by construction $r(\Delta_*^{i,l})\ll (\tau/8)\ell(I^i) $ and hence \eqref{conc-base-case} implies that
\begin{equation}\label{use-bc}
\int_{\Delta_*^{i,l}} \Big( \textbf{k}_*^{X_{\Delta_*^{i,l}}} \Big)^p d\sigma_*
\lesssim
\sigma_*(\Delta_*^{i,l})^{1-p}.
\end{equation}
where $X_{\Delta_*^{i,l}}$ is a corkscrew point relative to $\Delta_*^{i,l}$ in $\Omega_*$. Using Lemma \ref{lm:cop} applied with $\mathcal{D}=\Omega_*$ and Lebesgue's differentiation theorem  we have that $\textbf{k}_*^{X}(y)\approx \omega_*^X(\Delta_*^{i,l})\,\textbf{k}_*^{X_{\Delta_*^{i,l}}}(y)$ for $\sigma_*$-a.e. $y\in \Delta_*^{i,l}$. As a result, using \eqref{use-bc}
\begin{multline}
\int_{\Delta_*^{i,l}} \big( \textbf{k}_*^X \big)^p d\sigma_*
\approx
\big(\omega^X_*(\Delta_*^{i,l})\big)^p
\int_{\Delta_*^{i,l}} \Big( \textbf{k}_*^{X_{\Delta_*^{i,l}}} \Big)^p d\sigma_*
\lesssim
\big(\omega^X_*(\Delta_*^{i,l})\big)^p\,\sigma_*(\Delta_*^{i,l})^{1-p}
\\
=\sigma_*(\Delta_*^{i,l}) \left( \frac{\omega_*^X(\Delta_*^{i,l})}{\sigma_*(\Delta_*^{i,l}) } \right)^p \lesssim
\sigma_*(\Delta_*^i) \left( \frac{\omega_*^X(\Delta_*^{i})}{\sigma_*(\Delta_*^{i}) } \right)^p \label{eq:KP},
\end{multline}
where we used the Ahlfors regularity of $\sigma_\ast$ and the doubling properties of $\omega_\ast$.
We claim that  
	\begin{equation}\label{eq:cohm}
\frac{\omega_*^X(\Delta_*^{i})}{\sigma_*(\Delta_*^{i}) } \lesssim \frac{ \omega^X(Q^i) }{\sigma(Q^i)}.
\end{equation}
To see this write $u_1(Y)=\omega^Y_*(\Delta_*^{i} )$ and $u_2(Y)=\omega^Y(Q^i)$ for every $Y\in \Omega_*$ and note that $L u_1=L u_2=0$ in $\Omega_*\subset\Omega$. For $Y\in \Delta_*^{i}\subset \overline{\Omega_*}\subset \Omega$ we have $u_2(Y)\gtrsim 1$ by Lemma \ref{lm:Bourgain} applied in $\mathcal{D}=\Omega$, Harnack's inequality, \eqref{eq:sizeSi}, and \eqref{def:WQ}. Thus the maximum principle applied in the bounded open set $\Omega_*$ yields that $u_1(Y)\lesssim u_2(Y)$ for every $Y\in\Omega_*$, hence in particular for $Y=X$. This and the fact that $\pom$ and $\pom_*$ are Ahlfors regular (see Lemma \ref{lemma:sawtooth-inherit}) give  at desired \eqref{eq:cohm}.

Combining \eqref{eq:KP} and \eqref{eq:cohm}, and using Hölder's inequality and Ahlfors regularity of $\sigma, \sigma_*$, we get
	\begin{align}
		\int_{\Delta_*^{i,l}} \left( \textbf{k}_*^X \right)^p d\sigma_* \lesssim \sigma(Q^i) \left( \frac{ \omega^X(Q^i) }{\sigma(Q^i)} \right)^p \lesssim \int_{Q^i} \left(\textbf{k}^X \right)^p d\sigma.
	\end{align}
We recall that $S_*^i$ is covered by a uniformly bounded number of surface balls $\Delta_*^{i,l}$. Thus summing in $l$ we conclude \eqref{eq:Si} as desired. This completes \textbf{Step 1}.

\noindent\textbf{Step 2.} Study the interaction of the elements of the family $\{Q^i: i\in \N_{\Delta_*}\}$. 
	
We first note that	for every $i\in \N_{\Delta_*}$
\begin{equation}\label{eq5.24}
 \dist(\Delta_*, Q) \leq \dist(S_*^i\cap \Delta_*,Q) \lesssim \ell(Q^i)\approx \ell(I^i)\lesssim r_* 
\end{equation}
Pick $\hat x\in \overline Q$ such that $\dist(\hat x, \Delta_*) = \dist(Q, \Delta_*)$. If $\hat{x}\in \overline Q \setminus Q$, we 
	replace it by a point, which we call again $\hat x$, belonging to $B(\hat x, r_*/2) \cap Q$, so that $\hat x \in Q$ and $\dist(\hat x, \Delta_*) \lesssim r_*$. We claim that 
	there is a large constant $C>1$ such that $Q^i \subset \Delta_1$ where $\Delta_1:=B(\hat x, Cr_*)\cap\pom$. Indeed if $y\in Q^i$ then
	\[
	|y-\hat x|\le \diam(Q^i)+\dist(Q^i, I^i)+\diam(I^i)+|y^i-\hat x|\lesssim r_*,
	\]
	where we have picked $y^i \in S_*^i \cap \Delta_*$ for each $i\in\mathcal{N}_{\Delta_*}$.

Consider next the covering 
$\Delta_1\subset \cup_{k=1}^{N_1} P_k$, where $N_1$ depends on Ahlfors regularity and dimension, and $\{P_k\}_{k=1}^{N_1}$ is a pairwise disjoint collection of dyadic
	cubes on $\pO$, of the same generation, with length $\ell(P_k)\approx r_*$.  Since in the present scenario,
$\ell(Q^i) \lesssim r_*$,
we may further suppose that $\ell(P_k) \geq\ell(Q^i)$ for every $i$.   Moreover, since we have assumed that  $r_* \leq M_1^{-1} \ell(Q)$, taking $M_1$ large enough we may assume that $\ell(P_k)\le \ell(Q)$ for every $1\le k\le N_1$. 

Note that 
\[ \bigcup_{i\in \N_{\Delta_*}} Q^i \subset \Delta_1 \subset \bigcup_{k=1}^{N_1} P_k.\]
By relabeling if needed, we may assume that there exists $N_2$,  $1\le N_2\le N_1$, such that $P_k$ meets some $Q^i$, $i\in \N_{\Delta_*}$, for each $1\le k\le N_2$. 
Hence $\bigcup_{i\in \N_{\Delta_*}} Q^i \subset \bigcup_{k=2}^{N_2} P_k$ and, necessarily, $Q^i\subset P_k\subset Q$, and   since $Q^i \in \DD_{\F,Q}$, it follows that $P_k\in \DD_{\F,Q}$ for $1\le k\le N_2$.

For future reference, we record the following observation.  Recall that $X$ is a Corkscrew point relative to
$\Delta_*=B_*\cap\pO_*$, for the domain $\Omega_*$; i.e.,  $X\in B_*\cap\Omega_*$, with
$\delta_*(X) \approx r_*$.  By \eqref{eq5.24} and for every $1\le k\le N_2$ if we pick some $i$ so that $Q^i\subset P_k$ we have
\[
r_*
\approx 
\delta_*(X) 
\le
\delta(X) 
\le  
\dist(X,P_k)
\le 
\dist(X, Q^i)
\le
|X-x_*|+2\,r_*+\dist(\Delta_*, Q_i)
\lesssim
r_*
\approx \ell(P_k).
\]
Recalling that $X_{P_k}$ denotes a corkscrew point relative to the dyadic cube $P_k$ we then have that $\delta(X)\approx \ell(P_k)\approx \delta(X_{P_k})$ and also $|X-X_{P_k}|\lesssim \ell(P_k)$, hence by Harnack's inequality $\omega^X\approx\omega^{X_{P_k}}$ and eventually $\textbf{k}^X\approx \textbf{k}^{X_{P_k}}$, $\sigma$-a.e. in $\pom$. On the other hand, we have already mentioned that hypothesis \ref{1-thm:trsf} in Theorem \ref{thm:trsf} says that $\omega\in RH_{p_1}(\sigma)$, which clearly implies $\omega\in RH_{p}(\sigma)$ since $p\le p_1$. Note that this reverse Hölder condition is written for surface balls, but it is straightforward to see, using Lemmas \ref{lm:ddAR} and \ref{lm:doubling}, that the same reverse Hölder estimates hold for any dyadic cube. All these,  and the fact that both $\pom$ and $\pom_*$ are Ahlfors regular (see Lemma \ref{lemma:sawtooth-inherit}) lead to
\begin{equation}\label{eq5.25}
\int_{P_k} \Big(\textbf{k}^{X_{P_k}}\Big)^{p} d\sigma
\lesssim 
\sigma(P_k)\bigg(\frac{\omega^{{X_{P_k}}}(P_k)}{\sigma(P_k)}\bigg)^p
\le 
\sigma(P_k)^{1-p} \approx \sigma_*(\Delta_*)^{1-p}\,,
 \end{equation}
for each $k$, with uniform implicit constants.

As mentioned above, for every $i\in\N_*$, there exists $J^i\in\W$ so that $I^i\cap J^i\neq\emptyset$ and so that if we pick $\widetilde{Q}^i\in\dd$ with $\ell(\widetilde{Q}^i)=\ell(J^i)$ and $\dist(J,\pom)=\dist(J^i, \widetilde{Q}^i)$ then $\widetilde{Q}^i\notin \DD_{\F,Q}$.  In particular 
	\begin{equation}\label{bdQngh}
		\ell(\widetilde{Q}^i) \approx \ell(Q^i)\quad \text{ and }\quad \dist(\widetilde{Q}^i, Q^i) \lesssim \ell(Q^i). 
	\end{equation}  
By the definition of $\DD_{\F,Q}$, $\widetilde{Q}^i \notin \DD_{\F,Q}$ means either $\widetilde{Q}^i \subset \pom\setminus Q$, or $\widetilde{Q}^i \subset Q_j $, for some
	$Q_j\in \F$. Given $1 \le k\le N_2$, for each $i\in \N_{\Delta_*}$, we say $i\in \N_0(k)$, 
if the first case happens, with $Q^i\subset P_k$; 
and if the second case happens with $Q_j\in \F$, and with $Q^i\subset P_k$,
we say $i\in \N_j(k)$.
	For the second case we remark that
	\begin{equation}\label{eq:Qjclose}
		\dist(Q_j, P_k) \leq \dist(\widetilde{Q}^i, Q^i) \lesssim \ell(Q^i) \leq \ell(P_k).
	\end{equation}
For each $k$, $1\le k\le N_2$, we set
\[
\F_1(k):=\{Q_j\in\F: \exists\, i\in\mathcal{N}_j(k),\ \ell(Q_j) \geq \ell(P_k)\}
\]
and
\[
\F_2(k):=\{Q_j\in\F: \exists\, i\in\mathcal{N}_j(k),\ \ell(Q_j) <\ell(P_k)\}.
\]

With the previous notation, \eqref{eq:tosumN}, and the fact that $\bigcup_{i\in \N_{\Delta_*}} Q^i \subset \bigcup_{k=2}^{N_2} P_k$
we obtain
\begin{align}\label{eq:tosumN:cont}
\int_{\Delta_*} \left(\textbf{k}_*^X \right)^p d\sigma_* 
&\lesssim \sum_{i\in \N_{\Delta_*}} \int_{Q^i} \left(\textbf{k}^X \right)^p d\sigma
\\
&\le
\sum_{k=1}^{N_2} \Bigg(\sum_{i\in \N_0(k)}\int_{Q^i} \left(\textbf{k}^X \right)^p d\sigma+\sum_{Q_j\in\F}\sum_{i\in \N_j(k)}\int_{Q^i} \left(\textbf{k}^X \right)^p d\sigma
\bigg)
\nonumber \\
&\lesssim
\sum_{k=1}^{N_2} \Bigg(\sum_{i\in \N_0(k)}\int_{Q^i} \left(\textbf{k}^{X_{P_k}} \right)^p d\sigma+\sum_{Q_j\in\F_1(k)}\sum_{i\in \N_j(k)}\int_{Q^i} \left(\textbf{k}^{X_{P_k}} 
\right)^p d\sigma 
\nonumber \\
&\hskip3cm
+ \sum_{Q_j\in\F_2(k)}\sum_{i\in \N_j(k)}\int_{Q^i} \left(\textbf{k}^{X_{P_k}} 
\right)^p d\sigma\bigg)
\nonumber ,
\end{align}
where we have used that $\textbf{k}^X\approx \textbf{k}^{X_{P_k}}$, $\sigma$-a.e. in $\pom$. 
	
	At this stage we need the following lemma. We defer its proof until later.
	\begin{lemma}\label{lm:thinbd} 
		Let $\mathcal{D}$ be an open set with Ahlfors regular boundary and write $\sigma=\HH^{n-1}|_{\partial\mathcal{D}}$.
		Let $Q\in \dd=\DD(\partial\mathcal{D})$ and suppose that $\DD'\subset\DD$ is such that each 
$Q'\in \DD'$ satisfies one of the following conditions for some $C_1\ge 1$:
		\begin{itemize}\itemsep=0.2cm
			\item $Q' \subset Q$ and $\dist(Q',\pom\setminus Q) \leq C_1 \ell(Q')$.
			\item $Q' \cap Q = \emptyset$, $\ell(Q') \leq C_1 \ell(Q)$ and $\dist(Q', Q) \leq C_1 \ell(Q')$.
		\end{itemize} 
Then there is a subcollection of distinct cubes 
$\{\widetilde{Q}_m\}_{m=1}^{N_2}$, all of the same generation, with $N_2 = N_2(n,C_{AR}, C_1)$, 
satisfying $\ell(Q)\le \ell(\widetilde Q_m) \le C_2 \ell(Q)$ and $\dist(\widetilde{Q}_m,Q) \le  C_2\ell(Q)$, 
with $C_2 = C_2(n,C_{AR}, C_1)$, for every $m$, such that for any $s>1$ if $0\le h\in L_{\rm loc}^s(\partial\mathcal{D},\sigma)$ then 
		\begin{equation}\label{ccl:thinbd}
			\sum_{Q' \in \DD'} \int_{Q'} h d\sigma \le C_3 \sigma(Q) 
\sum_{m=1}^{N_2} \left( \fint_{ \widetilde Q_m} h^s d\sigma \right)^{\frac{1}{s}}
		\end{equation}
		where $C_3=C_3(n,C_{AR}, C_1,s)$.

As a consequence, if there exists $C_1'$ so that for each $m$, $1\le m\le N_2$, there holds
\begin{equation}\label{RH-h}
\left( \fint_{ \widetilde Q_m} h^s d\sigma \right)^{\frac{1}{s}}\le C_1' \fint_{ \widetilde Q_m} h\, d\sigma
\end{equation}
then
\begin{equation}\label{ccl:thinbd:Ch}
\sum_{Q' \in \DD'} \int_{Q'} h d\sigma \le C_3' 
\sum_{m=1}^{N_2} \int_{\widetilde Q_m } h d\sigma 
\end{equation}
with $C_3'=C_3'(n,C_{AR}, C_1,s, C_1')$.
\end{lemma}

	\begin{remark} 
It follows from the proof of that if $Q'\subset Q$ for all $Q'\in \DD'$ (i.e., we only consider the first case), then 
there is only one $\widetilde{Q}_m$, namely the unique one containing $Q$ satisfying the given conditions. 	
	\end{remark}

\begin{remark}\label{rem:daydic-lemma:omega}
Suppose that we are under the assumptions of the previous result. Assume further that $\mathcal{D}$ is a uniform domain with Ahlfors regular boundary and that $\omega_L\in RH_p(\sigma)$. Then, if $\textbf{k}_L=d\omega_L/d\sigma$ it follows that 
\begin{equation}
\label{ccl:thinbd2}
\sum_{Q' \in \DD'} \int_{Q'} \left(\textbf{k}_L^{X_Q}\right)^p d\sigma \lesssim \int_{Q} \left( \textbf{k}^{X_Q} \right)^p d\sigma.
\end{equation}
with an implicit constant depending on the allowable constants of $\mathcal{D}$, $C_1$, $p$, and the implicit constant in the condition  $\omega_L\in RH_p(\sigma)$.

To see this we recall that from Gehring's Lemma it follows that there exists $s>1$ such that  $\omega_L\in RH_{ps}(\sigma)$. This, combined with Harnack's inequality, implies that \eqref{RH-h} holds with $h=\big(\textbf{k}_L^{X_Q}\big)^p$. As a result \eqref{ccl:thinbd:Ch} readily gives \eqref{ccl:thinbd2}: 
\begin{multline*}
\sum_{Q' \in \DD'} \int_{Q'}\left(\textbf{k}_L^{X_Q}\right)^pd\sigma \lesssim 
\sum_{m=1}^{N_2} \int_{\widetilde Q_m } \left(\textbf{k}_L^{X_Q}\right)^p d\sigma 
\approx
\sum_{m=1}^{N_2} \int_{\widetilde Q_m } \left(\textbf{k}_L^{X_{\widetilde Q_m}}\right)^p d\sigma 
\\
\lesssim
\sum_{m=1}^{N_2} \sigma(\widetilde Q_m)^{1-p}
\lesssim
\sigma(Q)^{1-p},
\end{multline*}		
where we have used Harnack's inequality (to change the pole of the elliptic measure from $X_Q$ to $X_{\widetilde Q_m}$ and the fact that $N_2$ is uniformly bounded). 
\end{remark}
	
We will use the previous remark to estimate \eqref{eq:tosumN:cont}. Fixed then $1\le k\le N_2$ and we split the proof in three different steps.

\noindent\textbf{Step 2.1.} Estimate for $\N_0(k)$. 

If $i\in \N_0(k)$ we have $\widetilde{Q}^i \subset \pom\setminus Q \subset \pom\setminus P_k$ and 
	\[ \dist(Q^i, \pom\setminus P_k) \leq \dist(Q^i, \widetilde{Q}^i) \lesssim \ell(Q^i). \]
	Since $Q^i \subset P_k$,  we may apply Lemma \ref{lm:thinbd} to $P_k$ 
	and the collection $\DD' := \{Q^i: i\in \N_0(k)\}$ (note that we are in the first scenario), to obtain by Remark \ref{rem:daydic-lemma:omega}
\begin{equation}\label{eq:N0}
\sum_{i\in \N_0(k)} 
\int_{Q^i} \left(\textbf{k}^{X_{P_k}}\right)^p d\sigma
\lesssim 
\int_{P_k} \left(\textbf{k}^{X_{P_k}} \right)^p d\sigma  
\lesssim \sigma_*(\Delta_*)^{1-p},
	\end{equation}
 where in the last inequality we have used \eqref{eq5.25}.

\noindent\textbf{Step 2.2.} Estimate for $Q_j\in \F_1(k)$.

	By \eqref{eq:Qjclose}, the cardinality of $\F_1(k)$ is uniformly bounded. Moreover, for each $Q_j\in \F_1(k)$ we necessarily have $Q_j \cap P_k = \emptyset$, since otherwise, the 
condition $\ell(Q_j) \geq \ell(P_k)$ 
guarantees that $P_k\subset Q_j$, and thus $Q^i \subset P_k \subset Q_j\in\F$. This 
contradicts that $Q^i \in \DD_{\F,Q}$. On the other hand $Q_j \cap P_k = \emptyset$ implies 
$\widetilde{Q}^i\subset Q_j\subset \pom\setminus P_k $, for each $i\in \N_j(k)$. 
Combined with \eqref{bdQngh}, this yields
	\[ \dist(Q^i, \pom\setminus P_k) \leq \dist(Q^i, \widetilde{Q}^i) \lesssim \ell(Q^i). \]

	Applying Lemma \ref{lm:thinbd} to $P_k$
	and the collection $\DD' = \{Q^i: i\in \N_j(k)\} $ (note that we are in the first scenario), we obtain from \eqref{ccl:thinbd2}
\begin{equation}\label{eq:Nj1}
		\sum_{i\in \N_j(k)} \int_{Q^i} \left(\textbf{k}^{X_{P_k}} \right)^p d\sigma \lesssim 
\int_{P_k} \left(\textbf{k}^{X_{P_k}} \right)^p d\sigma 
\lesssim \sigma_*(\Delta_*)^{1-p}, 
	\end{equation}
 where again we have used \eqref{eq5.25}.  
	The above estimate holds for each $Q_j\in \F_1(k)$, which as uniformly bounded cardinality, hence
	\begin{equation}\label{eq:Nj1all}
		\sum_{Q_j\in \F_1(k)} \sum_{i\in \N_j(k)} \int_{Q^i} \left(\textbf{k}^{X_{P_k}}  \right)^p d\sigma  
		\lesssim \sigma_*(\Delta_*)^{1-p}. 
	\end{equation}

\noindent\textbf{Step 2.3.} Estimate for $Q_j\in \F_2(k)$. 

For each $Q_j\in \F_2(k)$ we claim that
	\begin{equation}\label{eq:Nj2}
		\sum_{i\in \N_j(k)} 
		\int_{Q^i} \left(\textbf{k}^{X_{P_k}} \right)^p d\sigma \lesssim \int_{Q_j} \left( \textbf{k}^{X_{P_k}}  \right)^p d\sigma.
	\end{equation}
	In fact, for each $i\in \N_j(k)$, by \eqref{bdQngh} and $\widetilde{Q}^i \subset Q_j$, we have
	\begin{equation}\label{sizeupbd}
		\ell(Q^i) \approx \ell(\widetilde{Q}^i) \leq \ell(Q_j).
	\end{equation}
	Since $Q^i \in \DD_{\F,Q}$, we either have $Q^i \cap Q_j = \emptyset$, or $Q_j \subsetneq Q^i$. 
In the first case, note that
\[
\dist(Q^i,Q_j) \leq \dist(Q^i,\widetilde{Q}^i))\lesssim\ell(Q^i),
\]
hence $Q^i\cup Q_j\subset \Delta(x_{Q_j}, C\,\ell(Q_j))$ which $x_{Q_j}$ being the center of $Q_j$ an a uniform constant $C$. By Lemma \ref{lm:cop} applied with $\mathcal{D}=\Omega$ (or Harnack's inequality if $\ell(Q_j)\approx \ell(P_k)$), Lebesgue's differentiation theorem, Lemma \ref{lm:doubling}, and Harnack's inequality one can see that
\[
\textbf{k}^{X_{P_k}}(y)\approx \omega^{X_{P_k}}(Q_j)\, \textbf{k}^{X_{Q_j}}(y),
\qquad
\text{for $\sigma$-a.e. $y\in  \Delta(x_{Q_j}, C\,\ell(Q_j))$}.
\]
This, Lemma \ref{lm:thinbd} with $Q_j$ and 
the collection $\DD': = \{Q^i: i\in \N_j(k), Q^i \cap Q_j = \emptyset \}$ (we are in the second scenario), and Remark \ref{rem:daydic-lemma:omega} lead to
\begin{multline}\label{eq:Nj2out}
	\sum_{\substack{i\in \N_j(k) \\ Q^i \cap Q_j = \varnothing} } \int_{Q^i} \left(\textbf{k}^{X_{P_k}} \right)^p d\sigma 
	\approx
		\left(\omega^{X_{P_k}}(Q_j)\right)^p \sum_{\substack{i\in \N_j(k) \\ Q^i \cap Q_j = \varnothing} } \int_{Q^i} \left(\textbf{k}^{X_{Q_j}} \right)^p d\sigma 
		\\
		\lesssim \left(\omega^{X_{P_k}}(Q_j) \right)^p\int_{Q_j} \left( \textbf{k}^{X_{Q_j}} \right)^p d\sigma
		\approx
		\int_{Q_j} \left( \textbf{k}^{X_{P_k}} \right)^p d\sigma
	.
\end{multline}

On the other hand, if $Q_j \subsetneq Q^i$, then \eqref{sizeupbd} gives $\ell(Q_j)\approx \ell(Q^i)$, hence the cardinality of $\{Q^i: i\in \N_j(k), Q_j \subsetneq Q^i\}$ is uniformly bounded. On the other hand, by Lemma \ref{lm:cop} applied with $\mathcal{D}=\Omega$ (or Harnack's inequality if $\ell(Q^i)\approx \ell(P_k)$), Lebesgue's differentiation theorem, Lemma \ref{lm:doubling}, and Harnack's inequality we readily obtain 
\[
\textbf{k}^{X_{P_k}}(y)\approx \omega^{X_{P_k}}(Q_j)\, \textbf{k}^{X_{Q_j}}(y),
\qquad
\text{for $\sigma$-a.e. $y\in  Q^i$}.
\]
Thus, using  Corollary \ref{cor:doublingPk} we have
\begin{align}\label{eq:Nj2in}
\sum_{\substack{i\in \N_j(k) \\ Q^i \supsetneq Q_j}} \int_{Q^i} \left(\textbf{k}^{X_{P_k}} \right)^p d\sigma 
& \approx
\left ( \omega^{X_{P_k}}(Q_j)\right)^p\,\sum_{\substack{i\in \N_j(k) \\ Q^i \supsetneq Q_j}}  \int_{Q^i} \left(\textbf{k}^{X_{Q_j}} \right)^p d\sigma 
\\ \nonumber
 &\lesssim
\left ( \omega^{X_{P_k}}(Q_j)\right)^p\, \sum_{\substack{i\in \N_j(k) \\ Q^i \supsetneq Q_j}} \int_{Q_j} \left(\textbf{k}^{X_{Q_j}} \right)^p d\sigma 
\\ \nonumber
 &\lesssim 
\left ( \omega^{X_{P_k}}(Q_j)\right)^p\, \int_{Q_j} \left(\textbf{k}^{X_{Q_j}} \right)^p d\sigma 
\\ \nonumber
 &\approx
 \int_{Q_j} \left(\textbf{k}^{X_{P_k}} \right)^p d\sigma.
\end{align}
The claim \eqref{eq:Nj2} now follows from \eqref{eq:Nj2out} and \eqref{eq:Nj2in}. 
	
To continue, let us recall that for each $Q_j\in \F_2(k)$, 
	\[ \ell(Q_j) < \ell(P_k)\quad \text{ and }\quad  \dist(Q_j,P_k) \lesssim \ell(P_k), \]
	where the second inequality is \eqref{eq:Qjclose}.
Consequently, each $Q_j\in \F_2(k)$, is contained in some $P\in
	{\bf N}(P_k):= \{P\in\dd: \ell(P) =\ell(P_k),\  \dist(P,P_k) \lesssim \ell(P_k)\}$ and, clearly, the cardinality of ${\bf N}(P_k)$ is uniformly bounded.
Recalling that $\F=\{Q_j\}_j$ is a pairwise disjoint family of cubes, 
by \eqref{eq:Nj2}, Corollary \ref{cor:doublingPk}, and \eqref{eq5.25}, we arrive at
\begin{multline}
		\sum_{Q_j\in \F_2(k)} \sum_{i\in \N_j(k)} 
		\int_{Q^i} \left(\textbf{k}^{X_{P_k}} \right)^p d\sigma \lesssim 
		\sum_{Q_j\in \F_2(k)} \int_{Q_j} \left(\textbf{k}^{X_{P_k}} \right)^p d\sigma \\[4pt] = \,
		\int_{\bigcup\limits_{Q_j\in \F_2(k)} Q_j} \left(\textbf{k}^{X_{P_k}} \right)^p d\sigma 
		 \le \sum_{P\in {\bf N}(P_k)}\int_{P} \left(\textbf{k}^{X_{P_k}} \right)^p d\sigma  
		\lesssim \int_{P_k} \left(\textbf{k}^{X_{P_k}} \right)^p d\sigma \label{eq:Nj2all}
\lesssim\sigma_*(\Delta_*)^{1-p}.
	\end{multline}

\noindent\textbf{Step 2.4.} Final estimate.

We finally combine \eqref{eq:tosumN:cont} with \eqref{eq:N0}, \eqref{eq:Nj1all}, and \eqref{eq:Nj2all}, and use the fact that $N_2\le N_1=N_1(n, C_{AR})$  
to conclude that 
	\begin{align}
\int_{\Delta_*} \left(\textbf{k}_*^X \right)^p d\sigma_*
\lesssim
\sum_{k=1}^{N_2} \sigma_*(\Delta_*)^{1-p}
\lesssim \sigma_*(\Delta_*)^{1-p}.
\label{eq:sumallN}
	\end{align} 
Hence, we have obtained the desired estimate \eqref{RH},
and therefore the proof of Theorem \ref{thm:trsf} is complete, 
provided that the sawtooth domain $\Omega_\ast$ is \textbf{compactly contained} in $\Omega$ and modulo the proof of Lemma \ref{lm:thinbd}.

To consider the general case we need the following theorem which generalizes \cite[Theorem 4.1]{KKiPT} and \cite{DJK} (see also \cite{DKP, Zh}):
\begin{theorem}[{\cite[Theorem 1.1]{CHMT}}]\label{thm:CMEAinfty}
	Let $\mathcal{D}$ be uniform domain $\mathcal{D}$ Ahlfors regular boundary, and let $\wcalA$ be a real (non necessarily symmetric) uniformly elliptic matrix on  $\mathcal{D}$. The following are equivalent:
	\begin{enumerate}[label=\textup{(\arabic*)}, itemsep=0.2cm] 	
		\item\label{1-thm:CMEAinfty}  The elliptic measure $\omega_L$ associated with the operator $L=-\divg(\wcalA\nabla)$ is of class $A_\infty$ with respect to the surface measure. 
		
		\item\label{2-thm:CMEAinfty} Any bounded weak solution to $Lu=0$ satisfies the Carleson measure estimate 
		\begin{equation}
		\sup_{\substack{x\in\partial\mathcal{D} \\ 0<r<\infty}} \frac1{r^n}\iint_{B(x,r) \cap \mathcal{D}} |\nabla u(Y)|^2 \dist(Y,\partial\mathcal{D})\, dY \leq C\|u\|_{L^\infty(\mathcal{D})}^2.
		\end{equation}
	\end{enumerate}
The involved constants depend on the allowable constants and the constant appearing in the corresponding hypothesis of the implication in question. 
\end{theorem} 

Consider next a \textbf{general} sawtooth domain $\Omega_* = \Omega_{\F,Q}$ which, although bounded, is not necessarily compactly contained in $\Omega$. By \ref{2-thm:CMEAinfty} $\implies$ \ref{1-thm:CMEAinfty} in Theorem \ref{thm:CMEAinfty} with $\mathcal{D}=\Omega$, in order to obtain that the elliptic measure associated with $L$ relative to $\Omega_*$ belongs to $A_\infty$ with respect to the surface measure, we just need to see that \ref{2-thm:CMEAinfty} holds with $\mathcal{D}=\Omega_*$. With this goal in mind we take $u$, a bounded weak solution to $Lu=0$ in $\Omega_*$, and let $x\in\pom$ and $0<r<\infty$.

Given $N\ge 1$ we recall the definition of $\F_N:=\F(2^{-N}\,\ell(Q))$ in Section \ref{section:sawtooth}  and write $\Omega_*^N := \Omega_{\F_N,Q}$. Note that by construction $\ell(Q')> 2^{-N}\,\ell(Q)$ for every $Q'\in\dd_{\F_N, Q}$, hence $\Omega_*^N$ is compactly contained in $\Omega$ (indeed is at distance of the order $2^{-N}\,\ell(Q)$ to $\pom$). Then we can apply the previous case to obtain that for each $N$, the associated elliptic measure associated with $L$ relative to $\Omega_*^N$ satisfies the $A_\infty$ property with respect to the surface measure of $\partial\Omega_\ast^N$, and the implicit constants depend only on the allowable constants. Hence 
\ref{1-thm:CMEAinfty} $\implies$ \ref{2-thm:CMEAinfty} in Theorem \ref{thm:CMEAinfty} with $\mathcal{D}=\Omega_*^N$ implies
\begin{equation}\label{CME-N}
\sup_{\substack{z\in\pom_*^N \\ 0<s<\infty}} \frac1{s^n}\iint_{B(z,s) \cap \Omega_*^N} |\nabla u(Y)|^2 \delta_*^N(Y)\, dY \lesssim \|u\|_{L^\infty(\Omega_N^*)}^2
\le
C\|u\|_{L^\infty(\Omega_*)}^2,
\end{equation}
since $u$ is a bounded weak solution to $Lu=0$ in $\Omega_*$ and so in each $\Omega_N^*$,  where $\delta_*^N=\dist(\cdot\,,\pom_*^N)$ and where the implicit constants depend only on the allowable constants.

Let $\omega_*$ and $\omega_*^N$ denote the elliptic measures to $L$ relative to $\Omega_*$ and $\Omega_*^N$ respectively, and let $\sigma_*^N := \mathcal{H}^{n-1}|_{\pO_*^N}$ denote the surface measure of $\partial\Omega_\ast^N$. By construction $\{\Omega_*^N\}_{N\ge 1}$ is an increasing sequence of sets with $\Omega_\ast=\cup_{N\ge 1}\Omega_*^N$. Hence, for any $Y\in \Omega_*$ there is $N_Y\ge 1$ such that $Y\in\Omega_\ast^N$ for all $N\ge N_Y$.  Clearly $\delta_*^N(Y) \nearrow \delta_*(Y)$ as $N\to\infty$ and
\[ 
|\nabla u(Y)|^2 \delta_*^N(Y) \chi_{\Omega_*^N}(Y) \nearrow |\nabla u(Y)|^2 \delta_*(Y) \chi_{\Omega_*}(Y), \quad \text{ as } N\to \infty.
\]
On the other hand since $x\in \pO_*$, using the Corkscrew condition we can find a sequence $\{x_N \}_{N\ge 1}$ with $x_N\in \pO_*^N$ such that $x_N \to x$. In particular, $B(x,r) \subset B(x_N,2r)$ for sufficiently large $N$. By Fatou's Lemma and \eqref{CME-N} it then follows
	\begin{multline}\label{tt-2}
		\iint_{B(x,r) \cap \Omega_*} |\nabla u(Y)|^2 \delta_*(Y)\, dY 
		\leq \liminf_{N\to\infty} \iint_{B(x,r) \cap \Omega_*^N} |\nabla u(Y)|^2 \delta_*^N(Y)\, dY 
		\\
			\leq \liminf_{N\to\infty} \iint_{B(x_N,2r) \cap \Omega_*^N} |\nabla u|^2 \delta_*^N(y)\, dY
			\lesssim
			r^n\,
			\|u\|_{L^\infty(\Omega_*)}^2,
		\end{multline}
		where the implicit constant depend only on the allowable constants. Since $x$, $r$, and $u$ are arbitrary we have obtained as desired \ref{2-thm:CMEAinfty} in Theorem \ref{thm:CMEAinfty} for $\mathcal{D}=\Omega_*$ and as a result we conclude that $\omega_*\in A_\infty(\sigma_*)$. This completes the proof for an arbitrary sawtooth domain $\Omega_*$, and therefore the proof of Theorem \ref{thm:trsf} modulo the proof of Lemma \ref{lm:thinbd}.
	\end{proof}

	\begin{proof}[Proof of Lemma \ref{lm:thinbd}]
For fixed $k\in \ZZ$, write $\DD'_k:=\{Q'\in\dd': \ell(Q') = 2^{-k} \ell(Q)\}$, which is a pairwise disjoint family. In the first case since $Q'\subset Q$, we have that  $k\ge 0$; in the second case since $\ell(Q') \leq C_2 \ell(Q)$, we may assume that $k\ge -\log_2 C_2$. Set then $k_0=0$ in the first case and $k_0$ the integer part of $\log_2 C_1$. We define for $k\ge -k_0$
		\[ A_k^+ = \{x\in Q: \dist(x,\pom\setminus Q) \lesssim 2^{-k} \ell(Q) \}, \ \ \quad A_k^- = \{x\in\partial\Omega\setminus Q: \dist(x,Q) \lesssim 2^{-k} \ell(Q) \}, \]
		so that for
		appropriate choices of the implicit constants, 
each $Q'\in \DD'_k$ is contained in either $A_k^+$ (the first case) or $A_k^-$ (the second case).
		Recall that by the thin boundary property of the 
dyadic decomposition $\DD$ (cf. \eqref{thin-boundary}), there is $\gamma \in (0,1)$ such that for all $k$ under consideration,
		\[ 
			\sigma(A_k^+) \lesssim 2^{-k\gamma} \sigma(Q), \quad \sigma(A_k^-) \lesssim 2^{-k\gamma}\sigma(Q). 
		\]
Set 
$$\F_-:= \big\{Q'\subset \pom\setminus Q: \ell(Q') \leq C_1 \ell(Q),\ \dist(Q',Q) \leq C_1 \ell(Q')\big\}.$$
Observe that each $Q'\in \F_-$ is contained in some dyadic cube $\widetilde{Q}$, with
 $\ell(\widetilde Q)\approx \ell(Q)$ and $\dist(\widetilde{Q},Q) \lesssim \ell(Q)$  depending on $C_1$.
We may therefore define a collection of distinct cubes 
$\F_*:= \{\widetilde{Q}_m\}_{m=1}^{N_2}$, all of the same dyadic generation, 
one of which (say, $\widetilde{Q}_1$) contains $Q$, with $\ell(\widetilde Q_m) \approx \ell(Q)$, and with
$\dist(\widetilde{Q}_m,Q) \lesssim \ell(Q)$  for every $m$,
such that each $Q'\in\F_-$ is contained in some $\widetilde{Q}_m\in\F_*$, and 
$$\bigcup_k A_k^+ \,\subset\, Q\,\subset \,\widetilde{Q}_1\,, \quad {\rm and} \quad 
\bigcup_k A_k^- \subset \, \bigcup_{m=2}^N \widetilde{Q}_m.$$
Clearly, we have $\#\F_* = N_2= N_2(n,C_{AR}, C_1)$. Using all the previous observations we get for any $s>1$
\begin{align}\label{tt-100}
		\sum_{Q' \in \DD'} \int_{Q'} h\, d\sigma 
		&=
			 \sum_{k=-k_0}^\infty \sum_{Q' \in \DD'_k} \int_{Q'} h\, d\sigma 
\\
&			\leq  \sum_{k=-k_0}^\infty \int_{A_k^+ \cup A_k^-} h\, d\sigma
			\nonumber\\ 
			& \leq \sum_{k=-k_0}^\infty \sigma(A_k^+ \cup A_k^-)^{1-\frac1s} 
			\left( \int_{\cup_m\widetilde Q_m} h^s d\sigma \right)^{\frac{1}{s}}
			\nonumber\\
			& \lesssim \sum_{k=-k_0}^\infty \left(2^{-k\gamma }\sigma(Q) \right)^{1-\frac{1}{s}} 
	 \left( \int_{\cup_m\widetilde {Q}_m} h^s\, d\sigma \right)^{\frac{1}{s}} 
	\nonumber \\
			& \lesssim \sigma(Q) \left(\sum_m \fint_{\widetilde Q_m} h^s\, d\sigma \right)^{\frac{1}{s}}
	\nonumber\\
		&	\lesssim \sigma(Q) \sum_m 
		\left( \fint_{\widetilde Q_m} h^s d\sigma \right)^{\frac{1}{s}}.\nonumber
		\end{align}
This shows \eqref{ccl:thinbd}. To obtain \eqref{ccl:thinbd:Ch} we combine \eqref{tt-100} together with \eqref{RH-h} and the fact that $\sigma(Q)\approx\sigma(Q_m)$ for every $1\le m\le N_2$ by the Ahlfors regular property and the construction of the family $\F_*$. 
\end{proof}

\section{Optimality}\label{optimal}

As we mentioned in the introduction, the class of elliptic operators we consider is optimal to guarantee the $A_\infty$ property. In this section we illustrate the optimality from two different points of view. See Proposition \ref{p:bruno} and Theorem \ref{thm:oscop}.

As mentioned right after Definition \ref{def:AinftyHMU}, one has that $\omega_L\in A_\infty(\sigma)$ if and only if 
$\omega_L\in RH_q(\sigma)$ for some $q>1$ in the following sense: $\omega_L \ll \sigma$ and the
Radon-Nikodym derivative $\textbf{k}_L:= d\omega_L/d\sigma$ satisfies the reverse H\"older estimate \eqref{eq1.wRH}. We can then define the $RH_q(\sigma)$-characteristic of $\omega_L$ as folows 
\begin{equation}\label{eq1.wRH-char}
[\omega_L]_{RH_q}:=\sup\left(\fint_{\Delta'} \big(\textbf{k}_L^{A(q,r)}\big)^q d\sigma \right)^{\frac1q} \left(\fint_{\Delta'} \textbf{k}_L^{A(q,r)} \,d\sigma\right)^{-1},
\end{equation}
where the sup runs over all $q\in\partial\Omega$, $0<r<\diam (\Omega)$, and all surface balls $\Delta'=B'\cap \pO$ centered at $\pO$ with $B'\subset B(q,r)$. 

The following example, based on the work in \cite{MM} and communicated to us by Bruno Guiseppe  Poggi Cevallos,  illustrates the relationship between the size of the constant in the DKP condition and the $RH_q(\sigma)$-characteristic of elliptic measure. 

\begin{prop}[\cite{MM, P}]\label{p:bruno}
There exist ${\mathcal A}$ and a sequence $\{{\mathcal A}_j\}_j$ of diagonal elliptic matrices with smooth, bounded, real coefficients in $\RR^{n}_+$, uniformly continuous on $\overline{\RR^{n}_+}$, such that ${\mathcal A}_j$ converges to ${\mathcal A}$ uniformly on $\overline{\RR^{n}_+}$ and the following hold: 
\begin{enumerate}[label=\textup{(\arabic*)}, itemsep=0.2cm]
\item 
$\displaystyle{
\sup_{\substack{q\in\RR^{n-1} \\ 0< r<\infty} } \frac{1}{r^{n-1}} \iint_{B(q,r) \cap\RR^{n}_+} 
\bigg(
\sup_{Y\in B(X,\frac{\delta(X)}{2})} |\nabla \mathcal{A}_j(Y)|^2 \delta(Y)\bigg)dX \gtrsim j.
}$

\item For each $q>1$, one has $\omega_j\in RH_q(\sigma)$ with $\displaystyle{\lim_{j\to\infty}[\omega_j]_{RH_q}=\infty}$, where $\omega_j$ denotes the elliptic measure associated with the operator $L_j=-\divg({\mathcal A}_j(\cdot)\nabla)$.

\item The elliptic measure associated with the operator $L=-\divg({\mathcal A}(\cdot)\nabla)$ is singular with respect to the Lebesgue measure on $\partial \RR^{n}_+=\RR^{n-1}$.
\end{enumerate}
\end{prop}

On the other hand, we can immediately extend Theorem \ref{thm:main} to a larger and optimal class of elliptic operators, pertaining to the condition on the oscillation of the coefficient matrix:

\begin{corollary}\label{cOsc}
Let $\Omega\subset \RR^n$, $n\ge 3$, be a uniform domain with Ahlfors regular boundary. Let $\wcalA$ be a (not necessarily symmetric) uniformly elliptic matrix on $\Omega$  such that 
	\begin{equation}\label{eq:oscCM}
		\sup_{\substack{q\in\pO\\0< r<\diam(\Omega)} } \frac{1}{r^{n-1}} \iint_{B(q,r) \cap\Omega} \frac{\osc(\wcalA, X)^2}{\delta(X)} \,dX<\infty.
	\end{equation}  where
	$\osc(\wcalA, X) := \sup_{Y,Z\in B(X, \delta(X)/2)}   |\wcalA(Y) -\wcalA(Z)|$. Then the following are equivalent:
\begin{enumerate}[label=\textup{(\alph*)}, itemsep=0.2cm] 
		\item\label{1-corol-osc} The elliptic measure $\omega_L$ associated with the operator $L=-\divg(\wcalA(\cdot)\nabla)$ is of class $A_\infty$ with respect to the surface measure. 
		\item\label{3-corol-osc} $\pO$ is uniformly rectifiable.
		\item\label{2-corol-osc} $\Omega$ is a chord-arc domain.
		\end{enumerate} 
\end{corollary}

\begin{proof}
Let $\varphi$ be a non-negative radial, smooth bump function supported in the unit ball, such that $\int_{\RR^n} \varphi \,dX = 1$. We define for $X\in\Omega$ and $t\in (0, \delta(X))$
	\[ P_t \, \wcalA(X) := \varphi_t * \wcalA(X) = \iint_{\RR^n} \frac{1}{t^n} \,\varphi\left( \frac{X-Y}{t} \right) \wcalA(Y) \,dY,  \]
	and write for $X\in\Omega$
	\begin{equation}\label{eq:oscdcp}
		\wcalA(X) = P_{\frac{\delta(X)}{4}} \, \wcalA(X) + \left( \Id - P_{\frac{\delta(X)}{4}} \right) \wcalA(X) =: \widetilde{\wcalA}(X) + \left( \wcalA(X) - \widetilde{\wcalA}(X) \right).
	\end{equation} 
It is easy to see that $\widetilde{\wcalA}$ is uniformly elliptic with the same constants as $\wcalA$ and also that for every $X\in\Omega$
	\begin{equation}\label{eq:osc}
		|\nabla \widetilde{\wcalA}(X)| \leq C\, \frac{\osc(\wcalA, X)}{\delta(X)} \qquad\mbox{and}\qquad \sup_{Y\in B(X, \delta(X)/4)} |\wcalA(Y) - \widetilde{\wcalA}(Y)| \leq \osc(\wcalA, X).
	\end{equation}
	Note that under assumption \eqref{eq:oscCM}, the second estimate in \eqref{eq:osc} allows us to invoke \cite[Theorem 1.3]{CHMT} to obtain that $\omega_{L_\wcalA} \in A_\infty(\sigma)$  if and only if $\omega_{L_{\widetilde{\wcalA}}} \in A_\infty(\sigma)$. On the other hand, the first estimate in \eqref{eq:osc} readily implies that $\widetilde{\wcalA}$ satisfies  \ref{H1} and \ref{H2}. Hence,  Theorem \ref{thm:main} applied to $\widetilde{\wcalA}$ gives at once the desired equivalences. 
	We remark that the direction $(c)\implies (a)$ was also proved earlier in \cite[Theorem 2.4]{Rios}.
\end{proof}

Corollary \ref{cOsc} is sharp in terms of the class of operators satisfying \eqref{eq:oscCM}. We recall the following examples that illustrate this fact.

\begin{theorem}{\cite[Theorem 4.11]{FKP}}\label{thm:oscop}
	Suppose $\alpha$ is a non-negative function defined on $\RR^2_+$ satisfying the doubling condition: $\alpha(X) \leq C\alpha(X_0)$ for any $X_0\in \RR^2_+$ and $X\in B(X_0, \delta(X_0)/2)$. Assume that $\alpha^2(x,t) dx \, dt/t$ is not a Carleson measure in the unit square. Then there exists an elliptic operator $L= - \divg(\mathcal{A}(\cdot)\nabla)$ on $\RR^2_+$, such that
	\begin{enumerate}[label=\textup{(\arabic*)}, itemsep=0.2cm]
		\item For any interval $I \subset \RR$ and $T(I)=I\times [0,\ell(I)]$,
		\begin{equation}\label{eq:oscbd}
			\frac{1}{|I|} \iint_{T(I)} \frac{\osc^2(\mathcal{A}(x,t))}{t} dx \, dt \leq C\left[ \frac{1}{|I|} \iint_{T(2I)} \frac{\alpha^2(x,t)}{t} dx \, dt + 1 \right]; 
		\end{equation} 
		\item The elliptic measure $\omega_{L}$ is not of class $A_\infty(dx)$ on the unit interval $[0,1]$.
	\end{enumerate}
\end{theorem}
\begin{remarks}
	The examples above are constructed using quasi-conformal maps in $\RR^2$. In \cite{FKP} the authors show the estimate \eqref{eq:oscbd} holds when $\osc (\mathcal{A,}(x,t))$ is replaced by the oscillation of elliptic matrix $\mathcal{A}(X)$ minus the identity matrix, i.e., $a(X): = \sup_{Y\in B(X, \delta(X)/2 )} |\mathcal{A}(Y) - \Id|$. It is easy to see that for those examples \eqref{eq:oscbd} follows. As in \cite[Theorem 3]{CFK}, one can extend the 2 dimensional examples to $\RR^n_+$ with $n\geq 3$ by using the Laplacian operator in the remaining tangential directions.

\end{remarks}

\addtocontents{toc}{\protect\vspace{10pt}}

\end{document}